\documentclass[11pt]{article}
\usepackage{amsthm,amsfonts,amssymb,amsmath,oldgerm}
\usepackage{epsfig}
\numberwithin{equation}{section}
\newcommand{\ip}{\int_{\mathbb{R}}}
\newcommand{\btheta}{{\bar \theta}}

\newcommand\R{\mathbb R}

\def\eps{\varepsilon}

\def\bV{{\hat{v}}}

\newcommand\br{\begin{remark}}
\newcommand\er{\end{remark}}
\newcommand\bp{\begin{pmatrix}}
\newcommand\ep{\end{pmatrix}}
\newcommand\be{\begin{equation}}
\newcommand\ee{\end{equation}}
\newcommand\ba{\begin{equation}\begin{aligned}}
\newcommand\ea{\end{aligned}\end{equation}}

\newcommand{\bap}{\begin{app}}
\newcommand{\eap}{\end{app}}
\newcommand{\begs}{\begin{exams}}
\newcommand{\eegs}{\end{exams}}
\newcommand{\beg}{\begin{example}}
\newcommand{\eeg}{\end{exaplem}}
\newcommand{\bpr}{\begin{proposition}}
\newcommand{\epr}{\end{proposition}}
\newcommand{\bt}{\begin{theorem}}
\newcommand{\et}{\end{theorem}}
\newcommand{\bc}{\begin{corollary}}
\newcommand{\ec}{\end{corollary}}
\newcommand{\bl}{\begin{lemma}}
\newcommand{\el}{\end{lemma}}
\newcommand{\bd}{\begin{definition}}
\newcommand{\ed}{\end{definition}}
\newcommand{\brs}{\begin{remarks}}
\newcommand{\ers}{\end{remarks}}

\newcommand{\rank}{{\rm rank }}

\newcommand{\CalT}{\mathcal{T}}

\newcommand{\RR}{{\mathbb R}}

\newcommand{\CC}{{\mathbb C}}

\newcommand{\const}{\text{\rm constant}}

\newcommand{\Span}{{\rm Span }}

\newcommand{\sgn}{\text{\rm sgn}}
\newtheorem{theorem}{Theorem}[section]
\newtheorem{proposition}[theorem]{Proposition}
\newtheorem{corollary}[theorem]{Corollary}
\newtheorem{lemma}[theorem]{Lemma}
\newtheorem{definition}[theorem]{Definition}

\newtheorem{example}[theorem]{Example}
\newtheorem{remark}[theorem]{Remark}

\newtheorem{exams}[theorem]{Examples}

\newcommand\cV{{\cal  V}}
\newcommand\cW{{\cal  W}}

\newcommand\cR{{\cal  R}}

\newcommand\cL{{\cal  L}}

\newcommand\cP{{\cal  P}}

\newcommand\cM{{\mathcal M}}

\title{
One-dimensional stability of parallel shock\\ 
layers in isentropic magnetohydrodynamics
}

\author{\sc \small
Blake Barker\thanks{ Brigham Young University, Provo, UT 84602;
bhbarker@gmail.com: Research of B.B. was partially supported 
under NSF grants number DMS-0607721 and DMS-0300487.},
Jeffrey Humpherys \thanks{ Brigham Young University, Provo, UT 84602;
jeffh@math.byu.edu: Research of J.H. was partially supported 
under NSF grant DMS-0607721 and DMS-CAREER-0847074.},
Kevin Zumbrun\thanks{Indiana University, Bloomington, IN 47405;
kzumbrun@indiana.edu: Research of K.Z. was partially supported
under NSF grants number DMS-0070765 and DMS-0300487.}
}

\begin{document}

\maketitle


\begin{abstract}
Extending investigations of Barker, Humpherys, Lafitte, Rudd, and Zumbrun for compressible gas dynamics and Freist\"uhler and Trakhinin for compressible magnetohydrodynamics, we study by a combination of asymptotic ODE estimates and numerical Evans function computations the one-dimensional stability of parallel isentropic magnetohydrodynamic shock layers over the full range of physical parameters (shock amplitude, strength of imposed magnetic field, viscosity, magnetic permeability, and electrical resistivity) for a $\gamma$-law gas with $\gamma\in [1,3]$.  Other $\gamma$-values may be treated similarly, but were not checked numerically.  Depending on magnetic field strength, these shocks may be of fast Lax, intermediate (overcompressive), or slow Lax type;  however, the shock layer is independent of magnetic field, consisting of a purely gas-dynamical profile.  In each case, our results indicate stability. Interesting features of the analysis are the need to renormalize the Evans function in order to pass continuously across parameter values where the shock changes type or toward the large-amplitude limit at frequency $\lambda=0$ and the systematic use of winding number computations on Riemann surfaces.
\end{abstract}

\section{Introduction}
In this paper, continuing investigations of \cite{BHRZ, HLZ,FT}, we study by a combination of asymptotic ODE estimates and numerical Evans function computations the one-dimensional stability of parallel isentropic magnetohydrodynamic (MHD) shock layers over a full range of physical parameters, including arbitrarily large shock amplitude and strength of imposed magnetic field, for a $\gamma$-law gas with $\gamma\in [1,3]$, with our main emphasis on the case of an ideal monatomic or diatomic gas.  The restriction to $\gamma \in [1,3]$ is an arbitrary one coming from the choice of parameters on which the numerical study is carried out; stability for other $\gamma$ can be easily checked as well. (Note that our analytical results are for any $\gamma \ge 1$.)  In each case, we obtain results indicative of stability.  Recall that Evans stability, defined in terms of the Evans function associated with the linearized operator about the wave, by the ``Lyapunov-type'' results of \cite{MaZ3,MaZ4,Z2,Z3,HR,HRZ,RZ}, implies linear and nonlinear stability for all except the measure-zero set of parameters on which the characteristic speeds of the endstates coincide with the shock speed or each other.\footnote{For these degenerate cases, the stability analysis has not been carried out in the generality considered here.  However, see the related analyses for Lax shock of \cite{HoZ,H} in the case that shock and characteristic speed coincide and \cite{Z2} in the case that characteristic speeds coincide, which suggest that the shocks may be nonetheless stable.}

Parallel shocks may be of fast Lax, intermediate (overcompressive), or slow Lax type depending on magnetic field strength; however, the shock layer is independent of magnetic field, consisting of a purely gas-dynamical profile.  Thus, the study of their stability is both a natural next step to and an interesting generalization of the investigations of stability of gas-dynamical shocks in \cite{HLZ}.  See also the investigations of stability of fast parallel Lax shocks in certain parameter regimes in \cite{FT} using energy methods, and of general fast Lax shocks in the small-magnetic field limit in \cite{GMWZ5,GMWZ6} using Evans function techniques.

\subsection{Equations}\label{eqs}
In Lagrangian coordinates, the equations for compressible isentropic MHD in one dimension take the form 
\begin{equation}
\left\{\begin{array}{l}
v_t -u_{1x} = 0,\\
 u_{1t} + (p+ (1/2\mu_0)(B_2^2+B_3^2))_x =(((2\mu+\eta)/v) u_{1x})_x,\\
 u_{2t}  - ((1/\mu_0)B_1^*B_2)_x =((\mu/v) u_{2x})_x,\\
 u_{3t}  - ((1/\mu_0)B_1^*B_3)_x =((\mu/v) u_{3x})_x,\\
 (vB_2)_{t}  - (B_1^*u_2)_x =((1/\sigma\mu_0 v) B_{2x})_x,\\
 (vB_3)_{t}  - (B_1^*u_3)_x =((1/\sigma\mu_0 v) B_{3x})_x,\\
\end{array}\right.
\label{MHD}
\end{equation}
where $v$ denotes specific volume, $u=(u_1,u_2,u_3)$ velocity, $p=p(v)$ pressure, $B=(B_1^*,B_2,B_3)$ magnetic induction, $B_1^*$ constant, and $\mu>0$ and $\eta>0$ the two coefficients of viscosity, $\mu_0>0$ the magnetic permeability, and $\sigma>0$ the electrical resistivity; see \cite{A,C,J,Kaw} for further discussion.

We restrict to an ideal isentropic polytropic gas, in which case the pressure function takes form
\begin{equation}
p(v)=av^{-\gamma}
\label{eq:ideal_gas}
\end{equation}
where $a > 0$ and $\gamma>1$ are constants that characterize the gas.  In our numerical investigations, we shall focus mainly on the most common cases of a monatomic gas, $\gamma=5/3$, and a diatomic gas, $\gamma=7/5$; more generally, we investigate all $\gamma\in [1,3]$.  With brief exceptions (e.g., Section \ref{ratio}), we take
\begin{equation}
\label{eta}
\eta=-2\mu/3,
\end{equation}
as typically prescribed for (nonmagnetic) gas dynamics \cite{Ba}.

Here, we are allowing $u$ and $B$ to vary in full three-dimensional space, but restricting spatial dependence to a single direction $e_1$ measured by $x$.  That is, we consider {\it planar solutions}, or three-dimensional solutions with one-dimensional dependence
on spatial variables.  Note that the divergence-free condition ${\rm div}_x B\equiv 0$ of full MHD reduces in the planar case to our assumption that $B_1\equiv \const=B_1^*$.  In the simplest, {\it parallel} case 
\be
\label{par}
B_2=B_3\equiv 0; \, 
u_2=u_3\equiv 0,
\ee
equations \eqref{MHD} reduce to the one-dimensional isentropic compressible Navier--Stokes equations
\begin{equation}
\left\{\begin{array}{l}
v_t -u_{1x} = 0,\\
 u_{1t} + p_x =(((2\mu+\eta)/v) u_{1x})_x.\\
\end{array}\right.
\label{psys}
\end{equation}
In the remainder of the paper, we study traveling-wave solutions in this special parallel case and their stability with respect to general (not necessarily parallel) planar perturbations.

\subsection{Viscous shock profiles}
\label{sec:viscous}
A \emph{viscous shock profile} of \eqref{MHD} is an asymptotically-constant traveling-wave solution
\ba
(v,u,B)(x,t)&=(\hat v,\hat u, \hat B) (x-st),
\quad 
\lim_{z\to \pm \infty} = (v_\pm,u_\pm,B_\pm).
\label{eq:tw_ansatz}
\ea
In the parallel case, these are of the simple form
$$
(\hat v,\hat u, \hat B) (x-st)=
(\hat v,\hat u_1,0,0,B_1^*,0,0 ) (x-st),
$$
where $(\hat v, \hat u_1)$ is a gas-dynamical shock profile satisfying the traveling-wave ODE
\begin{equation}
\left\{\begin{array}{l}
-sv_x -u_{1x} = 0,\\
  -s u_{1x} + p_x =(((2\mu+ \eta)/v) u_{1x})_x.\\
\end{array}\right.
\label{prof1}
\end{equation}

\subsection{Rescaled equations}\label{rescaled}

By a preliminary rescaling in $x$, $t$, we may arrange without loss of generality $\mu=1$.  Following the approach of \cite{HLZ,HLyZ1,HLyZ2}, we now rescale
\be\nonumber
(v,u_1,u_2,u_3,\mu_0,x,t,B)\to
\Big(\frac{v}{\varepsilon}, -\frac{u_1}{\varepsilon s}, 
\frac{u_2}{\varepsilon}, \frac{u_3}{\varepsilon},
\varepsilon \mu_0, -\varepsilon s (x-st), \varepsilon s^2t,\frac{B}{s}\Big)
\ee
holding $\mu$, $\sigma$ fixed, where $\eps:=v_-$, transforming \eqref{MHD} to the form
\begin{small}
\begin{equation}\label{redeqs}
\left\{
\begin{aligned}
    v_t+v_x-u_{1x}&=0\\
    u_{1t}+u_{1x}+\left(av^{-\gamma}+\left(\frac{1}{2\mu_0}\right)\left(B_2^2+B_3^2\right)\right)_x&=(2\mu +\eta)\left(\frac{u_{1x}}{v}\right)_x\\
    u_{2t}+u_{2x}-\left(\frac{1}{\mu_0}B_1^*B_2\right)_x&=\ \mu \left(\frac{u_{2x}}{v}\right)_x  \\
    u_{3t}+u_{3x}-\left(\frac{1}{\mu_0}B_1^*B_3\right)_x&=\mu \left(\frac{u_{3x}}{v}\right)_x\\
    \left(vB_2\right)_t+\left(vB_2\right)_x-\left(B_1^*u_2\right)_x&=\left(\left(\frac{1}{\sigma \mu_0 v}\right)B_{2x}\right)_x\\
    \left(vB_3\right)_t+\left(vB_3\right)_x-\left(B_1^*u_3\right)_x&=\left(\left(\frac{1}{\sigma \mu_0 v}\right)B_{3x}\right)_x
\end{aligned}
\right.
\end{equation}
\end{small}
where $p(v)=a_0v^{-\gamma}$ and $a=a_0\varepsilon^{-\gamma-1}s^{-2}$.

By this step, we reduce without loss of generality to the case of a shock profile with speed $s=-1$, left endstate
\be\label{left}
(v,u_1,u_2,u_3,B_1,B_2,B_3)_-= (1,0,0,0,B_1^*,0,0),
\ee
and right endstate
\be\label{right}
(v,u_1,u_2,u_3,B_1,B_2,B_3)_+= (v_+,v_+-1 ,0,0,B_1^*,0,0),
\ee
satisfying the profile ODE 
\begin{equation}
\label{profeq}
(2\mu+\eta) v' =H(v, v_+):= v(v-1 + a  (v^{-\gamma}-1))
\end{equation}
(obtained by integrating \eqref{prof1} and substituting the first equation into the second) where $1=v_-\ge v_+>0$ and (setting $v'=0$ at $v=v_+$ and solving)
\begin{equation}
\label{RH}
a = -\frac{v_+ - 1}{v_+^{-\gamma} - 1} = v_+^\gamma \frac{1-v_+}{1-v_+^\gamma}.
\end{equation}
See \cite{BHRZ,HLZ} for further details.

\begin{proposition} [\cite{BHRZ}]
\label{profdecay}
For each $\gamma\ge 1$, $0<v_+\le 1-\eps$, $\eps>0$, \eqref{profeq} has a unique (up to translation) monotone decreasing solution $\hat v$ decaying to its endstates with a uniform exponential rate, independent of $v_+$, $\gamma$.  In particular, for $0<v_+\le \frac{1}{12}$ and $\hat v(0):=v_+ + \frac{1}{12}$, 
\begin{subequations}
\label{decaybd}
\begin{align}
|\bV(x)-v_+|&\le \Big(\frac{1}{12}\Big)e^{-\frac{3x} {4}} \quad x\ge 0,\label{decaybd_1}\\
|\bV(x)-v_-|&\le 
\Big(\frac{1}{4}\Big)
e^{\frac{x+12}{2}} \quad x\le 0\label{decaybd_2}.
\end{align}
\end{subequations}
\end{proposition}

\bc\label{limv}
Initializing $\hat v(0):=v_+ + \frac{1}{12}$ as in Proposition \ref{profdecay}, $\hat v$ converges uniformly as $v_+\to 0$ to a translate $\hat v_0$ of $ \frac{1-\tanh \big( \frac{x}{2(2\mu+\eta)} \big)}{2} $.
\ec

\begin{proof}
By \eqref{RH}, $a\sim v_+^\gamma \to 0$ as $v_+\to 0$, whence the result follows on any bounded set $|x|\le L$ by continuous dependence, taking the limit as $a\to 0$ in \eqref{profeq} to obtain a limiting flow of $ v'=\frac{v(1-v)}{2\mu+\eta}$.  Taking now $L\to \infty$, the result follows for $|x|\ge M$ by $v_+\to 0$ and $|\hat v-v_+|\le Ce^{-\theta M}$; see \eqref{decaybd_1}--\eqref{decaybd_2}.
\end{proof}

\subsection{Families of shock profiles}
At this point, we have reduced our study of parallel shock stability, for a fixed gas constant $\gamma$, to consideration of a one-parameter family of profiles indexed by the right endstate $1\ge v_+ > 0$ and a four-parameter family of equations \eqref{redeqs} indexed (through \eqref{RH}) by $v_+$ and the three remaining physical parameters 
\be
\label{paramlist}
\mu_0 >0, \, \sigma >0,  \,B_1^*\ge 0,
\ee
where we have taken $B_1^*$ without loss of generality to be nonnegative by use of the symmetry under $B\to -B$ of 
\eqref{MHD}.  Here, the small-amplitude limit corresponds to $v_+\to v_-=1$ and the large-amplitude limit to $v_+\to 0$, where in this scaling the amplitude is given by $|v_--v_+|$.

A straightforward computation 
shows that the characteristics of the first-order hyperbolic system obtained by neglecting second-derivative terms in \eqref{redeqs} at the endstates $v_\pm$ have values
\be\label{charvals}
(1\pm c(v)), 1, 1, \Big(1\pm \frac{B_1^*}{\sqrt{\mu_0 v_\pm}}\Big),
\ee
where $c(v):=\sqrt{-p'(v)}=\sqrt{\gamma a v^{-\gamma-1}}$ is the gas-dynamical sound speed, satisfying $c_+>1>c_-$.  Thus, the shock is a {\it Lax $1$-shock} for $0\le B_1^* < \sqrt{\mu_0 v_+}$, meaning that it has six positive characteristics at $v_-$ and one at $v_+$; an {\it intermediate doubly overcompressive shock} for $\sqrt{\mu_0 v_+} < B_1^* < \sqrt{\mu_0 }$, meaning that it has six positive characteristics at $v_-$ and three at $v_+$; and a {\it Lax $3$-shock} for $\sqrt{\mu_0 }< B_1^*$, meaning that it has $4$ positive characteristics at $v_-$ and three at $v_+$.

For Lax $1$- and $3$-shocks, the profile \eqref{eq:tw_ansatz} is generically (and always for $1$-shocks) unique up to translation as a traveling-wave solution of the full equations connecting endstates \eqref{left} and \eqref{right}, i.e., even among possibly nonparallel solutions.  That is, it lies generically within a one-parameter family
$\{\hat U^\xi\}=\{ (\hat v, \hat u_1, \hat u_2, \hat u_3, \hat B_2, \hat B_3)^\xi\}$
of viscous shock profiles, $\xi \in \R$, with $\hat U^\xi(x):=\hat U(x-\xi)$. For overcompressive shocks, it lies generically within a three-parameter family
$\{ (\hat v, \hat u_1, \hat u_2, \hat u_3, \hat B_2, \hat B_3)^\xi\}$
of viscous profiles and their translates, $\xi \in \R^3$, of which it is the unique parallel solution up to translation \cite{MaZ3}.  For further discussion of hyperbolic shock type and its relation to existence of viscous profiles, see, e.g., \cite{LZ,ZH,Z1,Z2,MaZ3}.

\subsection{Evans, spectral, and nonlinear stability}\label{sec:spec}

Following \cite{ZH,MaZ3,Z2}, define {\it spectral stability} as nonexistence of nonstable eigenvalues $\Re \lambda\ge 0$ of the linearized operator about the wave, other than at $\lambda=0$ (where there is always an eigenvalue, due to translational invariance of the underlying equations).  A slightly stronger condition is {\it Evans stability}, which for Lax or overcompressive shocks may be defined \cite{ZH,MaZ3,HLyZ2} as nonvanishing for all $\Re\lambda\ge 0$ of the Evans function associated with the integrated eigenvalue equation about the wave.  See \cite{AGJ,GZ,Z1,Z2,MaZ3} for a general definition of the Evans function associated with a system of ordinary differential equations; for a definition in the present context, see Section \ref{evans}.  Recall that zeros of the Evans function (either integrated or nonintegrated) agree with eigenvalues of the linearized operator about the wave on $\{\Re \lambda\ge 0\} \setminus \{0\}$, so that Evans stability implies spectral stability.

The following ``Lyapunov-type'' result of Raoofi \cite{Ra}, specialized to our case, states that, for generic parameter values, Evans stability implies {\it nonlinear orbital stability}, regardless of the type 
of the shock; see also \cite{MaZ4,Z2,HRZ,RZ}.

\begin{proposition}[\cite{Ra}]\label{orbital}
Let $\hat U:=(\hat v, \hat u_1, \hat u_2, \hat u_3, \hat B_2, \hat B_3)$ be a parallel viscous shock profile of \eqref{MHD}--\eqref{eq:ideal_gas} connecting endstates \eqref{left}--\eqref{right}, 
with characteristics \eqref{charvals} distinct and nonzero, that is Evans stable.
Then, for any solution 
$\tilde U:=(\tilde v, \tilde u_1, \tilde u_2, \tilde u_3, \tilde B_2, \tilde B_3)$ 
of \eqref{MHD} with $L^1\cap H^3$ initial difference and $L^1$-first moment
$E_0:=\| \tilde U(\cdot, 0)- \hat U\|_{L^1\cap H^3}$
and
$E_1:=\| |x| \, | \tilde U(\cdot, 0)- \hat U| \|_{L^1}$
sufficiently small and some uniform $C>0$, $\tilde U$ exists for all $t\ge 0$, with
\begin{equation}
\label{stabstatement}
\begin{aligned}
\|  \tilde U(\cdot, t)- \hat U(\cdot -st)\|_{L^1\cap H^3}
&\le CE_0
\quad
\hbox{\rm (stability)}.\\
\end{aligned}
\end{equation}
Moreover, there exist $\alpha(t)$, $\alpha_\infty$ such that
\begin{equation}\label{stabstatement2}
\begin{aligned}
\| \tilde U(\cdot, t)-  \hat U^{\alpha(t)}(\cdot -st))\|_{L^p}
&\le CE_0(1+t)^{-(1/2)(1-1/p)},\\
\end{aligned}
\end{equation}
and
\begin{equation}
\label{phasebd}
|\alpha(t)- \alpha_\infty|,
\, 
(1+t)^{1/2}|\dot \alpha(t)|
\le C(\varepsilon) 
\max\{E_0,E_1\} (1+t)^{-1/2+\varepsilon}, 
\end{equation}
for all $1\le p\le \infty$,
$\varepsilon>0$ arbitrary
\quad
{\rm (phase-asymptotic orbital stability)}.
\end{proposition}

Finally, recalling that Evans stability for Lax shocks is equivalent to the three conditions of spectral stability, transversality of the traveling wave as a connecting orbit of \eqref{profeq}, and inviscid stability of the shock while Evans stability for overcompressive shocks is equivalent to spectral stability, transversality, and an ``inviscid stability''-like low-frequency stability condition generalizing the Lopatinski condition of the Lax case \cite{ZH,MaZ3,Z2}, we obtain the following partial converse allowing us to make stability conclusions from spectral information alone.

\bpr
\label{loptran}
A parallel viscous shock profile of \eqref{MHD}--\eqref{eq:ideal_gas}, 
\eqref{left}--\eqref{right}, 
that is a Lax $1$-shock and spectrally stable is also Evans stable (hence, for generic parameters, nonlinearly orbitally stable).
A parallel viscous shock profile that is an intermediate (overcompressive) shock, spectrally stable and low-frequency stable is Evans stable.  A parallel viscous shock profile that is a Lax $3$-shock, spectrally stable, and transverse is Evans stable.
For $\mu=1$ and 
$B_1^*\ge \max\left\{ \sqrt{ \frac{\gamma \mu_0}{2}},
\sqrt{\frac{\gamma}{2\sigma}} \, \right\} + \sqrt{\mu_0}$,
Lax $3$-shocks are transverse. For parallel viscous shocks of any type, spectral stability implies Evans (and nonlinear) stability on a generic set of parameters.
\epr

\begin{proof}
Lax $1$-shocks and intermediate-shocks, as extreme shocks (i.e.,
all characteristics entering the shock from the $-\infty$ side), 
are always transversal \cite{MaZ3}.
One-dimensional inviscid stability of either Lax $1$- or $3$-shocks
follows by a straightforward calculation using decoupling of the linearized
equations into $(v,u_1)$ and $(u_2, B_2)$ and $(u_3,B_3)$ systems
\cite{BT,T,FT}.
Transversality for large $B_1^*$ is shown in Proposition \ref{transest}.
Finally, both transversality (in the Lax $3$-shock case) 
and (in the overcompressive case) low-frequency stability conditions 
can be expressed as nonvanishing of functions that are analytic
in the model parameters, hence either vanish everywhere 
or on a measure zero set.
It may be shown that these are both
nonvanishing for sufficiently weak profiles $|1-v_+|$ small,\footnote{
For Lax $3$-shocks, transversality follows for small amplitudes by
the center-manifold analysis of \cite{P}.
For overcompressive shocks, taking $B_1^*= (1/2)(\sqrt{\mu_0} + \sqrt{\mu v_+})$
as $v_+\to 1$, using decoupling of $(v,u_1)$ and $(u_j,B_j)$ equations
and performing a center manifold reduction in the $(u_j,B_j)$ equation
of the traveling-wave ODE written as a first-order system, $j=2,3$, we find
that this reduces in each case to a one-dimensional fiber, whence decaying
solutions of the linearized profile equation, corresponding to variations
other than translation in the family of profiles $\hat U^\alpha$, 
are of one sign and thus have nonzero 
total integral $\int_{-\infty}^{+\infty}(u_j,B_j)(x) dx$.
But this is readily seen \cite{Z1} to be equivalent to low-frequency
stability in the small-amplitude limit.
}
hence they are generically nonvanishing.
From these facts, the result follows.
\end{proof}

\br
Our numerical results indicate Evans stability for all parameters,
which implies in passing uniform transversality of $3$- 
and overcompressive-shock profiles and low-frequency stability
of overcompressive profiles.
Transversality is a minimal condition for orbital
stability, being needed even to guarantee existence of the smooth
manifold $\hat U^\alpha$ under discussion \cite{MaZ3}.
As discussed above, it is {\rm not} implied by spectral stability
alone.\footnote{Thus, for example, the spectral stability results
obtained by energy estimates in \cite{FT} for intermediate- or Lax $3$-shocks
do not by themselves imply linearized or nonlinear stability,
but require an additional study of transversality/low-frequency
stability.} 
\er

\subsection{The reduced linearized eigenvalue equations}
Linearizing \eqref{redeqs} 
about a parallel shock profile 
$(\hat v, \hat u_1, 0,0,B_1^*,0,0)$, we obtain a decoupled system
\begin{equation}\label{fullsys}
\left\{
\begin{aligned}
    v_t+v_x-u_{1x}&=0\\
    u_{1t}+u_{1x}-a\gamma \left(\hat v^{-\gamma-1}v\right)_x&=(2\mu +\eta)\left(\frac{u_{1x}}{\hat v}+\frac{\hat u_{1x}}{\hat v^2} v \right)_x\\
    u_{2t}+u_{2x}-\frac{1}{\mu_0}(B_1^*B_2)_x&=\ \mu \left( \frac{u_{2x}}{\hat v}  \right)_x  \\
    u_{3t}+u_{3x}-\frac{1}{\mu_0}(B_1^*B_3)_x&=\ \mu \left( \frac{u_{3x}}{\hat v}  \right)_x  \\
    \left(\hat vB_2\right)_t+\left(\hat vB_2\right)_x-(B_1^*u_2)_x&=\left(\left(\frac{1}{\sigma \mu_0 }\right)\frac{B_{2x}}{\hat v}\right)_x\\
    \left(\hat vB_3\right)_t+\left(\hat vB_3\right)_x-(B_1^*u_3)_x&=\left(\left(\frac{1}{\sigma \mu_0 }\right)\frac{B_{3x}}{\hat v}\right)_x,
\end{aligned}
\right.
\end{equation}
consisting of the linearized isentropic gas dynamic equations in
$(v,u_1)$ about profile $(\hat v, \hat u_1)$, and two copies
of an equation in variables $(u_j,\hat v B_j)$, $j=2,3$. 

Introducing integrated variables $V:=\int v$, $U:=\int u_1$ and
$w_j:=\int u_j$, $\alpha_j:=\int \hat v B_j$, $j=2,3$,
we find that the integrated linearized eigenvalue equations decouple into
the integrated linearized eigenvalue equations for gas dynamics
in variables $(V,U)$ and two copies of
\begin{equation}\label{eval}
\left\{
\begin{aligned}
    \lambda w+w'-\frac{B_1^*\alpha'}{\mu_0\hat v}&=\mu
    \frac{w''}{\hat v}\\
    \lambda \alpha +\alpha' -B_1^*w'&=\frac{1}{\sigma \mu_0\hat
    v}\left( \frac{\alpha'}{\hat v}\right)'
\end{aligned}
    \right.
\end{equation}
in variables $(w_j, \alpha_j)$, $j=2,3$.

As noted in \cite{ZH,MaZ3,HLyZ2}, spectral stability is unaffected
by the change to integrated variables.
Thus, spectral stability of parallel MHD shocks, 
decouples into the conditions of
spectral stability of the associated gas-dynamical shock
as a solution of the isentropic Navier--Stokes equations \eqref{psys},
and spectral stability of system \eqref{eval}.
Assuming stability of the gas-dynamical shock (as has been
verified in great generality in \cite{HLZ, HLyZ1}), {\it 
spectral stability of parallel MHD shocks thus reduces
to the study of the reduced eigenvalue problem \eqref{eval},}
into which the shock structure enters only through density
profile $\hat v$.
Likewise, the Evans function associated with the full system
\eqref{fullsys} decouples into the product of the Evans function 
for the gas-dyamical eigenvalue equations and the Evans function 
for the reduced eigenvalue problem \eqref{eval}.
Thus, assuming stability of the associated gas-dynamical shock,
{\it Evans stability of parallel MHD shocks reduces
to Evans stability of \eqref{eval}}.

\br\label{intrmk}
The change to integrated coordinates removes two additional zeros of
the Evans function for the reduced equations \eqref{eval} that
would otherwise occur at the origin in the overcompressive case,
making possible a unified study across different parameter values/shock types.
\er

\subsection{Analytical stability results}\label{sec:analytical}

\subsubsection{The case of infinite resistivity/permeability}
We start with the observation that, by a straightforward energy estimate,
parallel shocks are {\it unconditionally stable} in transverse
modes $(\tilde u,\tilde B)$ in the formal limit as either electrical
resistivity $\sigma$ or magnetic permeability $\mu_0$ go to infinity,
for quite general equations of state.
This is suggestive, perhaps, of a general trend toward stability.

\bt\label{enprop}
In the degenerate case $\mu_0=\infty$ or $\sigma=\infty$,
parallel MHD shocks are transversal, Lopatinski stable (resp. low-frequency
stable), and spectrally stable with respect to 
transverse modes $(\tilde u, \tilde B)$, 
for all physical parameter values, hence are Evans (and thus
nonlinearly) stable whenever
the associated gas-dynamical shock is Evans stable.
\et

\begin{proof}
By Proposition \ref{loptran}, Lopatinski stability holds for
Lax-type shocks, and transversality holds for Lax $1$-shocks and
overcompressive shocks.  
Noting that the (decoupled)
transverse part of linearized traveling-wave ODE for $\sigma=\infty$
or $\mu_0=\infty$ reduces to $(d-1)$ copies of the same scalar equation,
and recalling that transversality/Lopatinski stability hold always for the
decoupled gas-dynamical part \cite{HLZ}, we readily verify low-frequency
stability in the overcompressive case and transversality in the Lax
$3$-shock case as well.\footnote{ 
For scalar equations, transversality is immediate.
Likewise, decaying solutions of the linearized profile equation, 
corresponding to variations other than translation in the family of 
profiles $\hat U^\alpha$, are necessarily of one sign and thus have nonzero 
total integral $\int_{-\infty}^{+\infty}(u_j,B_j)(x) dx$.
But this is readily seen \cite{Z1} to be equivalent to low-frequency
stability in the small-amplitude limit.
}
Thus, we need only verify transverse spectral stability, or nonexistence
of decaying solutions of \eqref{eval}.

For $\sigma=\infty$, we may rewrite \eqref{eval} in symmetric form as
\begin{equation}\label{infsigmares}
\begin{aligned}
 \mu_0 \hat v \lambda w + \mu_0 \hat v w'  - B_1^* \alpha' &=\mu \mu_0 w'',\\
 \lambda \alpha  + \alpha' - B_1^*w' &=0.
\end{aligned}
\end{equation}
Taking the real part of the complex $L^2$-inner product of 
$w$ against the first equation and
$\alpha$ against the second equation and summing gives
$$
\Re \lambda( \int (\hat v \mu_0 |w|^2+|\alpha|^2) =
-\int \mu \mu_0 |w'|^2 +\int \hat v_x |w|^2 <0,
$$
a contradiction for $\Re \lambda \ge 0$ and $w$
not identically zero.  If $w\equiv 0$ on the other
hand, we have a constant-coefficient equation for $\alpha$,
which is therefore stable. The $\mu_0=\infty$ case goes
similarly; see Appendix \ref{infmu0}.
\end{proof}

Notably, this includes all three cases: fast Lax, overcompressive,
and slow Lax type shock.
Further, the same proof yields the result
for the more general class of equations of state $p(\cdot)$
satisfying 
$\frac{p(v_+)-p(v_-)}{v_+-v_-}<0$, so that $\hat v_x<0$ for $s<0$.
With the analytical results of \cite{HLZ}, 
we obtain in particular the following asymptotic results.

\bc\label{rigstab}
For $\sigma=\infty$ or $\mu_0=\infty$,
parallel isentropic MHD shocks with ideal gas equation of state,
whether Lax or overcompressive type,
are linearly and nonlinearly stable in 
the small- and large-amplitude limits $v_+\to 1$ and $v_+\to 0$,
for all physical parameter values.
\ec

\subsubsection{Bounds on the unstable spectrum}
\label{s:hf}

By a considerably more sophisticated energy estimate,
we can bound the size of unstable eigenvalues 
{\it uniformly} in $1\ge v_+>0$ and the gas constant $\gamma\ge 1$ 
to a ball of radius depending on $\sigma$, $\mu_0$, $B_1^*$,
a crucial step in studying the limit $v_+\to 0$.

\begin{theorem}\label{hfthm}
Nonstable eigenvalues $\Re \lambda\ge 0$ of \eqref{eval}
are confined for $0<v_+\le 1$ to the region 
\begin{equation}
\label{est2}
\Re \lambda  + |\Im \lambda | < \frac{1}{2}\max 
\Big\{\frac{1}{\mu},\mu_{0}\sigma \Big\} 
+ (B_1^*)^{2} \sqrt{\dfrac{\sigma}{\mu\mu_{0}}}.
\end{equation}
\end{theorem}

\begin{proof}
See Appendix \ref{misc}.
\end{proof}

\subsubsection{Asymptotic Evans function analysis}

Denoting by $D(\lambda)$ the ``reduced'' Evans function 
(defined Section \ref{evans}) 
associated with the reduced eigenvalue equations \eqref{eval},
we introduce the pair of renormalizations
\be \label{checkevans}
\begin{aligned}
\check D(\lambda)&:=
\frac{ \Big( (1-B_1^*/\sqrt{\mu_0 })^2 + 4\lambda
 (\mu/2 + 1/2\sigma \mu_0 )
\Big)^{1/4}}
{ \Big( (1-B_1^*/\sqrt{\mu_0 })^2 + 4
 (\mu/2 + 1/2\sigma \mu_0 )
\Big)^{1/4}}
 \frac {\Big(v_+/4+\lambda\Big)^{1/4}} {\Big(v_+/4 + 1 \Big)^{1/4}}
\\
&\quad \times
\frac{ \Big( (1-B_1^*/\sqrt{\mu_0 v_+})^2 + 4\lambda
 (\mu/2v_+ + 1/2\sigma \mu_0 v_+^2)
\Big)^{1/4}} 
 {\Big( (1-B_1^*/\sqrt{\mu_0 v_+})^2 + 
4 (\mu/2v_+ + 1/2\sigma \mu_0 v_+^2) \Big)^{1/4}} 
D(\lambda)\\
\end{aligned}
\ee
and
\be \label{hatevans}
\begin{aligned}
\hat D(\lambda)&:=
\frac{ \Big( (1-B_1^*/\sqrt{\mu_0 })^2 + 4\lambda
 (\mu/2 + 1/2\sigma \mu_0 )
\Big)^{1/4}}
 {\Big( (1-B_1^*/\sqrt{\mu_0 })^2 + 
4 (\mu/2 + 1/2\sigma \mu_0 ) \Big)^{1/4}} D(\lambda).\\
\end{aligned}
\ee

\noindent{\bf Intermediate behavior.}
\bt\label{contD}
On $\Re \lambda\ge 0$, the reduced Evans function $D$ is analytic
in $\lambda$ and continuous in all parameters
except at $v_+=0$
and $B_1=\sqrt{\mu_0 v_\pm}$, at which points
it exhibits algebraic singularities (blow-up) at $\lambda=0$.
The renormalized Evans functions $\check D$ and $\hat D$
are analytic in $\lambda$
and continuous in all parameters except at $(\lambda, v_+)=(0,0)$.
\et

\begin{proof}
Immediate from Propositions \ref{evprops} and \ref{limconv}.
\end{proof}

\noindent{\bf The small-amplitude limit.}
\bpr [\cite{HuZ1,PZ}] \label{smallamp}
For $\sigma>0$, $\mu_0>0$, $B_1^*>0$ bounded, and 
$B_1^*$ bounded away from $\sqrt{\mu_0}$,
parallel shocks are Evans stable in both full and reduced sense
in the small-amplitude limit $v_+\to 1$.
Moreover, $D$ converges uniformly on compact subsets of $\{\Re \lambda \ge 0\}$
 as $v_+\to 1$ to a nonzero real constant.
\epr

\begin{proof}
For $|v_--v_+|=|1-v_+|$ sufficiently small and $B_1^*$ bounded
away from $\sqrt{\mu_0}$, the associated
profile must be a Lax $1$- or $3$-shock, whence stability
follows by the small-amplitude results obtained by energy
estimates in \cite{HuZ1} or
by asymptotic Evans function techniques in \cite{PZ}.
Convergence on compact sets
follows by the argument of
Proposition 4.9, \cite{HLyZ1}, which likewise uses techniques from \cite{PZ}.
\end{proof}

\noindent {\bf The large-amplitude limit.}
\bt\label{largeamp}
For $\sigma$, $\mu_0$ and $B_1^*$ bounded, the reduced Evans function
$D(\lambda)$ converges uniformly on compact subsets of 
$\{\Re \lambda\ge 0 \}\setminus \{0\}$ 
in the large-amplitude limit $v_+\to 0$ 
to a limiting Evans function $D^0(\lambda)$ obtained by substituting
$\hat v_0$ for $\hat v$ in \eqref{eval}, $\hat v_0$ as in Proposition \ref{limv};
see Definition \ref{limev} for a precise definition.
Likewise, $\check D$ and $\hat D$ converge to
\be\label{checkD1}
\check D^0(\lambda):=
\frac{ \Big(1-B_1^*/\sqrt{\mu_0 })^2 + 4\lambda
 (\mu/2 + 1/2\sigma \mu_0 )
\Big)^{1/4}}
{ \Big(1-B_1^*/\sqrt{\mu_0 })^2 + 4
 (\mu/2 + 1/2\sigma \mu_0 )
\Big)^{1/4}}
\lambda^{1/2} D^0(\lambda) 
\ee
and
\be\label{hatD1}
\hat D^0(\lambda):=\frac{ \Big( (1-B_1^*/\sqrt{\mu_0 })^2 + 4\lambda
 (\mu/2 + 1/2\sigma \mu_0 )
\Big)^{1/4}}
 {\Big( (1-B_1^*/\sqrt{\mu_0 })^2 + 
4 (\mu/2 + 1/2\sigma \mu_0 ) \Big)^{1/4}} D^0(\lambda),
	\ee
%
%
each continuous on $\Re \lambda\ge 0$.
Moreover, for $B_1^*< \sqrt{\mu_0}$, 
nonvanishing of  $\check D^0$ on $\{\Re \lambda > 0\}$ is necessary and
nonvanishing of $\check D^0$ on $\{\Re \lambda \ge 0\}$
is sufficient for reduced Evans stability (i.e., nonvanishing of $\check D$, $D$
on $\{\Re \lambda > 0\}$ for $v_+>0$ sufficiently small. 
For $B_1^* > \sqrt{\mu_0}$,
nonvanishing of $\hat D^0$ on $\{\Re \lambda > 0\}$ 
is necessary and
nonvanishing of $\hat D^0$ on $\{\Re \lambda \ge 0\}$ 
together with a certain sign condition on $\hat D(0)$ 
is sufficient for reduced Evans stability 
for $v_+>0$ sufficiently small.\footnote{
$D$, $\hat D$ are real-valued for real $\lambda$ by construction, 
so that $\sgn \hat D(0)$ is well-defined.}
This sign condition is implied in particular by
nonvanishing of $\hat D(0)$ on the range
\be\label{signcond}
\sqrt{\mu_0}\le B_1^* \le  \sqrt{\mu_0} + \max\Big\{ \sqrt{ \frac{\mu_0}{2}} ,
\sqrt{\frac{1}{2\sigma}} \, \Big\} . 
\ee
\et

\begin{proof}
Convergence follows by Proposition \ref{limconv};
for stability criteria, see Section \ref{largepf}.
\end{proof}

\br\label{signrmk}
The theoretically cumbersome condition \eqref{signcond}
is in practice no restriction, since we check in any case the
stronger condition of nonvanishing of
$\hat D$ for $\Re \lambda\ge 0$ on the entire range 
$\sqrt{\mu_0}\le B_1^*\le
\sqrt{\mu_0} + \max\Big\{ \sqrt{ \frac{\mu_0}{2}} ,
\sqrt{\frac{1}{2\sigma}} \, \Big\} $.
\er

\br\label{tr}
Recall \cite{HLZ} that the associated gas-dynamical
shock has already been shown to be Evans stable for $v_+>0$ sufficiently small.
Thus, not only reduced Evans stability, but full Evans stability, 
is implied for $v_+>0$ sufficiently small
by stability of the limiting function $\check D^0$ (resp. $\hat D^0$). 
\er

\noindent{\bf Large- and small-parameter limits.}
\bt\label{largeB}
For $\sigma$, $\mu_0$ bounded and bounded from zero, 
and $v_+$ bounded from zero,
parallel shocks are reduced Evans stable in the limit as
$B_1^*\to \infty$ or $B_1^*\to 0$.
For $\lambda $ bounded and $\Re \lambda\ge 0$, the Evans function
converges as $B_1^*\to \infty$ to a constant.
\et

\bt\label{sigmu}
For $B_1^*$ bounded, and $v_+$ bounded from zero,
parallel shocks are reduced Evans stable in the limit as
$\sigma \to 0$ with $\mu_0$ bounded, $\mu_0\to 0$ with
$\sigma$ bounded.
In each case, the Evans function converges uniformly
to zero on compact subsets of $\{\Re \lambda \ge 0\}$,
with $C^{-1}\sqrt{\sigma}\le |D^\sigma(\lambda)| \le C\sqrt{\sigma}$
for $C\sigma \le |\lambda|\le C$
and
$C^{-1}\sqrt{\mu_0}\le |D^\sigma(\lambda)| \le C\sqrt{\mu_0}$
for $C\mu_0 \le |\lambda|\le C$.
\et

\bt\label{sigmu2}
For $B_1^*$ bounded, and $v_+$ and $\mu_0$ bounded from zero,
parallel shocks are reduced Evans stable in the limit as
$\sigma \mu_0\to \infty$.
For $\lambda $ bounded and $\Re \lambda\ge 0$, the Evans function $D$,
appropriately renormalized, converges
as $\sigma \mu_0\to \infty$ to the Evans function $\hat D$ 
for \eqref{infsigmares};
more precisely, $D\sim e^{c_0 \sigma \mu_0 + c_1 + c_2\lambda}\hat D$ for $c_j$ constant.
\et

\br\label{corners}
Except for certain ``corner points'' consisting of simultaneous
limits of $B_1^*\to \infty$ together with  $\sigma \to 0$, 
or $\sigma \mu_0\to \infty$
together with $v_+\to 0$ or $\mu_0\to 0$, 
our analytic results verify stability on all but a (large but)
compact set of parameters.
We conjecture that stability holds in these
limits as well; this would be an interesting question for
further investigation.

As pointed out in \cite{HLyZ1}, the limit $v_+\to 0$ is connected
with the isentropic approximation, and does not occur for full
(nonisentropic) MHD for gas constant $\gamma>1$; 
thus, a somewhat more comprehensive analysis
is possible in that case.
Note that the reduced eigenvalue equations are identical in the
nonisentropic case \cite{FT}, except with $\hat v$ replaced 
by a full (nonisentropic) gas-dynamical profile, 
from which observation the reader may check that
all of the analytical results of this paper goes through unchanged
in the nonisentropic case, since the analysis depends only on
$\hat v$, and this only through properties of monotone decrease,
$|v_x|\le C|v|$ (immediate for $v$ bounded from zero),
and uniform exponential convergence as $x\to \pm\infty$,
that are common to both the isentropic and nonisentropic ideal gas cases.
\er

\noindent{\bf Discussion.}
Taken together, and along with the previous theoretical and numerical
investigations of \cite{HLZ} on stability of
gas-dynamical shocks, our asymptotic stability results reduce the study
of stability of parallel MHD shocks, 
in accordance with the general philosophy set out in \cite{HLZ,HLyZ1,HLyZ2},
mainly (i.e., with the exception of ``corner points'' discussed
in Remark \ref{corners}) to investigation of the continuous and numerically
well-conditioned renormalized functions $\check D$ and $\hat D$
on a compact parameter-range suitable for discretization, 
together with investigation of the similarly well-conditioned limiting
functions $\check D^0$ and $\hat D^0$.
However, notice that the same results show that the
unrenormalized Evans function $D$ blows up as $\lambda\to 0$,
both in the large-amplitude limit $v_+\to 0$ and in the
characteristic limits $B_1^*\to \sqrt{\mu_0v_\pm}$, hence
is {\it not} suitable for numerical testing across the entire
parameter range.  Indeed, in practice these
singularities dominate behavior even rather far from 
the actual blow-up points, making numerical investigation 
infeasibly expensive if renormalization is not carried out, even
for intermediate values of parameters/frequencies.
This is a substantial difference between the current and previous
analyses, and represents the main new difficulty that we have
overcome in the present work.


\subsection{Numerical stability results}
\label{sec:numerical}

For a given amplitude, the above analytical results truncate the computational domain to a compact set, thus allowing for a comprehensive numerical Evans function study patterned after \cite{HLZ,HLyZ1}, which yields Evans stability in the intermediate parameter range.  We then demonstrate Evans stability in the large-amplitude limit by (i) verifying convergence to the limiting Evans functions given in Theorem \ref{largeamp} (i.e., checking that convergence has occurred to desired tolerance at the limits of values $v_+$, $\lambda$ considered), and (ii) verifying nonvanishing on $\Re \lambda\ge 0$ of the limiting functions $\hat D^0$, $\check D^0$.  These computational results, together with the analytical results in Section \ref{sec:analytical}, give unconditional stability for all values except for cases where two or more parameters blow up simultaneously as described in Remark \ref{corners}.  The numerical computations were performed by the authors' STABLAB package, which is written in MATLAB, and has been used successfully for several systems \cite{BHRZ,HLZ,HLyZ1,CHNZ,Hu3,HLyZ2}.

When compared to the numerical study for isentropic Navier-Stokes \cite{BHRZ,HLZ}, this present system is better conditioned, yet much more computationally taxing since there are more free parameters to cover, i.e., $(\gamma,v_+,B^*_1, \mu_0,\sigma)$; the isentropic model by contrast has only two parameters $(\gamma,v_+)$.  Since each dimension adds, roughly, an order of magnitude to the runtime, we upgraded our STABLAB package to allow for parallel computation via MATLAB's parallel computing toolbox.  In our main study, we computed along $30,\!870$ semi-circular contours corresponding to the parameter values
\[
(\gamma,v_+,B^*_1, \mu_0,\sigma) \in [1.0,3.0] \times [10^{-5},0.8] \times [0.2,3.8] \times [0.2,3.8] \times[0.2,3.8].
\]
In every case, the winding number was zero, thus demonstrating Evans stability; see Section \ref{numerics} for more details.

We also carried out a number of small studies to illustrate our analytical work in the limiting fixed-amplitude cases.  These are briefly described below and are also given more detail in Section \ref{numerics}.

\begin{figure}[t]
\begin{center}
$\begin{array}{lr}
\includegraphics[width=7.5cm]{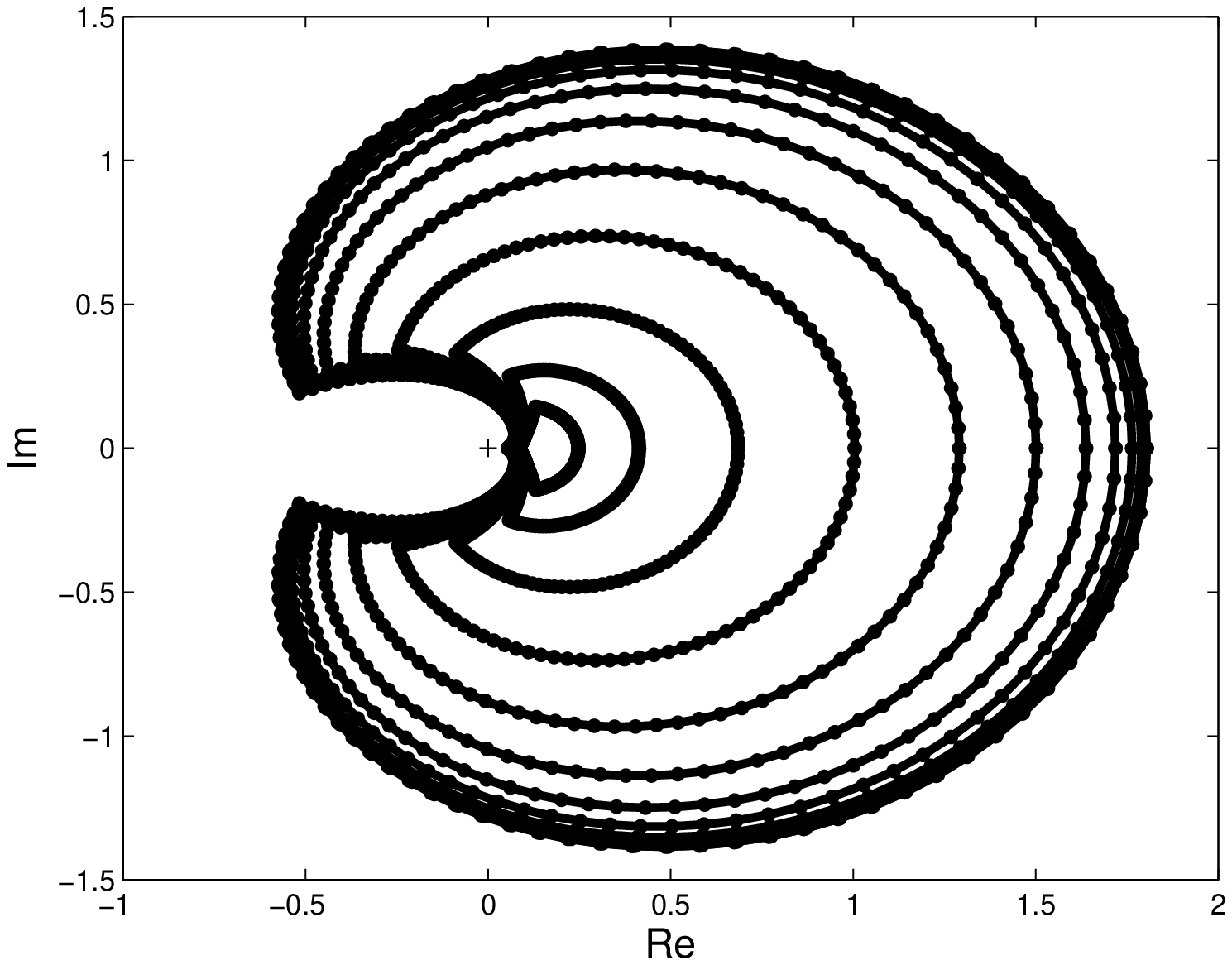} & \includegraphics[width=7.5cm]{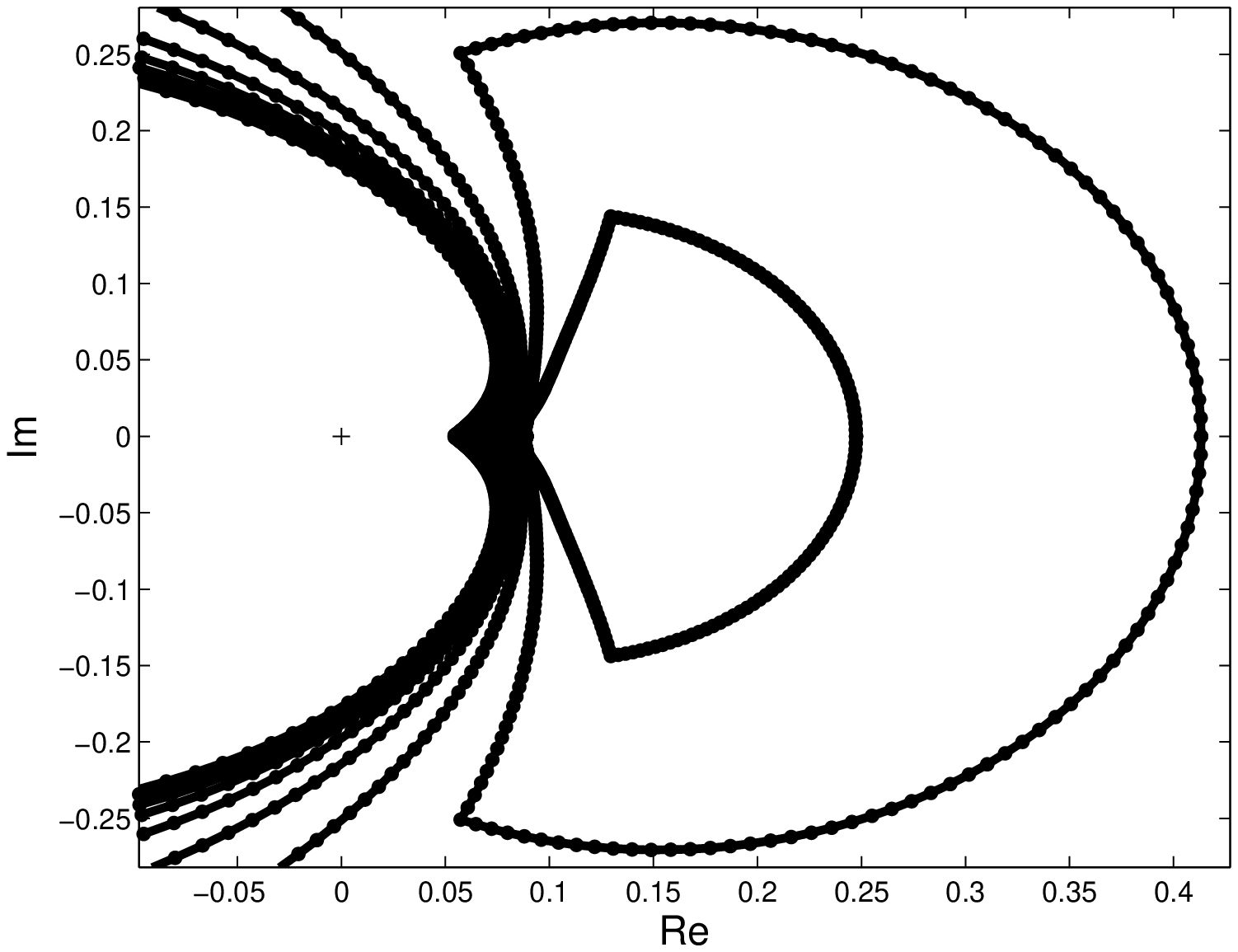}
\end{array}$
\end{center}
\caption{Renormalized Evans function output for semi-circular contour of radius $4.5$ (left) as the amplitude varies.  Parameters are $B^*_1=2$, $\mu_0=1$, $\sigma=1$, $\gamma = 5/3$, with $v_+= 10^{-1}, 10^{-1.5}, 10^{-2}, 10^{-2.5}, 10^{-3}, 10^{-3.5}, 10^{-4}, 10^{-4.5}, 10^{-5}, 10^{-5.5}, 10^{-6}$.  Note the striking concentric structure of the contours, which converge to the outer contour in the large-amplitude limit (i.e., $v_+\rightarrow 0$) and to a non-zero constant in the small-amplitude limit (i.e., $v_+\rightarrow 1$), indicating stability for all shock strengths since the winding numbers throughout are all zero.  The limiting contour given by $\hat D^0(\lambda)$ is also displayed, but is essentially identical to nearby contours.  When the image is zoomed in near the origin (right), which is marked by a crosshair, we see that the curves are well behaved and distinct from the origin.  Also clearly visible is the theoretically predicted square-root singularity at the origin of the limiting contour, as indicated by a right angle in the curve at the image of the origin on the real axis.}
\label{vplimit}
\end{figure}

In Figure \ref{vplimit}, we see the typical {\em concentric} structure as $v_+$ varies on $[0,1]$.  Note that in the strong-shock limit, the output converges to the outer contour representing the Evans function output of the limiting system.  In the small-amplitude limit, the system converges to a non-zero constant.  Since the origin is outside of the contours, one can visually verify that the winding number is zero thus implying Evans stability, even in the strong-shock limit.

In Figure \ref{BtoINfty}, we illustrate the convergence of the Evans function as $B^*_1\rightarrow \infty$.  Note that the contours converge to zero, but they are stable for all finite values of $B^*_1$.  Stability is proven analytically in Theorem \ref{largeB} by a tracking argument.  Prior to this computation, however, a significant effort was made to prove stability with energy estimates, but these efforts were in vain since the Evans function converges to zero as $B^*_1\rightarrow\infty$.

In Figure \ref{musmall}, we see the structure as $\mu_0\rightarrow 0$.  Once normalized (right), we see that the structure is essentially unchanged despite a large variation in $\mu_0$; in particular, the shock layers are stable in the $\mu_0\rightarrow 0$ limit.  This was proven analytically in Theorem \ref{sigmu}.

Finally, in Figure \ref{easy}, we see the behavior of the Evans function in the case that $r = \mu/(2\mu+\eta)\to \infty$.  This is the opposite case of that considered in \cite{FT}.  As we show in Proposition \ref{rinfty}, this case can be computed by disengaging the shooting algorithm and just taking the determinant of initializing e-bases at $\pm \infty$.  
Notice that in this limit the shock layers are also stable.

\begin{figure}[t]
\begin{center}
$\begin{array}{cc}
\includegraphics[width=7.5cm]{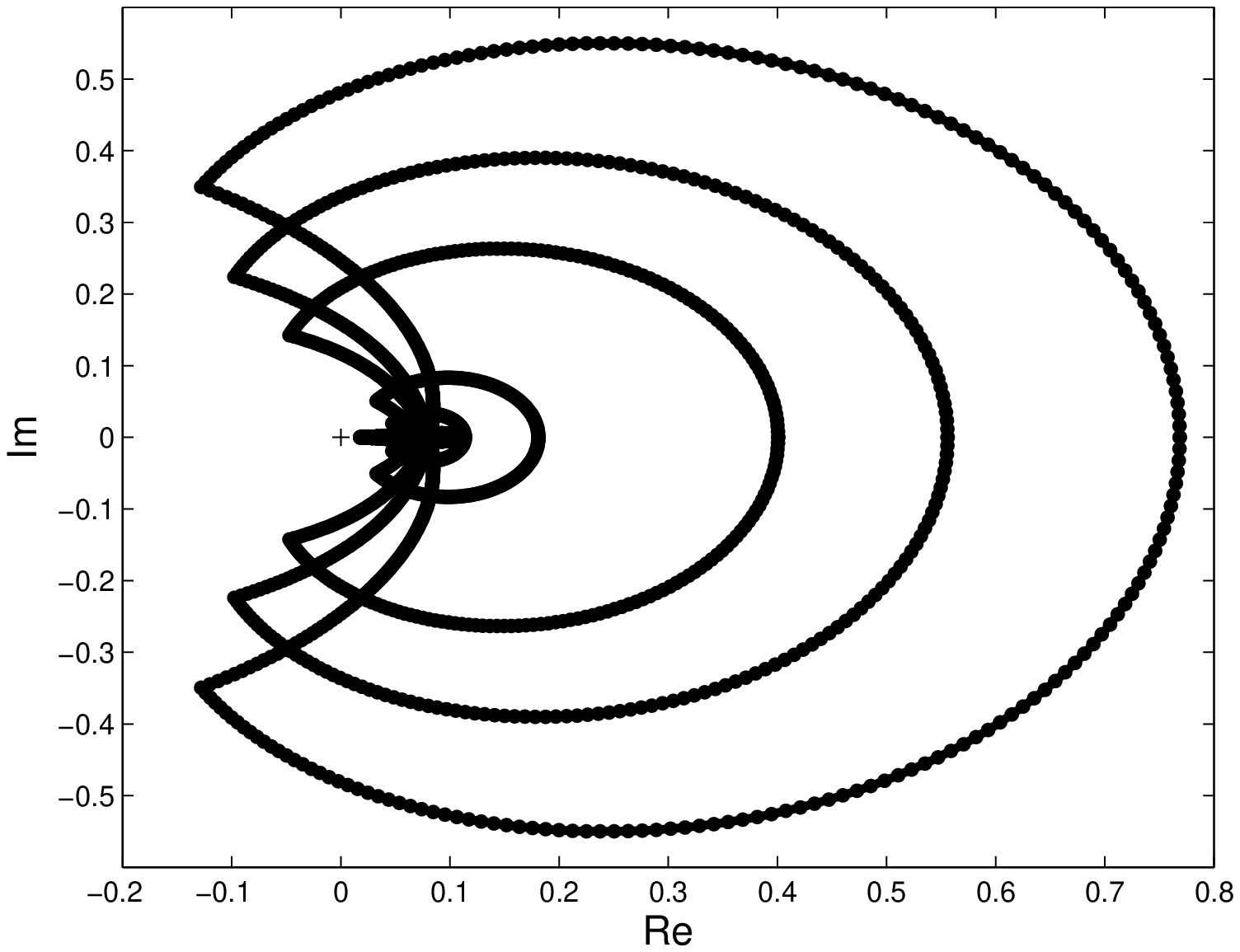} & \includegraphics[width=7.5cm]{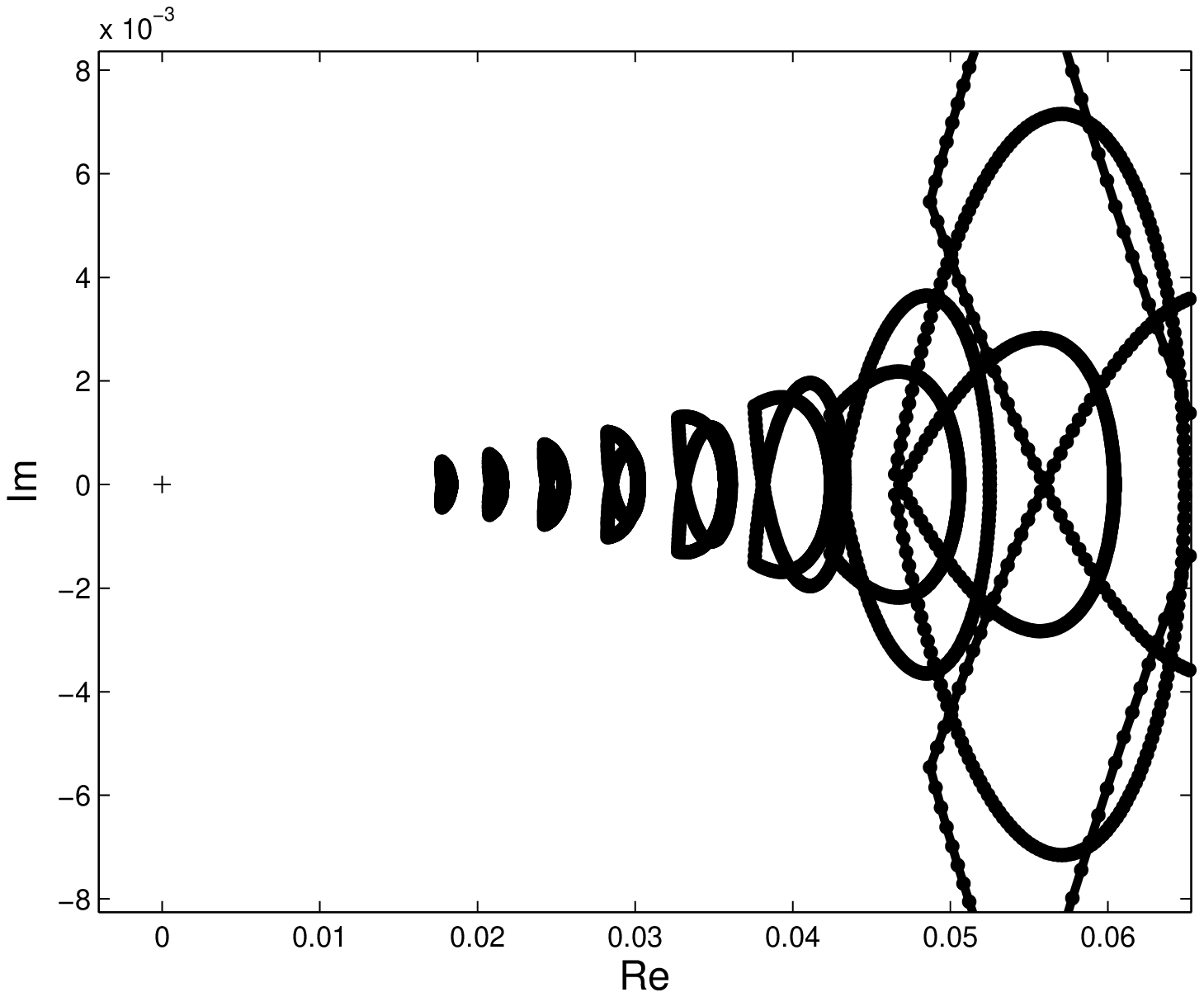}
\end{array}$
\end{center}
\caption{Evans function output for semi-circular contour of radius 5 (left) and a zoom-in of the same image near the origin (right).  Parameters are $v_+=10^{-2}$, $\mu_0=1$, $\sigma=1$, $\gamma = 5/3$, with $B^*_1= 2, 3.5, 5, 10, 15, 20, 25, 30, 35, 40$.  Note that the contours converge to zero, which is marked by a cross hair, as $B^*_1\rightarrow\infty$.}
\label{BtoINfty}
\end{figure}

\begin{figure}[t]
\begin{center}
$\begin{array}{cc}
\includegraphics[width=7.5cm]{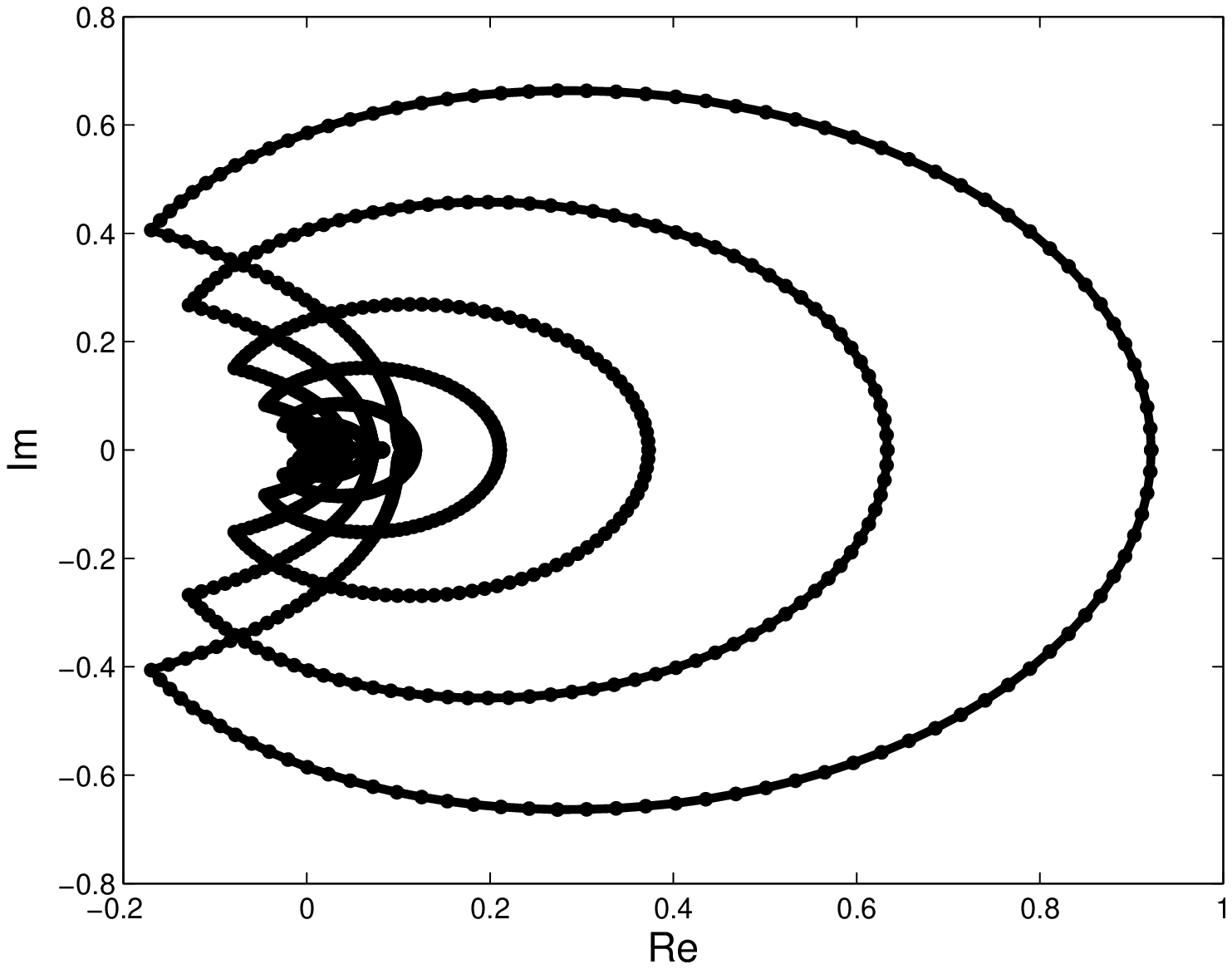} & \includegraphics[width=7.5cm]{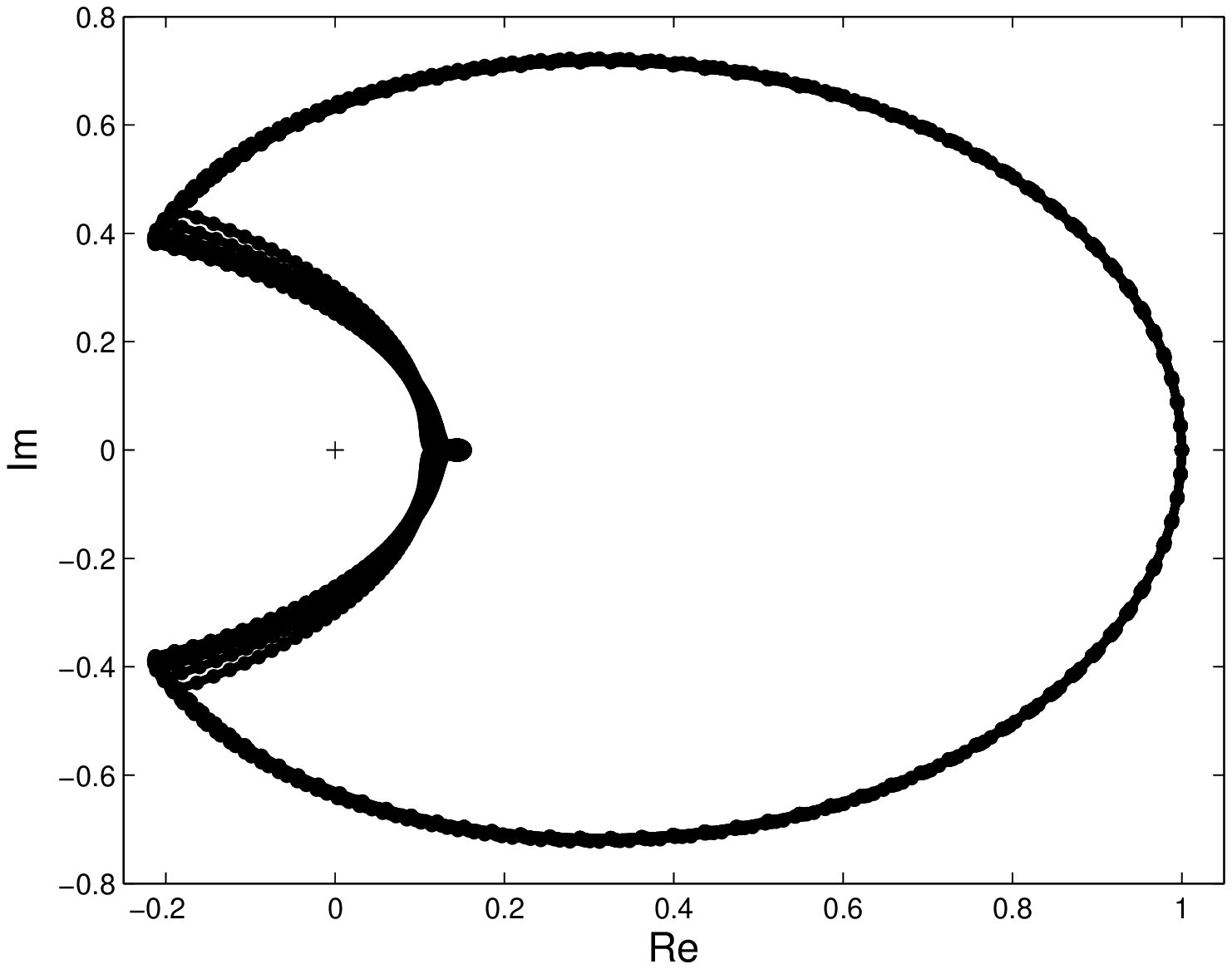}
\end{array}$
\end{center}
\caption{Evans function output for semi-circular contour of radius 5 (left) together with a renormalized version of the contours (right), where each contour is divided by its rightmost value, thus putting all contours through $z=1$ on the right side.  Although these results are typical, the parameters in this example are $B^*_1=2$, $v_+=10^{-2}$, $\sigma=1$, $\gamma = 5/3$, with $\mu_0=10^{-.5}, 10^{-1}, 10^{-1.5}, 10^{-2}, 10^{-2.5}, 10^{-3}, 10^{-3.5}, 10^{-4}, 10^{-4.5}, 10^{-5}$.  Note that the renormalized contours are nearly identical.  This provides a striking indication of stability for all values of $\mu_0$ in our range of consideration, and in particular for $\mu_0\rightarrow 0$.}
\label{musmall}
\end{figure}

\begin{figure}[t]
\begin{center}
$\begin{array}{cc}
\includegraphics[width=7.5cm]{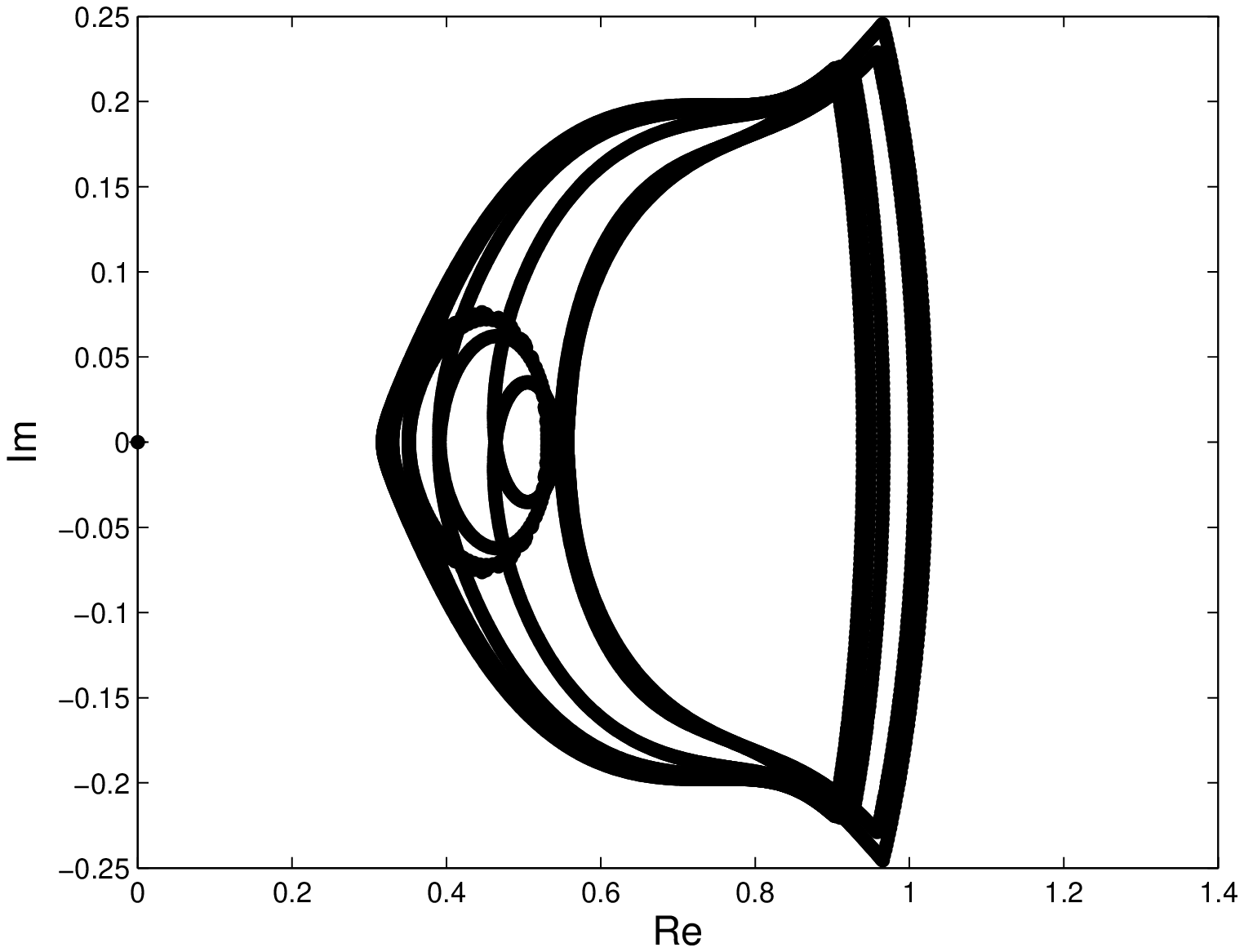} & \includegraphics[width=7.5cm]{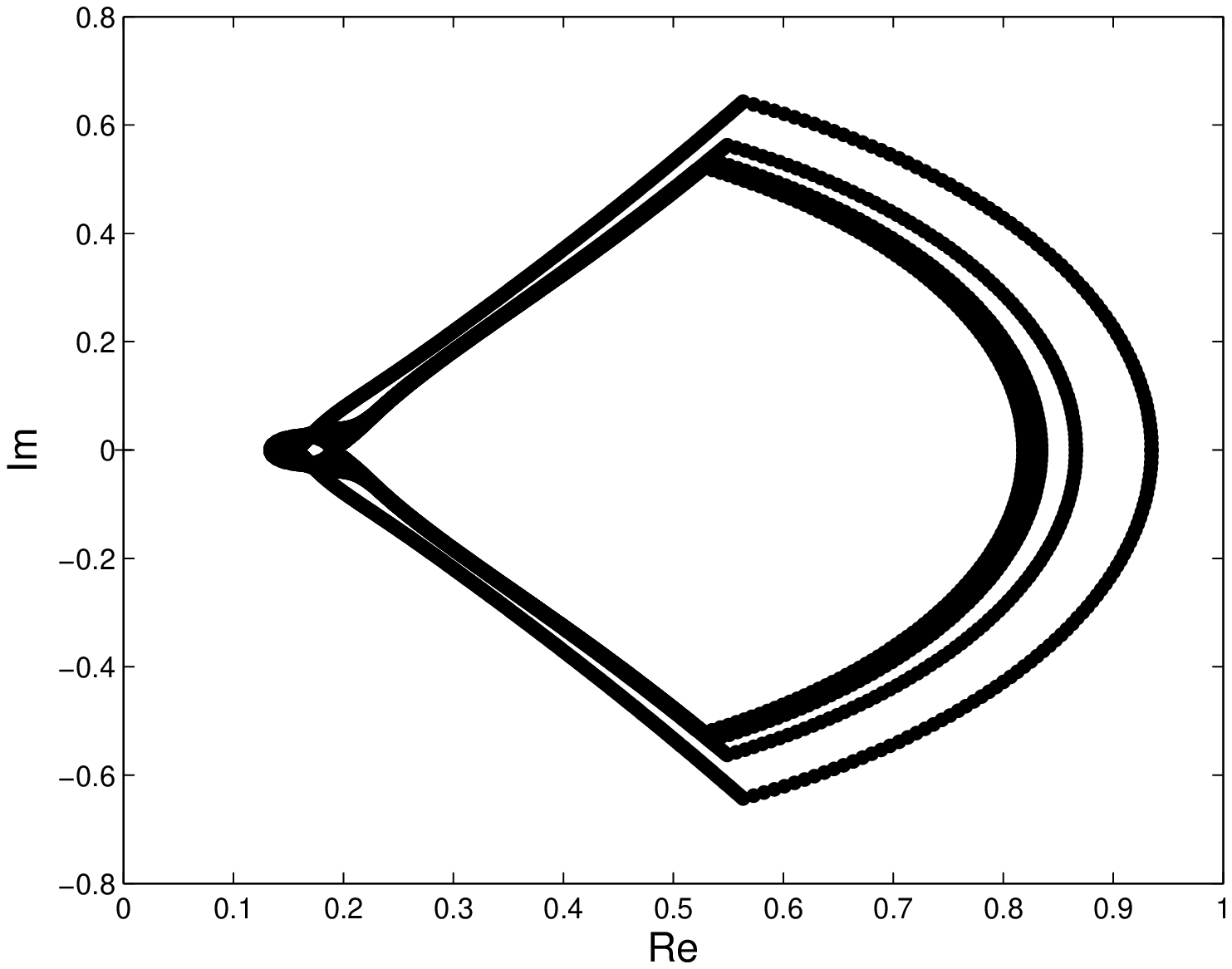}
\end{array}$
\end{center}
\caption{Renormalized Evans function in the $r=\infty$ case.  Parameters are $v_+=10^{-1}$, $10^{-2}$, $10^{-3}$, $10^{-4}$, $10^{-5}$, $10^{-6}$, $\mu_0=1$, $\sigma=1$, $\gamma=5/3$.  We also have a semi-circular radius of $4.5$ with $B_1^*=2$ (left), and a semi-circular radius of $1$ with $B_1^*=0.5$ (right).
}
\label{easy}
\end{figure}


\subsection{Discussion and open problems}\label{sec:disc}
Our numerical and analytical investigations suggest strongly
(and in some cases rigorously prove) reduced Evans stability
of parallel ideal isentropic MHD shock layers, independent of amplitude,
viscosity and other transport parameters, or magnetic field,
for gas constant $\gamma \in [1,3]$,
indicating that they are stable whenever the associated 
gas-dynamical shock layer is stable.
Together with previous investigations of \cite{HLZ}
indicating unconditional stability of isentropic gas-dynamical shock layers
for $\gamma\in [1,3]$, this suggests unconditional stability
of parallel isentropic MHD shocks for gas constant $\gamma\in [1,3]$,
the first such comprehensive result for shock layers in MHD.

It is remarkable that, despite the complexity of solution structure
and shock types occurring as magnetic field and other parameters vary,
we are able to carry out a uniform numerical Evans function analysis
across almost (see Remark \ref{corners})
the entire parameter range: a testimony to the power of the
Evans function formulation.
Interesting aspects of the present analysis beyond what has been done
in the study of gas dynamical shocks in \cite{HLZ,HLyZ1} are
the presence of branch singularities on certain parameter boundaries,
necessitating renormalization of the Evans function to remove blow-up
singularities, and the essential use of winding number
computations on Riemann surfaces in order to establish stability
in the large-amplitude limit.
The latter possibility was suggested in \cite{GZ}
(see Remark 3, Section 2.1),  but to our knowledge has not up to now
been carried out.

We note that Freist\"uhler and Trakhinin \cite{FT} have previously
established spectral stability of parallel viscous MHD shocks 
using energy estimates in the regime 
$$
r:=\mu/(2\mu+ \eta) \ll 1,
$$
whenever $B_1^*< 2\sqrt{\mu_0 v_-}$
(translating their results to our setting $s=-1$), which includes
all $1$- and intermediate-shocks, and some slow shocks 
($B_1^*>\sqrt{\mu_0 v_-}$).
Recall that we have here followed the standard physical prescription
$\eta=-2\mu/3$, so that
 $
\mu/(2\mu+ \eta)=3/4,
$
outside the regime studied in \cite{FT}.
Thus, the two analyses are complementary.  It would be an interesting mathematical question to investigate stability for general ratios $\mu/(2\mu+ \eta)$.  See Section \ref{ratio} for further discussion of this issue.  Here we study only the limit $r\to \infty$ complementary to that studied by \cite{FT}, the case $r=3/4$ suggested by nonmagnetic gas dynamics, and the remaining cases in the $r\to 0$ limit left open in \cite{FT}.  Other $r$-values may be studied numerically, but were not checked.

Stability of general (not necessarily parallel) fast shocks in the 
small magnetic field limit has been established in \cite{GMWZ6}
by convergence of the Evans function to the gas-dynamical limit,
assuming that the limiting gas-dynamical shock is stable, as has
been numerically verified for ideal gas dynamics in \cite{HLZ,HLyZ1,HLyZ2}.
%
Stability of more general, non-gas-dynamical shocks with large
magnetic field, is a very interesting open question.
In particular, as noted in \cite{TZ}, one-dimensional instability,
by stability index considerations, would for an ideal gas equation
of state imply the interesting phenomenon of
Hopf bifurcation to time-periodic, or ``galloping'' behavior
at the transition to instability.
For analyses of the related inviscid stability problem, 
see, e.g., \cite{T,BT,MeZ2} and references therein.

Another interesting direction for further investigation would be
a corresponding comprehensive study of multi-dimensional stability of 
parallel MHD shock layers, as carried out for gas-dynamical
shocks in \cite{HLyZ2}. 
As pointed out in \cite{FT}, instability results
of \cite{BT,T} for the corresponding inviscid problem imply 
that parallel shock layers become multi-dimensionally unstable for
large enough magnetic field,
by the general result \cite{ZS,Z2} that inviscid stability
is necessary for viscous stability, so that in multi-dimensions
instability definitely occurs.
The question in this case is whether viscous effects can hasten
the onset of instability, that is, whether
viscous instability can occur in the presence of inviscid stability.

\section{The Evans function and its properties}\label{evans}

We begin by constructing carefully the Evans function associated
with reduced system \eqref{eval}, and recalling its basic properties
for our later analysis.

\subsection{The Evans system}\label{s:evanssys}

The reduced eigenvalue equations \eqref{eval}
may be written as a first-order system
\begin{equation}\label{1sys}
    \begin{pmatrix}
        w\\ \mu w'\\ \alpha \\ \frac{\alpha'}{\sigma\mu_0 \hat v}
    \end{pmatrix}'=
    \begin{pmatrix}
        0 & 1/\mu & 0 &0\\
        \lambda \hat v& \hat v/\mu & 0& -\sigma B_1^*\hat v \\
        0 & 0 &0 & \sigma \mu_0 \hat v\\
        0 & -B_1^*\hat v/\mu & \lambda \hat v & 
        \sigma \mu_0 \hat v^2
    \end{pmatrix}
    \begin{pmatrix}
        w\\ \mu w'\\ \alpha \\ \frac{\alpha'}{\sigma\mu_0 \hat v}
    \end{pmatrix},
\end{equation}
or
\be
\label{stdsys}
W'=A(x,\lambda) W,
\ee
indexed by the three parameters $B_1^*$, $\sigma$, and
$\tilde \mu_0:= \sigma \mu_0$. Recall that we have already fixed
$ \mu=1$ and $(2\mu+\eta)= 4/3$.

\subsection{Limiting subspaces}\label{shocktype}
Denote by
\begin{equation}\label{Apm}
A_\pm(\lambda) :=\lim_{x\to \pm \infty} A(x,\lambda)=
\begin{pmatrix}
        0 & 1/\mu & 0 &0\\
        \lambda v_\pm& v_\pm /\mu & 0& -\sigma B_1^* v_\pm \\
        0 & 0 &0 & \sigma \mu_0 v_\pm\\
        0 & -B_1^* v_\pm/\mu & \lambda v_\pm & 
\sigma \mu_0 v_\pm^2 
    \end{pmatrix}
\end{equation}
the limiting coefficient matrices associated with \eqref{1sys}-\eqref{stdsys}.

\bl\label{limlem1}
For $\Re \lambda \ge 0$, $\lambda\ne 0$, and $1\ge v_+>0$,
each of $A_\pm$ has two eigenvalues with strictly
positive real part and two eigenvalues with strictly
negative real part, hence their stable and unstable
subspaces $S_\pm$ and $U_\pm$ vary smoothly in all
parameters and analytically in $\lambda$.
For $v_+>0$ they extend continuously to $\lambda=0$,
and analytically everywhere except at
$B_1^*= \sqrt{\mu_0 v_\pm}$, 
where they depend smoothly on $\sqrt{(1-B_1^*/\mu_0v_\pm)^2 + 4\lambda c}$,
$
c:=  \frac{1}{2}\Big(\frac{\mu}{v} + \frac{1}{\sigma \mu_0 v^2}\Big)_\pm.
$
\el

\begin{proof}
By standard hyperbolic--parabolic theory 
(e.g., Lemma 2.21, \cite{Z2}), $A_\pm$ 
have no pure imaginary eigenvalues for $\Re \lambda\ge 0$,
$\lambda \ne 0$, for any $v_+>0$ and parameter values
$\sigma$, $\tilde \mu_0$, $B_1^*$, whence the numbers of
stable (negative real part) and unstable (positive real part)
eigenvalues are constant on this set.  By homotopy taking
$\lambda$ to positive real infinity, we find readily that
there must be two of each.

Alternatively, we may see this directly by looking at the corresponding
second-order symmetrizable hyperbolic--parabolic system
$$
\lambda \bp w \\ \alpha \ep
+
\bp
1& -B_1^*/\mu_0 v_\pm\\
-B_1^* & 1
\ep 
\bp w \\ \alpha \ep_x=
\bp
\mu/v_\pm & 0\\
0 & 1/\sigma \mu_0 v_\pm^2
\ep 
\bp w \\ \alpha \ep_{xx},
$$
and applying the standard theory here, specifically,
noting that eigenvalues consist of solutions $\mu$ of
\be\label{evalcrit}
\lambda \in 
\sigma \Big(
-\mu
\bp
1& -B_1^*/\mu_0 v_\pm\\
-B_1^* & 1
\ep 
+ \mu^2
\bp
\mu/v_\pm & 0\\
0 & 1/\sigma \mu_0 v_\pm^2
\ep 
\Big),
\ee
so that $\mu=ik$ by a straightforward energy
estimate yields $\Re \lambda \le -\theta k^2$
for $\theta>0$.

Applying to reduced system \eqref{evalcrit}
Lemma 6.1 \cite{MaZ3} or Proposition 2.1, \cite{ZH},
we find further that, whenever the convection matrices
\be\label{betapm}
\beta_\pm:=
\bp
1& -B_1^*/\mu_0 v_\pm\\
-B_1^* & 1
\ep
\ee
are noncharacteristic in the sense that their
eigenvalues $\alpha= 1\pm \frac{B_1^*}{\sqrt{\mu_0 v_\pm}}$
are nonzero, these subspaces extend analytically to $\lambda=0$.

Finally, we consider the degenerate case that $B_1^*=\sqrt{\mu_0v_\pm}$
and the convection matrix $\beta$ is characteristic.
Considering \eqref{evalcrit} as determining $\lambda/\mu$ 
as a function of $\mu$ for $\mu$ small, we obtain, diagonalizing
$B$ and applying standard matrix perturbation theory \cite{Kato}
that in the nonzero eigendirection $r_j$ of $\beta$, associated
with eigenvalue $\beta_j\ne 0$,  $\lambda/\mu \sim \beta_j$,
and, inverting, we find that $\mu_j(\lambda)$ extends analytically
to $\lambda=0$.
In the zero eigendirection, on the other hand, associated with
left and right eigenvectors 
$l=(1/2, 1/2B_1^*)^T$ 
and $r=(B, 1)^T$, 
we find that 
$$
\lambda \sim -\mu (1-B_1^*/\sqrt{\mu_0 v_\pm}) + c\mu^2+ \dots, 
$$
where
$$
c=(l^T\beta r)_\pm=
\frac{1}{2}\Big(\frac{\mu}{v} + \frac{1}{\sigma \mu_0 v^2}\Big)_\pm
\ne 0,
$$
leading after inversion to the claimed square-root singularity.
Likewise, the stable eigendirections of $A_\pm$ associated with $\mu$ vary continuously
with $\lambda$, converging for $B_1^*=\sqrt{\mu_0 v_\pm}$ and $\lambda=0$
to 
$
(w , \alpha, w', \alpha')^T= \bp r\\0\ep,
$
where $r$ is the zero eigendirection of $\beta$.
This accounts for three eigenvalues $\mu$ lying near zero,
bifurcating from the three-dimensional kernel of $A_\pm$
near a degenerate, characteristic, value of $B_1^*$.
The fourth eigenvalue is far from zero and so varies analytically
in all parameters about $\lambda=0$.
\end{proof}

\br\label{contrmk}
Remarkably, even though the shock changes type upon passage through
the points $B_1^*=\sqrt{\mu_0 v_\pm}$, the stable and unstable subspaces of
$A_\pm$ vary {\rm continuously}, with stable and unstable eigendirections
coalescing in the characteristic mode.
\er

\subsection{Limiting eigenbases and Kato's ODE}

Denote by $\Pi_+$ and $\Pi_-$ the eigenprojections of
$A_+$ onto its stable subspace and $A_-$ onto its unstable
subspace, with $A_\pm$ defined as in \eqref{Apm}.
By Lemma \ref{limlem1}, these are analytic in $\lambda$
for $\Re \lambda \ge 0$, $v_+>0$, except for square-root
singularities at $\lambda=0$ for $B_1^*=\sqrt{\mu_0v_\pm}$.
Introduce the complex ODE \cite{Kato}
\be\label{Kato}
R'=\Pi'R, \quad R(\lambda_0)=R_0,
\ee
where $'$ denotes $d/d\lambda$, $\lambda_0$ is fixed
with $\Re \lambda_0>0$, $\Pi=\Pi_\pm$, 
and $R$ is a $4\times 2$ complex matrix.
By a partition of unity argument \cite{Kato}, there exists a
choice of initializing matrices $R_0$ that is smooth in the
suppressed parameters $v_+$, $B_1^*$, $\sigma$, and $\mu_0$,
is full rank, and satisfies $\Pi(\lambda_0) R_0=R_0$; that is,
its columns are a basis for the stable (resp. unstable) 
subspace of $A_+$ (resp. $A_-$).

\bl[\cite{Kato,Z4}]\label{katolem}
There exists a global solution
$R$ of \eqref{Kato} on $\{\Re \lambda \ge 0\}$, 
analytic in $\lambda$ and smooth in parameters $v_+>0$, $B_1^* \ge 0$,
$\sigma>0$, and $\mu_0>0$
except at the singular values $\lambda=0$, $B_1^*=\sqrt{\mu_0 v_\pm}$, 
such that (i) $\rank R \equiv \rank R^0$,
(ii) $\Pi R\equiv R$, and (iii) $\Pi R'\equiv 0$.
\el

\begin{proof}
As a linear ODE with analytic coefficients, \eqref{Kato}
possesses an analytic solution in a neighborhood of $\lambda_0$,
that may be extended globally along any curve, whence, by
the principle of analytic continuation, it possesses a global
analytic solution on any simply connected domain containing $\lambda_0$
\cite{Kato}.
Property (i) follows likewise
by the fact that $R$ satisfies a linear ODE.
Differentiating the identity $\Pi^2=\Pi$ following \cite{Kato} 
yields $\Pi\Pi'+\Pi'\Pi=\Pi'$, whence, multiplying on the right by $\Pi$,
we find the key property
\begin{equation}\label{Pprop}
\Pi\Pi'\Pi=0.
\end{equation}
From \eqref{Pprop}, we obtain
$$
\begin{aligned}
(\Pi R-R)'&=(\Pi'R+\Pi R'-R')
= \Pi'R+(\Pi-I) \Pi'R
= \Pi \Pi'R,\\
\end{aligned}
$$
which, by $\Pi\Pi'\Pi=0$ and $\Pi^2=\Pi$ gives
$$
\begin{aligned}
(\Pi R-R)'&= -\Pi\Pi'(\Pi R-R), \quad (\Pi R-R)(\lambda_0)=0,
\end{aligned}
$$
from which (ii) follows by uniqueness of solutions
of linear ODE.
Expanding $\Pi R'=\Pi\Pi'R$ and using
$\Pi R=R$ and $\Pi\Pi'\Pi=0$, we obtain $\Pi R'=\Pi\Pi'\Pi R=0$, verifying (iii).
\end{proof}

\br
Property (iii) indicates that the Kato basis is an
optimal choice in the sense that it involves minimal variation
in $R$.
It is useful also for computing the Kato basis in different ways
\cite{HSZ,BDG}; see Appendix \ref{branch}.
\er

\subsection{Characteristic values: the regularized Kato basis}
We next investigate the behavior of the Kato basis near $\lambda=0$ 
and the degenerate points $B_1^*=\sqrt{\mu_0v_\pm}$ at 
which the reduced convection matrix $\beta_\pm$ 
of \eqref{betapm} becomes characteristic
in a single eigendirection.

\begin{example}\label{sqrteg}
A model for this situation is the eigenvalue equation for a
scalar convected heat equation
$\lambda u + \eta u'=u''$ with
convection coefficient $\eta$ passing through zero.
The coefficient matrix for the associated first-order system 
is 
\be\label{egA}
A:=\bp 0 & 1 \\ \lambda & \eta\ep.
\ee
As computed in Appendix \ref{branch},
the stable eigenvector of $A$ determined by Kato's ODE \eqref{Kato} is
\be\label{charkato}
R(\eta, \lambda):=
\frac{ (\eta^2/4+ 1)^{1/4}} { (\eta^2/4+ \lambda)^{1/4}}
\Big(1, -\eta/2 - \sqrt{\eta^2/4 + \lambda}\, \Big)^T,
\ee
which, apart from the divergent factor 
$\frac{ (\eta^2/4+ 1)^{1/4}} { (\eta^2/4+ \lambda)^{1/4}}$,
is a smooth function of $ \sqrt{\eta^2/4 + \lambda}$.
\end{example}

The computation of Example \eqref{sqrteg} indicates that the
Kato basis blows up at $\lambda=0$ as
$\big( (1-B_1^*/\sqrt{\mu_0 v_\pm})^2+4\lambda\big)^{-1/4}$ as
$B_1^*$ crosses characteristic points $\sqrt{\mu_0v_\pm}$
across which the shock changes type, hence does not give a choice
that is continuous across the entire range of shock profiles.
However, the same example shows that there is a different choice
$(1, -\eta/2 - \sqrt{\eta^2/4 + \lambda}\, )^T$
that {\it is} continuous, possessing only a square-root singularity.
We can effectively exchange one for another, by rescaling the
Kato basis as we now describe.

Following \cite{AGJ,GZ}, associate with bases
$R_\pm=(R_1,R_2)^\pm$ the wedge (i.e., exterior algebraic) product 
$\cR_\pm:=(R_1\wedge R_2)^\pm$.

\bl\label{rkatolem}
The ``regularized Kato products''
\be\label{rkato+}
\tilde \cR_+:=
\frac{ \Big( (1-B_1^*/\sqrt{\mu_0 v_+})^2 + 4\lambda 
 (\mu/2v_+ + 1/2\sigma \mu_0 v_+^2)
\Big)^{1/4}} 
{ \Big( (1-B_1^*/\sqrt{\mu_0 v_+})^2 + 4
 (\mu/2v_+ + 1/2\sigma \mu_0 v_+^2)
\Big)^{1/4}} 
(R_1^+\wedge R_2^+)
\ee
and
\be\label{rkato-}
\tilde \cR_-:=
\frac{ \Big( (1-B_1^*/\sqrt{\mu_0 })^2 + 4\lambda 
(\mu/2 + 1/2\sigma \mu_0 )
\Big)^{1/4}} 
{ \Big( (1-B_1^*/\sqrt{\mu_0 })^2 + 4 
(\mu/2 + 1/2\sigma \mu_0 )
\Big)^{1/4}} 
(R_1^+\wedge R_2^+)
\ee
are analytic in $\lambda$ and smooth in remaining parameters
on all of $\lambda \ge 0$, $v_+>0$, $B_1^*\ge 0$,
$\sigma>0$, $\mu_0>0$ except the points $\lambda=0$, $B_1^*=\sqrt{\mu_0v_\pm}$,
where they are continuous with a square-root singularity,
depending smoothly on $\sqrt{ (1-B_1^*/\sqrt{\mu_0 v_\pm})^2+4\lambda}$.
Moreover, they are bounded from zero (full rank) on the entire
parameter range.
\el

\begin{proof}
A computation like that of Example \ref{sqrteg} 
applied to system \eqref{Apm},
replacing $\eta$ with the characteristic speed 
$1-B_1^*/\sqrt{\mu_0 v_\pm}$ and introducing a diffusion coefficient
$c_\pm=  \frac{1}{2}\Big(\frac{\mu}{v} + \frac{1}{\sigma \mu_0 v^2}\Big)_\pm$,
i.e., considering 
\be\label{approxeq}
\lambda u+ \eta u'=cu'',
\ee
shows that, for an appropriate choice of initializing
basis $R_0$ in \eqref{Kato}, 
there is blowup as $(\lambda, B_1^*)\to (0, \sqrt{\mu_0v_\pm})$
at rate
$$
\Big((1-B_1^*/\sqrt{\mu_0 v_\pm})^2+4\lambda c_\pm)\Big)^{-1/4}
$$
in a basis vector involving the characteristic mode,
while the second basis vector remains bounded and analytic,
whence the result follows.
For the derivation of approximate equation \eqref{approxeq},
see the proof of Lemma \ref{limlem1}.
\end{proof}

\br\label{val}
A review of the argument shows that estimate \eqref{rkato+}
derived for fixed $v_+$ remains valid so long
as $|\lambda |\le C|\mu| \ll v_+^2$.
Different asymptotics hold for $v_+\le C\sqrt{|\lambda|}$;
see Section \ref{lambig}.
\er

\br\label{blowup}
Lemma \ref{rkatolem} (by uniform full rank) 
includes the information that the unregularized
Kato bases blow up at rate $\lambda^{-1/4}$ at $B_1^*=\sqrt{\mu_0 v_\pm}$.
\er

\subsection{Conjugation to constant-coefficients}

We now recall the {\it conjugation lemma} of \cite{MeZ1}.
Consider a general first-order system
\be \label{gen1}
W'=A(x,\lambda,p)W
\ee
with asymptotic limits $A_\pm$ as $x\to \pm \infty$,
where $p\in \RR^m$ denotes up-to-now-supressed model parameters.
\bl [\cite{MeZ1,PZ}]\label{conjlem}
Suppose for fixed $\theta>0$ and $C>0$ that 
\be\label{udecay}
|A-A_\pm|(x,\lambda,p)\le Ce^{-\theta |x|}
\ee
for $x\gtrless 0$ uniformly for $(\lambda,p)$ in a neighborhood of 
$(\lambda_0,p_0)$ and that $A$ varies analytically in $\lambda$ 
and smoothly (resp. continuously) in $p$ 
as a function into $L^\infty(x)$.
Then, there exist in a neighborhood of $(\lambda_0,p_0)$
invertible linear transformations $P_+(x,\lambda,p)=I+\Theta_+(x,\lambda,p)$ 
and $P_-(x,\lambda,p) =I+\Theta_-(x,\lambda,p)$ defined
on $x\ge 0$ and $x\le 0$, respectively,
analytic in $\lambda$ and smooth (resp. continuous) in $p$ 
as functions into $L^\infty [0,\pm\infty)$, such that
\begin{equation}
\label{Pdecay} 
| \Theta_\pm |\le C_1 e^{-\bar \theta |x|}
\quad
\text{\rm for } x\gtrless 0,
\end{equation}
for any $0<\btheta<\theta$, some $C_1=C_1(\bar \theta, \theta)>0$,
and the change of coordinates $W=:P_\pm Z$ reduces \eqref{gen1} to 
\begin{equation}
\label{glimit}
Z'=A_\pm Z 
\quad
\text{\rm for } x\gtrless 0.
\end{equation}
\el

\begin{proof}
The conjugators $P_\pm$ are constructed by a
fixed point argument \cite{MeZ1}
as the solution of an integral equation
corresponding to the homological equation
\be\label{homolog}
P'=AP-A_\pm P.
\ee
The exponential decay \eqref{udecay} is needed to 
make the integral equation contractive in $L^\infty[M,+\infty)$ 
for $M$  sufficiently large.  Continuity of $P_\pm$ with respect 
to $p$ (resp. analyticity with respect to $\lambda$) then follow
by continuous (resp. analytic) dependence on parameters of 
fixed point solutions.
Here, we are using also the fact that \eqref{udecay} plus continuity
of $A$ from $p\to L^\infty$ together imply continuity of
$e^{\tilde \theta |x|}(A-A_\pm)$ from $p$ into $L^\infty[0,\pm\infty)$
for any $0<\tilde \theta < \theta$, in order to obtain the needed
continuity from $p\to L^\infty$ of the fixed point mapping.
See also \cite{PZ,GMWZ5}.
\end{proof}

\br\label{specialform}
In the special case that $A$ is block-diagonal or
-triangular, the conjugators $P_\pm$ may evidently be taken block-diagonal
or triangular as well, by carrying out the same fixed-point argument
on the invariant subspace of \eqref{homolog} consisting
of matrices with this special form.
This can be of use in problems with multiple scales; 
see, for example, the proof in Section \ref{further} of
Theorem \ref{sigmu} ($\sigma\to 0$).
\er

\subsection{Construction of the Evans function}\label{const}

\begin{definition} [\cite{MaZ3,Z2,Z3}] \label{evansdef}
The Evans function is defined on
\eqref{Kato} on $\Re \lambda \ge 0$, $v_+>0$, $B_1^* \ge 0$,
$\sigma>0$, $\mu_0>0$ as
\be\label{eq:evans}
\begin{aligned}
D(\lambda,p)&:=
\det( P^+R_1^+,P^+R_2^+, P^-R_1^-, P^-R_2^-)|_{x=0}\\
&=
 \langle P^+R_1^+ \wedge P^+R_2^+ \wedge P^-R_1^- 
\wedge P^-R_2^- |_{x=0} \rangle,
\end{aligned}
\ee
where $\langle \cdot \rangle$ of a full wedge product denotes its
coordinatization in the standard (single-element) basis
$e_1\wedge e_2\wedge e_3 \wedge e_4$, where $e_j$ are the standard Euclidean
basis elements in $\CC^4$.
\end{definition}

\begin{definition} \label{revansdef}
The regularized Evans function is defined as
\be \label{eq:revans}
\begin{aligned}
\tilde D(\lambda,p)&:=
\frac{ \Big( (1-B_1^*/\sqrt{\mu_0 })^2 + 4\lambda
 (\mu/2 + 1/2\sigma \mu_0 )
\Big)^{1/4}}
{ \Big( (1-B_1^*/\sqrt{\mu_0 })^2 + 4
 (\mu/2 + 1/2\sigma \mu_0 )
\Big)^{1/4}}
\\
&\quad \times
\frac{ \Big( (1-B_1^*/\sqrt{\mu_0 v_+})^2 + 4\lambda
 (\mu/2v_+ + 1/2\sigma \mu_0 v_+^2)
\Big)^{1/4}} 
 {\Big( (1-B_1^*/\sqrt{\mu_0 v_+})^2 + 4
 (\mu/2v_+ + 1/2\sigma \mu_0 v_+^2)
\Big)^{1/4}} 
D(\lambda,p)\\
&=
 \langle \cP_+\tilde R_+ \wedge \cP_-\tilde R_- |_{x=0} \rangle,
\end{aligned}
\ee
where $\cP_\pm (R_1 \wedge R_2):= P_\pm R_1 \wedge P_\pm R_2$
denotes the ``lifting'' to wedge product space
of conjugator $P_\pm$. 
\end{definition}

\bpr\label{evprops}
The Evans function $D$ is analytic in $\lambda$ and smooth 
in remaining parameters
on all of $\lambda \ge 0$, $v_+>0$, $B_1^*\ge 0$,
$\sigma>0$, $\mu_0>0$ except the points $\lambda=0$, $B_1^*=\sqrt{\mu_0v_\pm}$,
where it blows up as 
$$
\Big((1-B_1^*/\sqrt{\mu_0 v_\pm})^2+4\lambda
 (\mu/2v_\pm + 1/2\sigma \mu_0 v_\pm^2)
\Big)^{-1/4}.
$$
The regularized Evans function $\tilde D$
is analytic in $\lambda$ and smooth in remaining parameters on the same
domain, and continuous with a square-root singularity
at $(\lambda,B_1^*)=(0,\sqrt{\mu_0v_\pm})$,
depending smoothly on $\sqrt{ (1-B_1^*/\sqrt{\mu_0 v_\pm})^2+4\lambda}$.
\epr

\begin{proof}
Local existence/regularity is immediate, 
by Lemmas \ref{katolem}, \ref{rkatolem}, and \ref{conjlem},
Proposition \ref{profdecay}, and
Definitions \ref{evansdef}, \ref{revansdef}.
Global existence/regularity then follow \cite{MaZ3,PZ,Z2,Z3} by
the observation that the Evans function is independent of the
choice of conjugators $P_\pm$ (in general nonunique) on the region
where $A_\pm$ are hyperbolic (have no center subspace), in this
case $\{\Re \lambda\ge 0\}\setminus\{0\}$.\footnote{
In Evans function terminology, the ``region of consistent splitting'' 
\cite{AGJ,GZ,Z2,Z3}.
}
\end{proof}

\br
Evidently, for $B_1^*\ne \sqrt{\mu_0 v_\pm}$, Evans stability, defined
as nonvanishing of $D$ on $\Re \lambda \ge 0$ is 
equivalent to nonvanishing of the regularized Evans function
$\tilde D$ on $\Re \lambda \ge 0$.
On the other hand, $\tilde D$ is continuous throughout the physical
parameter range, making possible a numerical verification of nonvanishing,
even up to the characteristic points $B_1^*=\sqrt{\mu_0v_\pm}$.
\er

\br\label{altreg}
An alternative, simpler and more general regularization of the Evans function is 
\be
\label{eq:altreg}
\hat D(\lambda,p):=\frac{D(\lambda,p)}{(|\cP_+\cR_+||\cP_-\cR_-|)|_{x=0}}
\;,
\ee
where $|\cP_+\cR_+|$ and $|\cP_-\cR_-|$ denote norms of $(\cP\cR)_\pm$ in the 
standard basis $e_i\wedge e_j$, $i\ne j$.
Though not analytic, it is still $C^\infty$ wherever $D$ is analytic,
and its zeros agree in location and multiplicity with those of $D$,
and is continuous wherever the stable (unstable) subspaces of $A_+$
($A_-$) are continuous.
Indeed, it is somewhat more faithful than the usual Evans function
to the original idea \cite{AGJ} of a quantity measuring
the angle between subspaces.
The disadvantage of this regularization is that it eliminates 
structure (analyticity, asymptotic behavior) that has proved quite
useful both in verifying code by benchmarks, and in interpreting
behavior/trends \cite{HLZ,CHNZ,HLyZ1,HLyZ2}.
\er

\bpr
On $\{\Re \lambda\ge 0\}\setminus\{0\}$, the zeros of $D$ (resp. $\tilde D$)
agree in location and multiplicity with eigenvalues of $L$.
\epr

\begin{proof}
Agreement in location-- the part that concerns us here--
is an immediate consequence of the construction.
Agreement in multiplicity was established in \cite{GJ1,GJ2}.
For an alternative argument, see \cite{ZH,MaZ3}.
\end{proof}

\section{The strong shock limit}\label{stronglim}

We now investigate behavior of the Evans function in the strong
shock limit $v_+\to 0$.
By Lemma \ref{conjlem}, Proposition \ref{profdecay},
and Corollary \ref{limv}, this reduces to the problem of
finding the limiting Kato basis $R_+$ at $+\infty$ as $v_+\to 0$.
That is, this is a ``regular perturbation'' problem in 
the sense of \cite{PZ,HLyZ1}, and not a singular perturbation
as in the much more difficult treatment of the 
gas-dynamical part $(v,u_1)$ done in \cite{HLZ}.
On the other hand, we face new difficulties associated
with vanishing of the limiting Evans function at $\lambda=0$
and branch points in both limiting and finite Kato flows,
which require additional stability index and Riemann surface computations
to complete the analysis.

\subsection{Limiting eigenbasis at $+\infty$ as $v_+\to 0$, $|\lambda|\ge \theta>0$}\label{lambig}

Fixing $\mu=1$ without loss of generality, we
examine the limit of the stable subspace as $v_+\to 0$ of
\begin{equation}\label{Aplus}
A_+(\lambda) =
\begin{pmatrix}
        0 & 1 & 0 &0\\
        \lambda v_+& v_+  & 0& -\sigma B_1^* v_+ \\
        0 & 0 &0 & \sigma \mu_0\\
        0 & -B_1^* v_+ & \lambda v_+^2 & 
        v_+^2 (\sigma \mu_0)
    \end{pmatrix}.
\end{equation}

Making the ``balancing'' transformation
\be\label{balance}
\tilde A_+:=v_+^{-1/2}TA_+T^{-1},
\quad
T:={\rm diag  }\{ v_+^{1/2},1,1,1\}
\ee
and expanding in powers of $v_+$, we obtain
\ba
\tilde A_+ &= 
\tilde A^+_{0}+
v_+^{1/2}\tilde A^+_{1/2}+
v_+^{3/2}\tilde A^+_{3/2}\\
&:=
\begin{pmatrix}
        0 & 1 & 0 &0\\
        \lambda  & 0  & 0& 0\\
        0 & 0 &0 & 0\\
        0 & 0 &0 & 0\\
    \end{pmatrix}
+v_+^{1/2} \begin{pmatrix}
        0 & 0 & 0 &0\\
0& 1  & 0& -\sigma B_1^*  \\
        0 & 0 &0 &  \sigma \mu_0\\
        0 & -B_1^*  & \lambda  & 0
    \end{pmatrix}
\\
&\quad
+v_+^{3/2}
\begin{pmatrix}
        0 & 0 & 0 &0\\
        0 & 0 & 0 &0\\
        0 & 0 & 0 &0\\
        0 & 0 & 0 & \sigma \mu_0\\
    \end{pmatrix}.
\ea
Noting that the upper lefthand $2\times 2$ block of $\tilde A_{0}^+$
has eigenvalues $\pm \sqrt \lambda$ bounded from zero, we find \cite{Kato}
that $\tilde A_+$ has invariant projections 
$\Pi_1=\cR_1 \cL_1^*$ and $\Pi_2=\cR_2 \cL_2^*$
within $O(v_+^{1/2})$ of the standard Euclidean projections onto the 
first--second and the third--fourth coordinate directions, i.e.,
$$
\cR_1= \bp 1 & 0\\ 0 & 1\\ 
0 & 0\\ 0 & 0\\\ep
+ O(v_+^{1/2}), \;
\cR_2= \bp 
0 & 0\\ 0 & 0\\
1 & 0\\ 0 & 1\\ 
\ep
+ O(v_+^{1/2}),
$$
$$
\cL_1= \bp 1 & 0 & 0 & 0\\ 
0 & 1 & 0 & 0\\\ep
+ O(v_+^{1/2}) , \;
\cL_2= \bp 0 & 0 & 1 & 0\\ 
0 & 0 & 0 & 1\\\ep
+ O(v_+^{1/2}).
$$

Indeed, looking more closely-- expanding in powers of
$v_+^{1/2}$ and matching terms-- we find after a brief
calculation
$$
\cR_2= \bp 
0 & \frac{\sigma B_1^*v_+^{1/2}}{\lambda} \\ 0 & 0\\
1 & 0\\ 0 & 1\\ 
\ep
+ O(v_+),\;
\cL_2= \bp 0 & 0 & 1 & 0\\ 
B_1^* v_+^{1/2} & 0 & 0 & 1\\\ep
+ O(v_+).
$$

Looking at 
$
\cL_1\tilde A_+\cR_1=
\bp 0 & 1\\ \lambda &v_+^{1/2} \ep + O(v_+)
$
and noting that $\pm \sqrt{\lambda}$ are spectrally separated
by the assumption $|\lambda|\ge \theta>0$, we find that the stable
eigenvector within this space is 
$$
(1,-v_+^{1/2}/2-\sqrt{v_+/4+\lambda})^T+O(v_+^{1/2}),
$$
and thus the corresponding stable eigenvector within the full space
is 
$$
\tilde R_1= ( 1,-v_+^{1/2}/2-\sqrt{v_+/4+\lambda}, 0,0)^T+O(v_+^{1/2}).
$$

Looking at 
$
v_+^{-1/2}\cL_2\tilde A_+\cR_2=
\bp 0 & \sigma \mu_0 \\ \lambda & \sigma \mu_0 v_+ \ep + O(v_+^{3/2})
$
and noting that the eigenvalues 
$-\sigma \mu_0 v_+/2 \pm \sqrt{\sigma^2 \mu_0^2 v_+^2/4+ \sigma \mu_0\lambda}$ 
of the principal part are again spectrally separated so long as
$\sigma \mu_0>0$ are held fixed,
we find that the stable
eigenvector within this space is 
$$
\Big(1, - v_+/2 - \sqrt{ v_+^2/4+ \lambda/\sigma \mu_0}\;
\Big)^T+O(v_+^{3/2}),
$$
and thus the corresponding stable eigenvector within the full space
is 
$$
\tilde R_2= \Big(O(v_+^{1/2}),0, 1,
- v_+/2 - \sqrt{ v_+^2/4+ \lambda/\sigma \mu_0}\;
\Big)^T+O(v_+^{3/2}).
$$

Converting back to original coordinates, we find stable eigendirections
$T^{-1}\tilde R_1= (1,0,0,0)^T+O(v_+^{1/2})$
and
$$
T^{-1}\tilde R_2= \left(*,0, 1, - v_+/2 - \sqrt{ v_+^2/4+ \lambda/\sigma \mu_0}\right)^T+O(v_+^{3/2}),
$$
or, using an appropriate linear combination,
\be\label{Rs}
\hat R_1=\bp 1 \\0\\0\\0\ep + O(v_+^{1/2}),
\quad
\hat R_2=\bp 0\\0\\ 1\\
-v_+/2 - \sqrt{ v_+^2/4+ \lambda/\sigma \mu_0} \ep 
+ O(v_+^{1/2}).
\ee

Finally, we deduce the limiting Kato ODE flow as $v_+\to 0$.
A straightforward property of the Kato ODE is that it is 
invariant under constant coordinate transformations such as \eqref{balance}.
Thus, we find, for appropriate initialization, that 
$R_1^0\equiv (r,0,0,0)^T$,
where $(r,s)^T$ is the Kato 
eigenvector associated with $\bp 0 & 1\\ \lambda & v_+^{1/2} \ep $, or
(by the calculation of Example \ref{sqrteg}, setting $\eta=v_+^{1/2}$)
\be\label{R1onepre}
R_1\sim
\Big( \frac{v_+/4 + 1}{v_+/4+\lambda} \Big)^{1/4} 
 ( 1,0, 0,0)^T ,
\ee
with limit
\be\label{R1one}
R_1^0= (\lambda^{-1/4},0,0,0)^T.
\ee

Similar considerations yield a second limiting solution
$R_2=(0,0,r,s)^T$, where $(r,s)^T$ is the Kato eigenvector
associated with $\bp 0 & \sigma \mu_0\\ \lambda &\sigma \mu_0 v_+ \ep $, or
\be\label{R1twopre}
R_2\sim
\Big(\frac {v_+^2/4 + 1/\sigma\mu_0} {v_+^2/4+\lambda/\sigma \mu_0} \Big)^{1/4}
(0,0, 1, - v_+/2 - \sqrt{ v_+^2/4+ \lambda/\sigma \mu_0})^T,
\ee
with limit
\be\label{R2one}
R_2^0= 
\Big(0,0, \lambda^{-1/4}, 
- \lambda^{1/4}/\sqrt{\sigma\mu_0})\Big)^T.
\ee

We collect these observations as the following lemma.

\bl\label{limRlem}
On compact subsets of $\{\Re \lambda \ge 0\}\setminus \{0\}$, $R_1$ and $R_2$ converge uniformly in relative error
to fixed (i.e., independent of $\lambda$) linear combinations of $R_1^0$ and $R_2^0$ as defined in \eqref{R1one} and \eqref{R2one}.
\el

\br\label{goodcomp}
The above computations show that the formulae for $R_j^0$ remain valid so long as $|\lambda| \gg  v_+^2$.  Recall, for $|\lambda| \ll v_+^2$, the behavior is as described in Lemma \ref{rkatolem}.  This leaves only the case $|\lambda|\sim v_+^2$ unexamined.  
\er

\subsection{Limiting behavior at $+\infty$ as $v_+$, $\lambda \to 0$}
\label{lamsmall}
As suggested by the different behavior for $|\lambda| \gg v_+^2$ and
$|\lambda| \ll v_+^2$, behavior in
the transition zone $|\lambda|\sim v_+^2$ appears to be rather complicated,
and so we do not attempt to describe either the limiting subspace or
limiting Kato flow as $v_+$ and $\lambda$ simultaneously go to zero,
recording only the following topological information.

\bl\label{portrait}
In rescaled coordinates $(w,w',\alpha, \alpha'/\hat v)$,
for $v_+>0$ sufficiently small, the Kato product $R_1^+\wedge R_2^+$
defined above is analytic for $\Re \lambda \ge -\theta$, $\theta>0$
sufficiently small, except at two 
(possibly coinciding) 
singularities $\lambda_1$, $\lambda_2$ 
near the origin,
each of fourth-root type and blowing up as $(\lambda-\lambda_j)^{-1/4}$.
\el

\begin{proof}
Equivalently, by the computation of Example \ref{sqrteg},
we must show that each of the stable eigenvalues $\alpha_1$, $\alpha_2$ of
$A_+$ collide with unstable eigenvalues at precisely one point
$\lambda_j$, which is a branch point of degree two.
Computing the characteristic polynomial 
$ p(\lambda,\alpha):=\det (A_+(\lambda)-\alpha) $ 
with the aid of \eqref{evalcrit},
we obtain
$
p(\lambda, \alpha)=
(\alpha^2 - v_+ \alpha -\lambda v_+)
(\alpha^2 - \sigma \mu_0 v_+^2 \alpha -\lambda \sigma \mu_0  v_+^2)
- \sigma (B_1^*)^2 v_+^2\alpha^2,  $
a quadratic in $\lambda$.  Taking the resultant of $p$ with $\partial_\alpha 
p$,
we therefore obtain a quadratic polynomial $q(\lambda)$ whose
roots $\lambda_j$ are the points at which $A_+(\lambda)$
has double eigenvalues.  Noting that $\lambda_1=\lambda_2=0$
for $v_+=0$, we find by continuity that they lie near the origin
for $v_+$ sufficiently small.

Noting that 
$\partial_\alpha ^3 p= 8\alpha - 2( v_+ + \sigma \mu_0v_+^2)$,
we find that $\lambda_1=\lambda_2$ only if 
$\alpha =(1/4) ( v_+ + \sigma \mu_0v_+^2)$.
Plugging this into the linear equation 
$\partial_\alpha^2 p(\lambda, \alpha)=0$ in $\lambda$ gives
the further information
$ (2v_++ O(v_+^2))\lambda= v_+^2(-1/4-\sigma (B_1^*)^2 ) + O(v_+^3), $
hence $\lambda \sim v_+(-1/8-\sigma (B_1^*)^2/2 ) \gg v_+^2$ for $v_+$ small.
But, in this case, the analysis of \ref{limRlem}
implies that this coalescence represents a pair of branch points
of degree two and not a single branch point of degree four; 
see Remark \ref{goodcomp}.
The same analysis prohibits the possibility that either of
$\lambda_j$ represents a branch point of degree four,
hence they must each be degree two or three.
Finally, the global behavior described in Lemma \ref{limRlem}
excludes the possibility that they be degree three, leaving
the asserted result as the only possible outcome.
\end{proof}

\subsection{Limiting subspaces at $-\infty$ as $\lambda\to 0$ }\label{both}
{\bf Case (i)}($|B_1^*|/\sqrt{\mu_0}>1$)
For $\mu=1$, $v_-=1$, \eqref{evalcrit} becomes
\be\label{evalcrit2}
\lambda \in 
\sigma \Big(
\mu
\bp
1& -B_1^*/\mu_0 \\
-B_1^* & 1
\ep 
+ \mu^2
\bp
1 & 0\\
0 & 1/\sigma \mu_0 
\ep 
\Big),
\ee
whence we find by a standard limiting analysis \cite{ZH,MaZ3}
as $\lambda \to 0$ that the unstable subspace of $A_-$,
expressed in coordinates $(w,\alpha,w',\alpha')$, is spanned
by the direct sum of $R_1^-=\bp r_1\\0\ep$ and $R_2^-= \bp s_1\\ \mu_2 s_2\ep$,
where $r_1$ is the stable subspace of
$\bp
1& -B_1^*/\mu_0 \\
-B_1^* & 1
\ep 
$
and $s_2$ is the unstable subspace of
\be\label{ba}
\bp
1 & 0\\
0 & 1/\sigma \mu_0 
\ep^{-1}
\bp
1& -B_1^*/\mu_0 \\
-B_1^* & 1
\ep 
=
\bp
1& -B_1^*/\mu_0 \\
-B_1^* \sigma \mu_0 &  \sigma \mu_0
\ep ,
\ee
with $\mu_2$ the associated eigenvalue.

By direct computation, $r_1 \equiv (1,-\sqrt{\mu_0})^T$, while 
$$
\begin{aligned}
s_2&=\Big(1, 
\frac{(1-\sigma \mu_0)- \sqrt{(1-\sigma \mu_0)^2 + 4\sigma (B_1^*)^2}}
{2B_1^*/\mu_0}
\Big)^T\\
&=
\Big(1,
\frac{-2\sigma \mu_0 B_1^*}{(1-\sigma \mu_0)+
\sqrt{(1-\sigma \mu_0)^2 + 4\sigma (B_1^*)^2}}\Big)^T,
\end{aligned}
$$
and
$$
\mu_2=\frac{(1+\sigma\mu_0)+ \sqrt{(1-\sigma \mu_0)^2 + 4\sigma (B_1^*)^2}}
{2},
$$
from which we recover expressions in 
standard coordinates 
$\Big(w, w', \alpha, \frac{\alpha'}{\sigma \mu_0}\Big)^T$
of 
\be\label{limR1}
R_1^-=\Big(1,0,
-\sqrt{\mu_0} , 0\Big)^T
\ee
and
\be\label{limR2}
R_2^-=
\bp
1\\ 
\frac{(1+\sigma\mu_0) + \sqrt{(1-\sigma \mu_0)^2 + 4\sigma (B_1^*)^2}}{2}\\
\frac{(1-\sigma \mu_0)- \sqrt{(1-\sigma \mu_0)^2 + 4\sigma (B_1^*)^2}}
{2B_1^*/\mu_0}\\
-(2\sigma \mu_0 B_1^*)
\frac{(1+\sigma\mu_0)+ \sqrt{(1-\sigma \mu_0)^2 + 4\sigma (B_1^*)^2}}
{(1-\sigma \mu_0)+ \sqrt{(1-\sigma \mu_0)^2 + 4\sigma (B_1^*)^2}}\\
\ep.
\ee

{\bf Case (ii)}($|B_1^*|/\sqrt{\mu_0}\le 1$)
In this case, the unstable subspace of $A_-$ is spanned
by the direct sum of $(s_1, \mu_1 s_1)^T$ and $(s_2,\mu_2 s_2)$,
where $s_j$, $\mu_j$ are the unstable eigenvectors, eigenvalues
of \eqref{ba}, giving, by a similar computation as above,
\be\label{limR1ii}
R_1^-=
\bp
1\\ 
\frac{(1+\sigma\mu_0) - \sqrt{(1-\sigma \mu_0)^2 + 4\sigma (B_1^*)^2}}{2}\\
\frac{(1-\sigma \mu_0)+ \sqrt{(1-\sigma \mu_0)^2 + 4\sigma (B_1^*)^2}}
{2B_1^*/\mu_0}\\
-(2\sigma \mu_0 B_1^*)
\frac{(1+\sigma\mu_0)- \sqrt{(1-\sigma \mu_0)^2 + 4\sigma (B_1^*)^2}}
{(1-\sigma \mu_0)- \sqrt{(1-\sigma \mu_0)^2 + 4\sigma (B_1^*)^2}}\\
\ep.
\ee
and $R_2^-$ as in \eqref{limR2}.

\br\label{decayrmk}
The precise form of the eigenbases is not really important here,
only the fact that in case (i) there is a limiting direction
\eqref{limR1} corresponding to a nondecaying, zero-eigenvalue mode, 
whereas in case (ii) all solutions asymptotic
to $\Span\{R_1^-,R_2^-\}$ decay exponentially as $x\to -\infty$.
\er

\subsection{The limiting Evans function}\label{limev}
\begin{definition}
We define the limiting Evans function $D^0$ as the Evans function
associated with the limiting ODE \eqref{1sys} with $\hat v=\hat v^0$,
$\hat v^0$ as defined in Corollary \ref{limv}, with $R_+$ (indeterminate
for this system, since $A_+$ is almost empty) taken as
$R_+^0:=\lim_{v_+\to 0} R_+ $ computed above in \eqref{limR1}, \eqref{limR2},
and the renormalizations $\check D^0$, $\hat D^0$ as in \eqref{checkD1},
\eqref{hatD1}.
\end{definition}

\bpr\label{limconv}
Appropriately normalized,\footnote{As done automatically by our
method of numerical initialization; see Section \ref{numerics}.}
$D\to D^0$, $\check D\to \check D^0$, and $\hat D\to \hat D^0$
 uniformly on compact subsets of $\{\Re \lambda \ge 0\}
\setminus\{0\}$, up to a constant factor independent of $\lambda$.
Moreover, $\check D^0$ is
continuous on $\{\Re \lambda \ge 0\}$ and analytic
except for a square-root singularity at $\lambda=0$.
Both $D$ and $D^0$ extend meromorphically
to $B(0,r)$, for $r$, $v_+>0$ sufficiently small,
$D^0$ with a single square-root singularity $\lambda^{-1/2}$ at
the origin, and $\check D$ with a pair of fourth-root
singularities $(\lambda-\lambda_1)^{-1/4}$ and
$(\lambda-\lambda_2)^{-1/4}$ for  $\lambda_j\in B(0,r)$,
with $D\to D^0$ on $\partial B(0,r)$ for these extensions as well.
\epr

\begin{proof}
Convergence of $D$ on $\{\Re \lambda \ge 0\}\setminus \{0\}$ follows by
Lemmas \ref{conjlem} and \ref{limRlem},
Proposition \ref{profdecay}, and Corollary \ref{limv}, whereupon
convergence of $\check D$ and $\hat D$ 
follows by comparison of \eqref{checkevans} and \eqref{checkD1}
and of \eqref{hatevans} and \eqref{hatD1}.
Regularity of $\check D^0$ follows by Lemma \ref{conjlem}
and regularity of formulae \eqref{R1one}, \eqref{R2one},
as does holomorphic extension to $B(0,r)$.
Holomorphic extension of $\check D$ and the asserted description
of singularities follows by Lemmas \ref{conjlem} and \ref{portrait}.
\end{proof}

\subsubsection{Behavior near $\lambda=0$}
At the origin, we have the following 
striking bifurcation in behavior of $\check D^0$.

\bl\label{bif}
For $B_1^*\ge \sqrt{\mu_0}$, $\check D^0(0)\equiv 0$.
For $0\le B_1^*< \sqrt{\mu_0}$, $\check D^0(0)\ne 0$.
\el

\begin{proof}
The first assertion follows from the fact that, by \eqref{limR1},
for $B_1^*\ge \sqrt{\mu_0}$,
both the initializing eigenvector $R_1^-\equiv (1,0,-\sqrt{\mu_0},0)^T$
of $A_-$ and the initializing eigenvectors $R_1^0$, $R_2^0$
at $+\infty$ are {\it preserved by the flow of \eqref{1sys}} when
$\lambda=0$, for any value of $v_+$, 
corresponding to the fact that constant $w\equiv w_0$,
$\alpha\equiv \alpha_0$ are always solutions of 
\eqref{eval} when $\lambda=0$.
Thus, for $B_1\ge \sqrt{\mu_0}$, the first, third, and fourth columns in the
determinant \eqref{eq:evans} defining $\check D^0$,
consist of multiples of $(1,0,-\sqrt{\mu_0},0)^T$, $(1,0,0,0)$, 
and $(0,0,1,0)$, hence the determinant is zero. 
The second assertion follows similarly from the observation that for $B_1^*<\sqrt{\mu_0}$,
the solutions of \eqref{1sys} corresponding to 
$R_1^-$, $R_2^-$ at $\lambda=0$ are exponentially decaying as $x\to -\infty$,
hence independent of the constant solutions
corresponding to the initializing eigenvectors $R_1^0$, $R_2^0$ at $+\infty$.
\end{proof}

\br 
As the proof indicates, the bifurcation described in Lemma \ref{bif} originates in the nature (i.e., decaying vs. constant) of solutions as $x\to -\infty$, corresponding to change in type of the underlying shock.  Generically we expect that $\check D^0$ vanishes to square-root order at $\lambda=0$ for $B_1^*\ge \sqrt\mu_0$, since it has a square-root singularity there.
\er

\subsection{Proof of the limiting stability criteria}\label{largepf}

\begin{proof} [Proof of Theorem \ref{largeamp}]
By Theorem \ref{hfthm}, Proposition \ref{limconv},
and properties of limits of
analytic functions, it suffices to consider the case
that $|\lambda|$ and $|v_+|$ are arbitrarily small.
Denote the Evans function 
for a given $v_+$ as $D^{v_+}$, suppressing other parameters.
By Lemmas \ref{conjlem}, and \ref{portrait}, 
we may for $v_+$ sufficiently small
extend $D^{v_+}$ meromorphically to a
ball $B(0, r)$ about $\lambda=0$, and the resulting extension
is analytic (multi-valued) except at a pair of branch singularities
$\lambda_1$ and $\lambda_2$ at which $D^{v_+}$ behaves 
as $d_j(\lambda-\lambda_j)^{-1/4}$ for complex constants $d_j$.
Making a branch cut on the segment between $\lambda_1$ and 
$\lambda_2$ as in Figure \ref{branchcut}, we may view $\check D^{v_+}$ as an
analytic function on a slit, two-sheeted Riemann surface obtained by 
circling the deleted segment $\overline{\lambda_1 \lambda_2}$.
Applying Proposition \ref{limconv} again, we find that
$$
D^{v_+}(\lambda)\sim  D^0(\lambda) \sim 
c_0\lambda^{-1/2} + c_1 
$$
on $\partial B(0,r)$ as $v_+\to 0$, where $c_j$ are complex
constants.

By Lemma \ref{bif}, $c_0\ne0$ for $B_1^*<\sqrt{\mu_0}$.
For $B_1^*\ge \sqrt{\mu_0}$,
$c_0\equiv 0$, and the condition that $\hat D^0\sim D^0$ 
not vanish at the origin is the condition that $c_1\ne 0$.
Taking the winding number of $D^{v_+}$
around $\partial B(0,r)$, therefore, on the two-sheeted Riemann
surface we have constructed-- that is, circling twice as $D^{v_+}$
varies meromorphically-- we obtain in the first place {\it winding number
negative one}, and in the second (assuming $c_1\ne 0$) 
{\it winding number zero}.
Subtracting the winding number about 
the segment $\overline{\lambda_1 \lambda_2}$,
necessarily greater than or equal to negative one by the asymptotics
of $D^{v_+}$ at $\lambda_j$, we find by Cauchy's Theorem/Principle of
the Argument that for $B_1^*<\sqrt{\mu_0}$
there are {\it no zeros} of 
$\check D^{v_+}$ within $B(0,r)\setminus \overline{\lambda_1 \lambda_2}$,
concluding the proof in this case.

For $B_1^*\ge \sqrt{\mu_0}$, we find that there is {\it at most one zero}
of $\check D^{v_+}$ within $B(0,r)\setminus \overline{\lambda_1 \lambda_2}$.
To complete the proof, we appeal as in \cite{CHNZ} to the 
mod-two {\it stability index} of \cite{GZ,Z1,Z2}, which counts the
parity of the number of unstable eigenvalues according to its
sign, and is given by a nonzero real multiple of $D^{v_+}(0)$.
To establish the theorem, it suffices to prove then that this stability
index does not change sign, since we could then conclude stability
by homotopy to a limiting stable case $\sigma \to +\infty$ or
$B_1^*\to +\infty$.  (Alternatively, we could check the sign by
explicit computation, but we do not need to do so.)
Recall that $D^{v_+}(0)$ is a nonvanishing real multiple of
the product of the hyperbolic stability determinant and a transversality
coefficient vanishing if and only if the traveling wave connection is
not transverse.

As noted already in Proposition \ref{loptran}, the hyperbolic stability
determinant does not vanish for Lax $3$-shocks, so is nonvanishing
for $B_1^*>\sqrt{\mu_0}$.
The transversality coefficient is an Evans function-like 
Wronskian of decaying solutions of the linearized traveling-wave ODE
\be\label{trav2}
\hat v^{-1}
\bp
\mu_0 &0\\
0& 1/\sigma \mu_0
\ep
\bp \tilde u\\ \tilde B \ep'=
\bp
\mu_0 & -B_1^*\\
-B_1^*&\hat v 
\ep
\bp \tilde u\\ \tilde B \ep,
\ee
hence converges by Lemma \ref{conjlem} to the corresponding
Wronskian for the limiting system with $\hat v$ replaced by 
$\hat v^0$.
But, this limit must be nonzero wherever $c_1$ is nonzero, 
or else $\check D^0$ would
vanish at $\lambda=0$ to at least order $\lambda$ due to a second
linear dependence in decaying as well as asymptotically constant
modes, and so $c_1=0$ in contradiction to our assumptions.
Therefore, transversality holds by assumption for
$0\le B_1^*-\sqrt{\mu_0} \le \max\{ \sqrt{ \frac{\mu_0}{2}}, 
\sqrt{\frac{1}{2\sigma}} \, \} $ and $v_+$ sufficiently small.

On the other hand, an energy estimate like that of Section 
\ref{transest} sharpened by the observation that
$|\hat v^0_x|\le \hat v$, improving the general estimate
$\hat v_x\le \gamma \hat v$,
yields transversality of \eqref{trav2} for
$ B_1^*-\sqrt{\mu_0} \ge \max\{ \sqrt{ \frac{\mu_0}{2}},
\sqrt{\frac{1}{2\sigma}} \, \} $.
Thus, we have transversality for all $B_1^*\ge \sqrt{\mu_0}$,
and we may conclude by homotopy to the stable $B_1^*\to \infty$
limit that the transversality coefficient has a sign consistent
with stability, that is, there are an even number of nonstable zeros 
$\Re \lambda \ge 0$ of
the Evans function $D^{v_+}$ for $v_+$ sufficiently small.
Since we have already established that there is at most one nonstable 
zeros of $D^{v_+}$, this implies that there are no nonstable
zeros, yielding stability as claimed.
\end{proof}

\begin{figure}[t]
\begin{center}
\includegraphics[width=8cm]{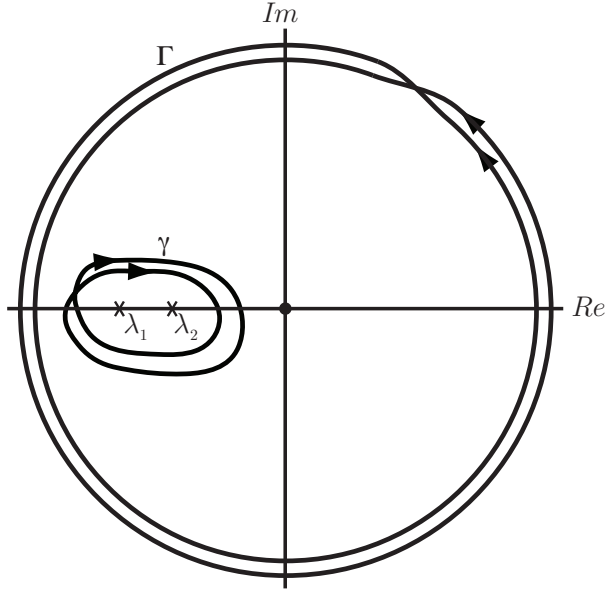} 
\end{center}
\caption{
Winding number on two-sheeted Riemann surface.
}
\label{branchcut}
\end{figure}

\br\label{rmA}
From Lemma \ref{bif}, there might appear to be inherent numerical
difficulty in verifying nonvanishing of $\check D$ for 
$B_1^*$ less than but close to $\sqrt{\mu_0}$, since $\check D$ vanishes
at the origin for $B_1^*=\sqrt{\mu_0}$.
However, this is only apparent, since we know analytically that $\check D^0$
does not vanish at, hence also near, $\lambda=0$ for $B_1^*<\sqrt{\mu_0}$.
\er

\section{Further asymptotic limits}\label{further}
In this section, we require beyond the conjugation
lemma the further asymptotic ODE tools of the convergence
and tracking/reduction lemmas of \cite{MaZ3,PZ}.
Statements and proofs of these results are given
for completeness in Appendix \ref{s:contrack}.

\subsection{The small-$\sigma$ and -$\mu_0$ limits}

\begin{proof}[Proof of Theorem \ref{sigmu} ($\sigma\to 0$)]
Considering \eqref{1sys} as indexed by $p:=\sigma$ with $A=A^\sigma$, we have \eqref{residualest}--\eqref{newest} by uniform exponential convergence of $\hat v$ as $x\to \pm \infty$.  Take without loss of generality $\mu=1$.  Applying Lemma \ref{evanslimit}, we find that the transformations $P^\sigma_\pm$ conjugating \eqref{1sys} to its constant-coefficient limits $Z'=A^\sigma_\pm Z$, by which the Evans function is defined in \eqref{eq:evans}, are given to $O(\sigma)$ by the transformations $P^0_\pm$ conjugating to its constant-coefficient limits the $\sigma=0$ system $W'=A^0(x,\lambda )W$, with
\begin{equation}\label{mu01sys}
A^0= \begin{pmatrix}
        0 & 1 & 0 &0\\
        \lambda \hat v& \hat v& 0& 0\\
        0 & 0 &0 & 0\\
        0 & -B_1^*\hat v& \lambda \hat v & 0 
    \end{pmatrix};
\end{equation}
that is, $P^\sigma_\pm=P^0_\pm +O(\sigma)$.
Moreover, for $v_+$ bounded from zero, and $\sigma/\lambda$ sufficiently
small, it is straightforward to verify that the stable subspace
of $A_+(\lambda,\sigma)$ is given to order $\sigma/\lambda$ by the
span of $(r_+,s_+)^T$ and $(0, q_+)^T$, where $r_+^T$ is the stable eigenvector
of $\begin{pmatrix} 0 & 1\\ \lambda \hat v& \hat v \end{pmatrix}$
and $q_+=( -\sqrt{\sigma \mu_0/\lambda}, 1)$, 
and, similarly, the unstable subspace
of $A_-(\lambda,\sigma)$ is given to order $\sigma/\lambda$ by the
span of $(r_-,s_-)^T$ and $(0, q_-)^T$, where $r_-^T$ is the 
unstable eigenvector of $\begin{pmatrix} 0 & 1\\ \lambda \hat v& \hat v \end{pmatrix}$
and $q_-=(\sqrt{\sigma \mu_0/\lambda},1)$.  Thus, the Evans function
for $\sigma>0$, appropriately rescaled, 
is within $O(\sigma/\lambda)$ of the product of the Evans function
of the diagonal block
$w'=\begin{pmatrix} 0 & 1\\ \lambda \hat v& \hat v \end{pmatrix}w$
initialized in the usual way, which is nonzero by our earlier
analysis of the decoupled case $B_1^*=0$, and of the trivial flow
\be\label{tflow}
 w'= \hat v(x) \begin{pmatrix} 0 & 0\\ \lambda  & 0 \end{pmatrix}w 
\ee
initialized with vectors parallel to $q\pm^T$ in the conjugated flow.
To estimate the second determinant, we produce explicit conjugators $P^0_\pm$
for the $\sigma=0$ flow, making use of the observation of Remark
\ref{specialform} that, by lower block-triangular form of the $A^0$,
these may be taken lower block-triangular as well, and so the problem
reduces to finding conjugators $p^0_\pm$ for the flow \eqref{tflow}
in the lower block.
But, these may be found by exponentiation to be
$$
p^0_\pm= \bp 1 & 0\\ c_\pm(x) & 1 \ep,
\qquad
c_\pm(x):=\int_x^{\pm \infty} (\hat v-v_\pm)(y) dy,
$$
yielding an Evans function to order $\sigma/\lambda$ of
$$
\det(p^0_- q_-, p^0_+q_+)=
\det\bp \sqrt{\sigma \mu_0/\lambda}& -\sqrt{\sigma \mu_0/\lambda}\\
1+ c_-(0) \sqrt{\sigma \mu_0/\lambda}& 1 -c_+\sqrt{\sigma \mu_0/\lambda}
\ep,
$$
or 
$$
(1+O(\sqrt{\sigma/\lambda})) 2  \sqrt{\sigma \mu_0/\lambda} 
\ne 0. 
$$
In particular, the Evans function is nonvanishing for
$1 \gg |\sqrt{\sigma/\lambda}| \gg |\sigma|, \, |\sigma/\lambda|$,
as occurs for 
$$
\sigma \ll |\lambda| \ll \sigma^{-1}.
$$
Since the Evans function is nonvanishing in any case for 
$|\lambda|$ sufficiently large, by Theorem \ref{hfthm},
we obtain nonvanishing except in the case $0\le |\lambda|\le C\sigma$,
which must be treated separately.

To treat $|\lambda|\le C\sigma$, notice that the stable/unstable subspaces
of $A^\sigma_\pm$ decouple to order $\sigma$ for $\sigma>0$ sufficiently small
and $\lambda$ only bounded
into the direct sum of $(r,s)_\pm^T$ already discussed and
$(0,\tilde q_\pm)^T$ with $\tilde q_\pm^T$ the stable/unstable
eigenvectors of
$$
\bp 0 & 1\\ \lambda/\sigma \mu_0& \sigma \mu_0v
\ep,
$$
which may be chosen holomorphically as
$\tilde q_\pm= (1, \sigma \mu_0/2 \mp\sqrt{\sigma^2\mu_0^2/4+  \lambda/\sigma \mu_0})$,
with a single square-root singularity at $\lambda= -(\sigma \mu_0)^3/4$.
For $|\lambda|\ge C\sigma$, these factor as
$$
\tilde q_\pm=
(1+O(\sigma/\lambda))(1, \mp \sqrt{\lambda/\sigma \mu_0})=
(1+O(\sigma/\lambda)) (\sqrt{\lambda/\sigma \mu_0}) 
q_\pm,
$$
and the Evans function for $\sigma>0$ by our previous computations
thus satisfies
$$
D^\sigma(\lambda)=
(\lambda/\sigma \mu_0)\times 
(1+O(\sqrt{\sigma/\lambda})) 2 e^{c_-(0)+c_+(0)} \sqrt{\sigma \mu_0/\lambda} 
\sim
C\sqrt{\lambda/\sigma \mu_0}.
$$
Taking the winding number about $|\lambda|=C\sigma$ on the punctured
Riemann surface obtained by circling twice the branch singularity
$\lambda_*=(\sigma \mu_0)^3/4$, we thus obtain winding number one.
Subtracting the nonnegative winding number obtained by circling
twice infinitesimally close to $\lambda_*$, we find (similarly as
in the treatment of the large-amplitude limit, case $B_1^*\ge \sqrt{\mu_0}$) 
that there is at most one root of $D^\sigma$ on $\Re \lambda \ge 0$, for
$\sigma>0$ sufficiently small, so that stability is decided by the
sign of the stability index, which is the product of a transversality
coefficient and the hyperbolic stability determinant (resp.
low-frequency stability condition, in the overcompressive case).
A singular perturbation analysis of \eqref{cr} as $\sigma\to 0$
shows that (since it decouples into scalar fibers) connections are
always transverse for $\sigma>0$, so the transversality coefficient
does not vanish.
Hyperbolic stability holds always for Lax $1$- and $3$-shocks
(Proposition \ref{loptran}), and the low-frequency stability
condition holds for intermediate (overcompressive)
shocks by a similar singular perturbation analysis, 
so the stability determinant does not vanish either.

Thus, the sign of the stability index is constant, and so there
is always either a single unstable root of $D^\sigma$ on $\Re \lambda\ge 0$
or none, in each of the three cases.  But, the former possibility may
be ruled out by homotopy to the stable, small-amplitude limiting case.
Thus, {all type shocks are reduced Evans stable for $\sigma >0$
sufficiently small.}
The asserted $\sqrt {\sigma}$ asymptotics follow from the estimates
already obtained in the proof; uniform convergence to zero follows
by estimating $D^\sigma$ instead to order $\sqrt{\sigma}$, at which
level we obtain a determinant involving two copies of the constant
solution $W\equiv (0,0,0,1)^T$ of the limiting $\sigma=0$
system, giving zero as the limiting value.
\end{proof}

\begin{proof}[Proof of Theorem \ref{sigmu} ($\mu_0\to 0$)]
The case $\mu_0\to 0$ is similar to but a bit tricker than the
case $\sigma\to 0$ just discussed.
Fixing without loss of generality $\mu=1$ and
applying Lemma \ref{evanslimit}, we deduce 
that the transformations $P^{\mu_0}_\pm$ conjugating
\eqref{1sys} to its limiting constant-coefficient systems,
by which the Evans function is defined in \eqref{eq:evans},
satisfy $P^{\mu_0}_\pm=P^0_\pm + O(\mu_0)$, where $P^0_\pm$
are the transformations conjugating to its constant-coefficient limits
the upper block-triangular $\mu_0=0$ system
\begin{equation}\label{mu1sys}
 \begin{pmatrix} W_1\\W_2\\W_4\\W_3 \end{pmatrix}'
= \begin{pmatrix}
        0 & 1 & 0 &0\\
        \lambda \hat v& \hat v &  -\sigma B_1^*\hat v& 0 \\
        0 & -B_1^*\hat v & 0&  \lambda \hat v  \\
        0 & 0 &0 & 0\\
    \end{pmatrix}
 \begin{pmatrix} W_1\\W_2\\W_4\\W_3 \end{pmatrix},
\end{equation}
which has a constant right zero-eigenvector
$r=(\sigma B_1^*, 0, \lambda, 0)^T$ and an orthogonal constant
left zero-eigenvector $\ell=(0,0,0,1)^T$, signaling a
Jordan block at eigenvalue zero.
It is readily checked for the limiting matrices at $\pm \infty$,
similarly as in the $\sigma\to 0$ case,
that for $\mu_0/\lambda$ sufficiently small, the Jordan block
splits to order $\sim \sqrt {\mu_0/\lambda}$, so that the ``slow''
stable eigenvector at $+\infty$ (that is, the one with eigenvalue
near zero) is given by
$$
r+  c_+\sqrt{\mu_0/\lambda}(*,*,*,1)^T + O(\mu_0/\lambda),
$$
and the slow unstable eigenvector at $-\infty$ by
$$
r+  c_-\sqrt{\mu_0/\lambda}(*,*,*,1)^T + O(\mu_0/\lambda),
$$
where $c_\pm $ are constants with a common sign.
(Here, we deduce nonvanishing of the final coordinate of
the second summand without
computation by noting that the dot product with $\ell$ must
be $\sim \sqrt{\mu_0/\lambda}$.)

As for the $\sigma\to 0$ case, we now observe that
\eqref{mu1sys} may be conjugated to constant-coefficients
by block-triangular conjugators
$ P_\pm=\begin{pmatrix} p_\pm & q_\pm \\ 0 &  1 \end{pmatrix} $,
where $p_\pm$ conjugate the upper lefthand block system
\be\label{ub}
 \begin{pmatrix}W_1\\W_2\\W_4 \end{pmatrix}'
=
 \begin{pmatrix}
        0 & 1 & 0 \\
        \lambda \hat v& \hat v &  -\sigma B_1^*\hat v \\
        0 & -B_1^*\hat v & 0  \\
    \end{pmatrix} 
 \begin{pmatrix}W_1\\W_2\\W_4 \end{pmatrix}.
\ee
Moreover, changing coordinates to lower block-triangular form
\be\label{bt}
 \begin{pmatrix}W_1\\W_2\\W_4-\lambda W_1/\sigma B_1^* \end{pmatrix}'
=
 \begin{pmatrix}
        0 & 1 & 0 \\
        0& \hat v &  -\sigma B_1^*\hat v \\
        0 & -B_1^*\hat v -\lambda/\sigma B_1^* & 0  \\
    \end{pmatrix} 
 \begin{pmatrix}W_1\\W_2\\W_4-\lambda W_1/\sigma B_1^* \end{pmatrix},
\ee
conjugating by a lower block-triangular conjugator, and changing back
to the original coordinates,
we see that the conjugators $p_\pm$ may be chosen to preserve the
exact solution $(W_1,W_2,W_4,W_3)^T\equiv r$.

Combining these observations, we find that the Evans function for
$\lambda$ bounded and $\mu_0/\lambda$ sufficiently small is given by
\be\label{Das}
\begin{aligned}
D^{\mu_0}(\lambda)&=
\det
\begin{pmatrix}
\tilde  r  & v_2^-&v_2^+& \tilde r\\
c_-\sqrt{\mu_0/\lambda} & 0 & 0& -c_+\sqrt{\mu_0/\lambda}\\
\end{pmatrix}
+O(\mu_0/\lambda)\\
&=
\det
\begin{pmatrix}
0& v_2^-&v_2^+& \tilde r\\
(c_-+c_+)\sqrt{\mu_0/\lambda} & 0 & 0& c_+\sqrt{\mu_0/\lambda}\\
\end{pmatrix}
+O(\mu_0/\lambda)\\
&=
(c_-+c_+)\sqrt{\mu_0/\lambda} d(\lambda) +O(\mu_0/\lambda) ,
\end{aligned}
\ee
where $r=:\begin{pmatrix} \tilde r\\0 \end{pmatrix}$ and
$d(\lambda):= \det \begin{pmatrix} v_2^-&v_3^+& \tilde r \end{pmatrix}$
is a nonstandard Evans function associated with the 
upper-block system \eqref{ub},
where $v_2^-$ and $v^+_3$ as usual are unstable and stable eigendirections
of the coefficient matrix, but we have included also the neutral mode
$\tilde r$.
Expressed in coordinate \eqref{bt},  $d(\lambda)$ reduces, finally,
to the standard
Evans function $\check d(\lambda)$ for the reduced system
$$
 \begin{pmatrix}W_2\\W_4-\lambda W_1/\sigma B_1^* \end{pmatrix}'
=
 \begin{pmatrix}
         \hat v &  -\sigma B_1^*\hat v \\
         -B_1^*\hat v -\lambda/\sigma B_1^* & 0  \\
    \end{pmatrix} 
 \begin{pmatrix}W_2\\W_4-\lambda W_1/\sigma B_1^* \end{pmatrix},
$$
which may be rewritten as a second order equation
\be\label{two}
\big(\lambda +\sigma (B_1^*)^2\big)z + z'= (z'/\hat v)'
\ee
in $z=W_2'$.
Taking the real part of the complex $L^2$-inner product 
of $z$ against \eqref{two} gives
$$
\Re \lambda \|v\|_{L^2}^2
+ \| v\sqrt{
\big(\Re \lambda +\sigma (B_1^*)^2\big)
}\|_{L^2}^2
=- \| v'/\sqrt{\hat v}\|_{L^2}^2,
$$
contradicting the existence of a decaying solution for $\Re \lambda \ge 0$
and verifying that $\check d(\lambda)\ne 0$.
Consulting \eqref{Das}, therefore, we find that $D^{\mu_0}(\lambda)$
for $\lambda$ bounded and $\mu_0/\lambda$ sufficiently small does
not vanish and, moreover, $D^\mu_0\sim c \sqrt{\mu_0/\lambda}$
for $\lambda$ sufficiently small, $c\ne 0$ constant.
Performing a Riemann surface winding number computation like that
for the case $\sigma \to 0$, 
we find, finally, that $D^{\mu_0}$ does not vanish for any $\Re \lambda \ge 0$.
We omit the details of this last step, since they are essentially
identical to those in the previous case.
Likewise, the asserted asymptotics follow exactly as before.
\end{proof}

\subsection{The large- and small-$B_1^*$ limits}

\begin{proof}[Proof of Theorem \ref{largeB}]
Stability in the small-$B_1^*$ limit follows readily by
continuity of the Evans function with respect to parameters,
the high-frequency bound of Theorem \ref{hfthm}, and
the zero-$B_1^*$ stability result of Proposition \ref{smallprop}.
We now turn to the large-$B_1^*$ limit.
Let us rearrange \eqref{1sys}, $\mu=1$, to
\be\label{Bsys}
    \begin{pmatrix}
        w\\\alpha\\  w'\\ \frac{\alpha'}{\sigma\mu_0 \hat v}
    \end{pmatrix}'=
    \begin{pmatrix}
        0 & 0 & 1 &0\\
        0 & 0& 0& \sigma \mu_0 \hat v \\
        \lambda \hat v & 0 &\hat v & -\sigma B_1^*\hat v \\
        0 & \lambda \hat v & -B_1^*\hat v & \sigma \mu_0 \hat v^2\\
    \end{pmatrix}
    \begin{pmatrix}
        w\\\alpha\\ w'\\ \frac{\alpha'}{\sigma\mu_0 \hat v}
    \end{pmatrix},
\ee
By Theorem \ref{hfthm}, we have stability for $|\lambda|\ge C|B_1^*|^2$
independent of $v_+$.  For $v_+>0$, we find easily stability
for $|\lambda|\ge C|B_1^*|$ for $B_1^*$
sufficiently large.
For, rescaling $x\to |B_1^*|x$, and 
$W\to (\lambda^{1/2}W_1, \lambda^{1/2}W_2, W_3,W_4)^T$,
we obtain $W'=\hat AW= \hat A_0W +O(|B_1^*|^{-1})W$, where
\ba\label{dec}
\hat A_0&=
    \begin{pmatrix}
        0 & 0 & \hat \lambda^{1/2} &0\\
        0 & 0& 0& \hat \lambda^{1/2}\sigma \mu_0 \hat v \\
        \hat \lambda^{1/2} \hat v & 0 &0 & -\sigma \hat v \\
        0 & \hat \lambda^{1/2} \hat v & - v & 0\\
    \end{pmatrix},\\
\ea
with $\hat \lambda^{1/2}:= \lambda^{1/2}/B_1^{*}$, and
$\hat v=\bar v(x/B_1^*)$, $\bar v$ independent of $B_1^*$.

For $\hat \lambda \gg |B_1^*|^{-1}$ it is readily calculated
that $\hat A_0$ has spectral gap $\gg |B_1^*|^{-1}$
for $\Re \lambda \ge 0$.
Indeed, splitting into cases $\hat \lambda \ge C^{-1}$ and
$\hat \lambda \ll 1$, it is readily verified 
in the first case by standard matrix perturbation theory 
that there exist
matrices $R(\hat v(x))$ and $L=R^{-1}$, both smooth functions of $\hat v$, 
such that
$$
L\hat A_0R= D:=\bp M&0\\0&N\ep,
$$
with $\Re M\ge \theta>0$ and $\Re N\le -\theta<0$.
Making the change of coordinates $W=RZ$, we obtain the
approximately block-diagonal equations $Z'=\tilde A Z$, where 
\be\label{form1}
\tilde A:= LAR-L'R=D+ O(|B_1^*|^{-1}).
\ee
Using the tracking/reduction lemma, Lemma \ref{reduction},
we find that there exist analytic functions $z_2=\Phi_2(z_1)=O(r)$ 
and $z_1=\Phi_1(z_2)=O(r)$ such that $(z_1, \Phi_2(z_1)$
and $(\Phi_1(z_2),z_2)$ are invariant under the flow of \eqref{zeq},
hence represent decoupled stable and unstable manifolds of the flow.
But, this implies that the Evans function is nonvanishing on
$\lambda \in \{\Re \lambda\ge 0\}$ for $B_1^*$ sufficiently large
and $|\lambda|^{1/2}\ge |B_1^*|/C$, for any fixed $C>0$.
See \cite{ZH,MaZ3,Z1} for similar arguments.

If $|\lambda^{1/2}| \ll B_1^*$ on the other hand, or, equivalently,
$|\hat \lambda^{1/2}| \ll 1$, then we can decompose $\hat A$ alternatively
as $W'=\hat AW= \hat B_0W +\hat \lambda^{1/2}B_1W+ O(|B_1^*|^{-1})$,
where
\ba\label{dec2}
\hat B_0&=
    \begin{pmatrix}
        0 & 0 & 0 &0\\
        0 & 0 & 0 &0\\
0& 0 &0 & -\sigma \hat v \\
0 & 0 &  - v & 0\\
    \end{pmatrix},
\qquad
\hat B_1&=
    \begin{pmatrix}
        0 & 0 & 1 &0\\
        0 & 0& 0& \sigma \mu_0 \hat v \\
         \hat v & 0 &0 & 0\\
        0 &  \hat v & 0 & 0\\
    \end{pmatrix}.
\ea
By smallness of $\hat \lambda^{1/2}$ together with spectral separation
between the diagonal blocks of $\hat B_0$, there exist $L$, $R$,
$LR\equiv 0$ such that the transformation $W=RZ$ takes the system to
$Z'=(LBR-L'R)Z=CZ+ O(|B_1^*|^{-1})\Theta $, where
$\Theta=\bp 0 & *\\* & *\\\ep$ and
\be\label{dec3}
C= \begin{pmatrix}
        -\hat \lambda \beta^{-1}+O(\hat \lambda^{2}) & 0 \\
        0 & \hat \beta 
    \end{pmatrix},
\qquad
\beta= \begin{pmatrix}
0 & -1/\mu_0 \hat v \\
 - 1 & 0\\
    \end{pmatrix},
\qquad
\hat \beta= \begin{pmatrix}
0 & -\sigma \hat v \\
 - v & 0\\
    \end{pmatrix}.
\ee
Diagonalizing $\beta$ into growing and decaying mode by a further transformation, and applying the tracking lemma again, we may decouple the equations into a scalar uniformly-growing mode, a scalar uniformly-decaying mode, and a $2$-dimensional mode governed by
\be\label{redB}
z'=-\hat \lambda \beta^{-1}z + O(|B_1^*|^{-2}+|\hat \lambda^{2}|)z.
\ee
If $\hat \lambda \gg |B_1^*|^{-1}$, or, equivalently, $|\lambda| \gg |B_1^*|$, then we make a further transformation diagonalizing $\beta^{-1}$ at the expense of an $O(|B_1^*|^{-1})$ error, then use the resulting $\ge |\hat \lambda|$ spectral gap together with the tracking lemma to again conclude nonvanishing of the Evans function.

Thus, we may restrict to the case $|\lambda|\le C|B_1^*|$,
or $|\hat \lambda|\le C|B_1^*|^{-1}$.
Considering again \eqref{redB} in this case,
we find that all
$O(|B_1^*|^{-2}+|\hat \lambda^{2}|)$ entries converge
at rate 
$$
O(|B_1^*|^{-2}) |\hat v-v_+|\le
C |B_1^*|^{-2} e^{-|x|/CB_1^{*}}
$$
to limiting values, whence, by the convergence lemma, Lemma \ref{evanslimit},
the Evans function for the reduced system \eqref{redB} converges to that for
$z'=-\hat \lambda \beta^{-1}z$ as $B_1^*\to \infty$.\footnote{
Here, as in Remark \ref{convrmk}, we are using the fact that also
the stable/unstable eigenspaces at $+\infty$/$-\infty$ converge to 
limits as $|B_1^*|\to \infty$.
Together with convergence of the conjugators $P^{B_1^*}_\pm$, this
gives convergence of the Evans function by definition \eqref{eq:evans}.
}
But, this equation, written in original coordinates, is exactly
the eigenvalue equation for the reduced inviscid system
\be\label{redz}
\lambda \bp w\\\alpha\ep+
\bp
0& -B_1^*/\mu_0 \hat v\\
-B_1^* & 0
\ep
\bp w\\\alpha\ep'=0,
\ee
which may be shown stable by an energy estimate as in the
case $\sigma=0$.

Finally, noting that the decoupled fast equations are independent
of $\lambda$ to lowest order, we find for $|\lambda|$ bounded and
$B_1^*\to \infty$ that the Evans function (which decomposes into
the product of the decoupled Evans functions) converges to a constant
multiple of the Evans function for \eqref{redz}.
For $|\lambda|\le C$, or $\hat \lambda\le C|B_1^*|^{-2}$, however,
we may apply to \eqref{redB} 
the convergence lemma, Lemma \ref{evanslimit}, together with 
Remark \ref{convrmk}, to see that the Evans function in fact 
converges to that for the piecewise
constant-coefficient equations obtained by substituting for the
coefficient matrix on $x\gtrless 0$ its asymptotic values at $\pm \infty$,
that is, the determinant $d:=\det (r^+,r^-)$, where
$r^+$ is the stable eigenvector of 
$$
\bp
0& -B_1^*/\mu_0 \hat v\\
-B_1^* & 0
\ep
$$
at $+\infty$ and
$r^-$ is the unstable eigenvector at $-\infty$.
Computing, we have $r^\pm=(1,\mp \sqrt{\mu_0 v_\pm})^T$ where
giving a constant limit $d=\sqrt{\mu_0 v_+}+ \sqrt{\mu_0}$ as claimed.
\end{proof}

\subsection{The large-$\sigma \mu_0$ limit}

\begin{proof}[Proof of Theorem \ref{sigmu2}]
By Theorem \ref{hfthm}, it is sufficient to treat the case
$|\lambda|\le C\sigma\mu_0$.
Decompose \eqref{1sys}, $\mu=1$, as $W'=R A_0 + A_1$, where
$R:=\sigma \mu_0$, $\hat \lambda:=\lambda/R$, and
$$
A_0=
    \begin{pmatrix}
        0 & 0 & 0 &0\\
        \hat \lambda \hat v & 0 & 0 & -B_1^*\hat v/\mu_0 \\
        0 & 0 &0 &  \hat v\\
        0 & 0 & \hat \lambda \hat v& \hat v^2
    \end{pmatrix},
\qquad
A_1=
    \begin{pmatrix}
        0 & 1 & 0 &0\\
        0& \hat v& 0& 0\\
        0 & 0 &0 & 0\\
        0 & -B_1^*\hat v & 0 & 0\\
    \end{pmatrix},
$$
with $|A_1|\le C$.
If $|\hat \lambda|\ge 1/C>0$, then the lower $2\times 2$ righthand block
of $A_0$ has eigenvalues $\pm \hat v\sqrt{\hat \lambda} $
uniformly bounded from the eigenvalues zero of the upper lefthand
$2\times 2$ block. 
By standard matrix perturbation theory, therefore, there exist
 well-conditioned coordinate transformations $L$,
$R$ depending smoothly on $\hat v$ such that 
$$
D:=LA_0R=
    \begin{pmatrix}
        0 & 0 & 0 &0\\
        \hat \lambda \hat v & 0 & 0 & 0\\
        0 & 0 &\hat v\sqrt{\hat \lambda}& 0\\
        0 & 0 &0 & -\hat v\sqrt{\hat \lambda}\\
    \end{pmatrix}.
$$
Making the coordinate transformation $W=RZ$, we obtain
$Z'=DZ + O(1)Z$.
Applying the tracking lemma, Lemma \ref{reduction}, we reduce to
a system of three decoupled equation, consisting of a uniformly growing scalar
equation, a uniformly decaying scalar equation, and a $2\times 2$ equation
$
    z'=\begin{pmatrix}
        0 & 1 \\
        R\hat \lambda \hat v & 0 
\ep
z + O(R^{-1})z.
$
Rescaling by $z:=\bp 1 & 0 \\0 & R^{1/2}\ep y$,
we obtain
$
    y'=R^{1/2}\begin{pmatrix}
        0 &  1 \\
        \hat \lambda \hat v & 0 
\ep
y + O(R^{-1/2})y,
$
which, by a second application of the tracking lemma, may
be reduced to a pair of decoupled, uniformly growing/decaying
scalar equations, thus completely decoupling the original system
into four growing/decaying scalar equations, from which we may
conclude nonvanishing of the Evans function.

It remains to treat the case $|\hat \lambda| \ll 1$.
We decompose \eqref{1sys}, $\mu=1$, in this case as $W'=R B_0 + B_1$, where
$$
B_0=
    \begin{pmatrix}
        0 & 0 & 0 &0\\
0& 0 & 0 & -B_1^*\hat v/\mu_0 \\
        0 & 0 &0 &  \hat v\\
        0 & 0 & 0& \hat v^2
    \end{pmatrix},
\qquad
B_1=
    \begin{pmatrix}
        0 & 1 & 0 &0\\
        \lambda \hat v & \hat v& 0& 0\\
        0 & 0 &0 & 0\\
        0 & -B_1^*\hat v & \lambda \hat v& 0\\
    \end{pmatrix},
$$
with $|\lambda| \ll R$.
Defining $T=\bp I & \theta \\ 0 & I\ep$ where
$\theta= (0 , -B_1^*\hat v/\mu_0 , \hat v)^T$, and making
the change of variables $W=TZ$, we obtain
$Z'=RC_0Z + C_1Z$, where
$$
C_0= \begin{pmatrix}
         c_0 &0\\
         0& \hat v^2+ \frac{\hat v^2}{R}
\Big(\frac{(B_1^*)^2}{\mu_0}+\lambda\Big)
    \end{pmatrix},
\: \;
c_0= \beta-\theta x =
    \begin{pmatrix}
        0 & 1 & 0 \\
        \lambda \hat v & \hat v-(B_1^*)^2\hat v^2/\mu_0& 
\lambda (B_1^*)^2 \hat v^2/\mu_0 \\
        0 & -B_1^*\hat v^2 & \lambda \hat v^2 \\
    \end{pmatrix},
$$
and
$$
\beta=
    \begin{pmatrix}
        0 & 1 & 0 \\
        \lambda \hat v & \hat v& 0 \\
        0 & 0 &0 \\
    \end{pmatrix},
\qquad
x=
    \begin{pmatrix}
        0 & -B_1^*\hat v & \lambda \hat v
    \end{pmatrix},
\qquad 
C_1=
    \begin{pmatrix}
        0 & *\\
*&0
    \end{pmatrix}
=O(1).
$$

Applying the tracking lemma, we reduce to a decoupled system consisting
of a uniformly growing scalar equation 
\be\label{scalary}
y'=(R+ \frac{(B_1^*)^2}{\mu_0}+ \lambda) \hat v^2 y+O(1/R)y 
\ee
associated with the lower right
diagonal entry and a $3\times 3$ system
$$
 z'=c_0z+ O(1/R)z.
$$
For $|\lambda| \gg 1$, we may write
$c_0=
    \begin{pmatrix}
        0 & 1 & * \\
        \lambda \hat v & 0 & *\\
        0 & 0 & \lambda \hat v^2 \\
    \end{pmatrix}
+ O(1),
$
and apply the tracking lemma again to obtain three decoupled equations
uniformly growing/decaying at rates $\pm \sqrt{ \lambda \hat v}$
and $\lambda \hat v^2$, giving nonvanishing of the Evans function.
For $|\lambda |\le C$ on the other hand, we may apply the convergence
lemma, Lemma \ref{evanslimit}, 
using the fact that the $O(1/R)$ coefficient converges to its
limits as $CR^{-1}e^{-\eta |x|}$, $\eta>0$, together with 
Remark \ref{convrmk},
to obtain convergence to the unperturbed system $z'=c_0z$.
But, this may be recognized as exactly 
the formal limiting system \eqref{infsigmares} for ($\sigma=\infty$), which 
is stable by Theorem \ref{enprop}
(established by energy estimates).
Noting that the Evans function for the full system is the
product of the Evans functions of its decoupled components,
and that The Evans function for the scalar component 
converges likewise to that for
$y'=(R+ \frac{(B_1^*)^2}{\mu_0}+ \lambda) \hat v^2 y$,
or (by direct computation/exponentiation) 
$d(\lambda)=e^{c_0R +c_1 +c_2 \lambda }$ for constants
$c_j$,
we find, finally, that the full Evans function
after renormalization by factor $e^{-c_0R}$
converges to a constant multiple of the Evans function for 
\eqref{infsigmares}.
\end{proof}

\subsection{The limit as $\mu/(2\mu +\eta) \to 0$ or $\to \infty$}\label{ratio}

Finally, we briefly discuss the effect of dropping the gas-dynamical
assumption $\eta=-4\mu/3$, and considering more general values of 
$(2\mu+\eta)>0 $.
This parameter does not appear in the transverse equations, so enters
only indirectly to our analysis, through its effect on the gas-dynamical
profile $\hat v(x)$.
Specifically, denoting $r:=\mu/(2\mu +\eta) \to 0$, and taking as
usual the normalization $\mu=1$, we find that
$$\hat v(x)=\bar v(rx),$$
where $\bar v$ is a profile independent of the value of $r$.
Thus, the study in \cite{FT} of the limit $r\to 0$ is the limit
of slowly-varying coefficients, and the opposite limit $r\to \infty$
is the limit rapidly-varying coefficients.
We consider each of these limiting cases in turn.
Intermediate values of $\mu/(2\mu+\eta)$ would presumably need
to be studied numerically, an interesting direction
for further investigation.

\subsubsection{The $r\to 0$ limit}

In the $r\to 0$ limit, we have the following
result completing the analysis of \cite{FT}

\bpr\label{r0}
Parallel isentropic MHD shocks with ideal gas equation of state
are reduced Evans stable in the limit as $r\to 0$ with
other parameters held fixed.
\epr

\begin{proof} The case $B_1^*<2\sqrt{\mu_0}$ including Lax
$1$-type, overcompressive type, and some Lax $3$-type shocks 
has been established in \cite{FT} by energy estimates.
Thus, it suffices to treat the case of Lax $3$-shocks 
and (by Proposition \ref{hfthm}) bounded $|\lambda|$.

For shocks of any type, it is straightforward to verify that
the Evans function is nonvanishing on
$\lambda \in \{\Re \lambda\ge 0\}\setminus B(0,\eps)$, any $\eps>0$,
for $r$ sufficiently small.  
For, on this set of $\lambda$, 
there is a uniform spectral gap between the real parts of 
the stable and unstable eigenvalues of $A(x,\lambda)$,
for all $x\in (-\infty,+\infty)$, by the hyperbolic-parabolic
structure of the equations, similarly as in Lemma \ref{limlem1}.
It follows by standard matrix perturbation theory that there
exist matrices $R(\hat v(x))$ and $L=R^{-1}$ such that
$$
LAR= D:=\bp M&0\\0&N\ep,
$$
with $\Re M\ge \theta>0$ and $\Re N\le -\theta<0$.
Making the change of coordinates $W=RZ$, we obtain the
approximately block-diagonal equations
\be\label{zeq}
Z'=\tilde A Z,
\ee
where 
\be\label{form}
\tilde A:= LAR-L'R=D+ O(\hat v_x)=
D+ O(r \bar v_x).
\ee
Using the tracking/reduction lemma, Lemma \ref{reduction}, 
we find that there exist analytic functions $z_2=\Phi_2(z_1)=O(r)$ 
and $z_1=\Phi_1(z_2)=O(r)$ such that $(z_1, \Phi_2(z_1)$
and $(\Phi_1(z_2),z_2)$ are invariant under the flow of \eqref{zeq},
hence represent decoupled stable and unstable manifolds of the flow.
But, this implies that the Evans function is nonvanishing on
$\lambda \in \{\Re \lambda\ge 0\}\setminus B(0,\eps)$, any $\eps>0$,
for $r$ sufficiently small.  See \cite{ZH,MaZ3,Z1} for similar arguments.

Now, restrict to the case of a Lax $3$-shock for
$\lambda \in \{\Re \lambda\ge 0\}\cap B(0,\eps)$ and 
$\eps>0$ sufficiently small.
By examination of $A(x,\lambda)$ at $\lambda=0$ in the Lax
$3$-shock case, we find
that on $B(0,\eps)$ it has one eigenvalue $\mu_+$ that is uniformly
negative, one eigenvalue $\mu_-$
that is uniformly positive, and two that are small.
By standard matrix perturbation theory \cite{MaZ3,Z1}, there 
exist matrices $L=\bp L_+\\ L_0\\ L_-\ep$, 
$R=\bp R_+& R_0&R_-\ep$
with $LR\equiv I$ and $L_j'R_j\equiv 0$ such that
$$
LAR(x,\lambda)=\bp  \mu_+ & 0 & 0\\ 0 & \lambda M_0 & 0\\ 0 & 0 & \mu_-\ep,
$$
where the crucial factor $\lambda$ in $\lambda M_0$ is found by 
explicit computation/Taylor expansion \cite{ZH,MaZ3,Z1,MeZ1},
and $M_0=-\beta^{-1}+O(\lambda)$, where $\beta$ as in
\eqref{betapm} is the hyperbolic convection matrix
\be\label{beta}
\beta:=
\bp
1& -B_1^*/\mu_0 \hat v\\
-B_1^* & 1
\ep.
\ee
Moreover, $R$, $L$ depend only on $\hat v$, $\lambda$.
Making the change of coordinates $Z:=LW$, we obtain
$Z'=B(x,\lambda)Z$, where 
$$
B=LAR-L'R
= \bp  \mu_+ & O(\hat v_x) & O(\hat v_x)\\ O(\hat v_x) & \lambda M_0 & O(\hat v_x)\\ O(\hat v_x) & O(\hat v_x) & \mu_-\ep.
$$
Applying the tracking/reduction lemma again, we reduce to three decoupled
equations associated with the three diagonal blocks.  The two scalar
equations associated with $\mu_\pm$ are uniformly growing/decaying,
so do not support nontrivial decaying solutions at both infinities.
Thus, vanishing of the Evans function reduces to vanishing or nonvanishing
on the central block
$$
w'=(\lambda M_0+ O(\hat v_x^2))w,
$$
$w\in \CC^2$.
Noting that $\|\hat v_x^2\|_{L^1}=O(r)\to 0$, we may apply the
convergence lemma, Lemma \ref{evanslimit}, together with Remark 
\ref{convrmk}, 
to reduce finally to
\be\label{csys}
z'=\lambda M_0z,  \quad M_0:=L_0 A R_0.
\ee

For $|\lambda| \ll r$, we have $|\lambda M_0 - \lambda M_0(+\infty)|
\le C\lambda e^{-\theta r}$, hence 
$$
\|\lambda M_0 - \lambda M_0(+\infty)\|_{L^1[0,+\infty)}
=O(\lambda/r)\to 0
$$
and we may apply the conjugation lemma to obtain that the Evans
function for the reduced central system \eqref{csys}
is given by $(1+O(\lambda/r))\det(r_-,r_+)$,
where $r_-$ is an unstable eigenvalue of $M_0(-\infty)$ and
$r_+$ is a stable eigenvalue of $M_0(-\infty)$. 
Noting that these to order $\lambda$ are stable/unstable 
eigenvectors of $\beta_\pm$, we find by direct computation that
the determinant does not vanish.  Indeed, this is exactly the computation that
the Lopatinski determinant does not vanish for $3$-shocks.
Thus, we may conclude that the Evans function does not vanish for
$|\lambda| \ll r$.

Finally, we consider the remaining case $r/C\le |\lambda|\le Cr$, for $C>0$ 
large but fixed.
In this case, we may for the same reason drop terms of order $\lambda^2$
in the expansion of $\lambda M_0$, to reduce by an application of the
convergence lemma, Lemma \ref{evanslimit}, 
and Remark \ref{convrmk},
to consideration of the explicit system
$z'=-\lambda \beta^{-1}z$, which is exactly the {\it inviscid system}
$$
\lambda \bp w\\\alpha\ep'+
\bp
1& -B_1^*/\mu_0 \hat v\\
-B_1^* & 1
\ep
\bp w\\
\alpha
\ep'=0.
$$
But, this may be shown stable by an energy estimate as in the
case $\sigma=0$.
Thus, we conclude that the
Evans function does not vanish either for $|\lambda| \sim r$
and $\Re \lambda\ge 0$, completing the proof.
\end{proof}

\br\label{rrmk}
The Lax $1$-shock case may be treated by a similar but much
simpler argument, since growing and decaying modes decouple
into fast and slow modes.
The overcompressive case is nontrivial from this point of view,
since $\hat v$ passes through characteristic points as $x$ is varied.
However, we conjecture that the argument could be carried out in
this case by separating off the single uniformly fast mode and
treating the resulting $3$-dimensional system by an energy estimate
like that in the $\sigma=0$ or $\mu\to 0$ case.
\er

\subsubsection{The $r\to \infty$ limit}
The opposite limit $r\to \infty$ is that of rapidly-varying coefficients,
and is much simpler to carry out.  By the change of coordinates
$x\to x/r$, we reduce $\hat v(x)$ to a uniformly exponentially
decaying function $\bar v(x)$, and the coefficient matrix $A(x,\lambda)$
to a function $\bar A=r^{-1}A$ that decays to its limits as
$$
|\bar A(x,\lambda)-\bar A_\pm|\le Cr^{-1}e^{-\theta |x|}
\quad \hbox{\rm for} \; x\gtrless 0,
$$
where $\theta\ge \theta_0>0$.
Applying the convergence lemma, Lemma \ref{evanslimit}, 
together with Remark \ref{convrmk},
we obtain the following simple result.
\bpr\label{rinfty}
In the limit $r\to \infty$, the reduced Evans function $D^r$ converges
uniformly on compact subsets of $\Re \lambda\ge 0$ to
$D^0(\lambda)=\det(R^+,R^-)$, 
where $R^\pm$ are matrices solving Kato's ODE, whose columns
span the stable (resp. unstable) subspaces of $A_\pm$.
\epr
That is, determination of stability reduces to evaluation of
a purely linear algebraic quantity whose vanishing may be studied
without reference to the evolution of a variable-coefficient ODE.
This can be seen in the original coordinates by the formal
limit 
$$
A^r(x,\lambda)\to \begin{cases}
A_+(\lambda) & x>0,\\
A_-(\lambda) & x< 0.\\
\end{cases}
$$
We examine stability of $D^0$ numerically, as it does not
appear to be readily accessible analytically.


\section{Numerical Investigation}
\label{numerics}

In this section, we discuss our approach to Evans function computation, which is used to determine whether any unstable eigenvalues exist in our system, particularly in the intermediate parameter range left uncovered by our analytical results in Section \ref{sec:analytical}.  Our approach follows the polar-coordinate method developed in \cite{HuZ2}; see also \cite{BHRZ,HLZ,HLyZ1,Hu3,CHNZ}.  Since the Evans function is analytic in the region of interest, we can numerically compute its winding number in the right-half plane around a large semicircle $B(0,\Lambda)\cap \{\Re \lambda \ge 0\}$ containing \eqref{est2}, thus enclosing all possible unstable roots.  This allows us to systematically locate roots (and hence unstable eigenvalues) within.  As a result, spectral stability can be determined, and in the case of instability, one can produce bifurcation diagrams to illustrate and observe its onset.  This approach was first used by Evans and Feroe \cite{EF} and has been applied to various systems since; see for example \cite{PSW,AS,Br2,BDG}.

\subsection{Approximation of the profile}
\label{profnum}

Following \cite{BHRZ,HLZ}, we can compute the traveling wave profile using one of \textsc{MATLAB}'s boundary-value solvers {\tt bvp4c} \cite{SGT}, {\tt bvp5c} \cite{KL}, or {\tt bvp6c} \cite{HM}, which are adaptive Lobatto quadrature schemes and can be interchanged for our purposes.  These calculations are performed on a finite computational domain $[-L_-,L_+]$ with projective boundary conditions $M_\pm (U-U_\pm)=0$.  The values of approximate plus and minus spatial infinity $L_\pm$ are determined experimentally by the requirement that the absolute error $|U(\pm L_\pm)-U_\pm|$ be within a prescribed tolerance, say $TOL=10^{-3}$; see \cite[Section 5.3.4]{HLyZ1} for a complete discussion.  Throughout much of the computation, we used $L_\pm=\pm 20$, but for some rather extreme values in our parameter range, we had to lengthen our interval to maintain good error bounds.

\subsection{Approximation of the Evans function}\label{evansnum} 

Throughout our numerical study, we used the polar-coordinate method described in \cite{HuZ2}, which encodes $\cW=\rho\,\Omega$, where ``angle'' $\Omega=\omega_1\wedge \cdots \wedge \omega_k$ is the exterior product of an orthonormal basis $\{\omega_j\}$ of $\Span \{W_1, \dots, W_k\}$ evolving independently of $\rho$ by some implementation (e.g., Drury's method) of continuous orthogonalization and ``radius'' $\rho$ is a complex scalar evolving by a scalar ODE slaved to $\Omega$, related to Abel's formula for evolution of a full Wronskian; see \cite{HuZ2} for further details.  This might be called ``analytic orthogonalization'', as the main difference from standard continuous orthogonalization routines is that it restores the important property of analyticity of the Evans function by the introduction of the radial function $\rho$ ($\Omega$ by itself is not analytic); see \cite{HuZ2,Z5} for a discussion on this method.

\subsubsection{Shooting and initialization}

The ODE calculations for individual $\lambda$ are carried out using \textsc{MATLAB}'s {\tt ode45} routine, which is the adaptive 4th-order Runge-Kutta-Fehlberg method (RKF45).  This method is known to have excellent accuracy with automatic error control.  Typical runs involved roughly $300$ mesh points per side, with error tolerance set to {\tt AbsTol = 1e-6} and {\tt RelTol = 1e-8}.  

To produce analytically varying Evans function output, the initial data $\cV(-L_-)$ and $\widetilde{\cV}(L_+)$ must be chosen analytically using \eqref{Kato}.  The algorithm of \cite{BrZ} works well for this purpose, as discussed further in \cite{BHRZ,HuZ2}.

\subsubsection{Winding number computation}
\label{windingalg}
We compute the winding number by varying values of $\lambda$ around the semicircle $B(0,\Lambda)\cap \{\Re \lambda \ge 0\}$ along $120$ points of the contour, with mesh size taken quadratic in modulus to concentrate sample points near the origin where angles change more quickly,
and summing the resulting changes in ${\rm arg}(D(\lambda))$, using $\Im \log D(\lambda) = {\rm arg} D(\lambda) ({\rm mod} 2\pi)$, available in \textsc{MATLAB} by direct function calls.
As a check on winding number accuracy, we test a posteriori that the change in argument of $D$ for each step is less than $0.2$, and add mesh points, as necessary to achieve this.  Recall, by Rouch\'e's Theorem, that accuracy is preserved so long as the argument varies by less than $\pi$ along each mesh interval.

\subsection{Description of experiments: broad range}

In our first numerical study, we covered a broad intermediate parameter range to demonstrate stability in the regions not amenable to our analytical results in Section \ref{sec:analytical}, and also to close our study for unconditional stability for all (finite) system parameters.  Since Evans function computation is essentially ``embarrassingly parallel'', we were able to adapt our STABLAB code to take advantage of MATLAB's parallel computing toolbox, sending to each of 8 ``workers'' on our 8-core Power Macintosh workstation, different values of $\lambda$ producing a net speedup of over 600\%.  The following parameter combinations were examined:
\begin{align*}
(\gamma,v_+,B^*_1, \mu_0,\sigma) &\in \{1.0, 1.1, 11/9, 9/7, 7/5, 5/3, 1.75, 2.0, 2.5, 3.0\}\\
&\quad \times \{0.8, 0.6, 0.4, 0.2, 10^{-1}, 10^{-2}, 10^{-3}, 10^{-4}, 10^{-5} \}\\
&\quad \times \{0.2, 0.8, 1.4, 2.0, 2.6, 3.2, 3.8\}\\
&\quad \times \{0.2, 0.8, 1.4, 2.0, 2.6, 3.2, 3.8\}\\
&\quad \times \{0.2, 0.8, 1.4, 2.0, 2.6, 3.2, 3.8\}.
\end{align*}
In total, this is $30,\!870$ contours, each consisting of at least 120 points in $\lambda$.  In all cases, we found the system to be Evans stable.  Typical output is given in Figure \ref{vplimit}.

We remark that the Evans function is symmetric under reflections along the real axis (conjugation).  Hence, we only needed to compute along half of the contour (usually 60 points in the first quadrant) to produce our results.

\subsection{Description of experiments: limiting parameters}

The purpose of our second study is to verify convergence in the large-amplitude limit ($v_+\rightarrow 0$), as well as illustrate the analytical results the limiting cases, namely as $B^*_1\rightarrow \infty$, $B^*_1\rightarrow 0$, $\mu_0\rightarrow\infty$, $\mu_0\rightarrow 0$, $\sigma\rightarrow\infty$, $\sigma\rightarrow 0$, $r\rightarrow\infty$, and $r\rightarrow 0$.  In all cases, we found our results to be consistent with stability.

In Table \ref{converge_table}, we provide typical relative errors between the normalized and limiting-normalized Evans functions in the large-amplitude limit; we varied $B_1^*$ for illustrative purposes.  The relative errors are given by computing, respectively,
\[
\max_j\left|\frac{\hat D(\lambda_j)-\hat D^0(\lambda_j)}{\hat D^0(\lambda_j)}\right|\quad\mbox{and}\quad\max_j\left|\frac{\check D(\lambda_j)-\check D^0(\lambda_j)}{\check D^0(\lambda_j)}\right|
\]
along the contours except for small $\lambda$ (that is, when $|\lambda| < 10^{-2}$).  Note that in the large-amplitude limit, the relative errors go to zero, as expected.

\begin{table}[!t]
\begin{tabular}{|c||c|c|c|c|c|c|c|}
\hline
$v_+$&$B_1^*=0.2$&$B_1^*=0.8$&$B_1^*=1.4$&$B_1^*=2$&$B_1^*=2.6$&$B_1^*=3.2$&$B_1^*=3.8$\\
\hline
\hline
10(-1)&9.94(-1)&1.23&3.46&9.33&2.16(1)&4.89(1)& 1.09(2)\\
\hline
10(-2)&4.36(-1)&5.19(-1)&1.36&2.82&4.92&8.19&1.32(1)\\
\hline
10(-3)&1.42(-1)&1.72(-1)&4.50(-1)&8.34(-1)&1.25&1.86&2.53\\
\hline
10(-4)&4.23(-2)&5.04(-2)&1.32(-1)&2.30(-1)&3.23(-1)&4.55(-1)&5.88(-1)\\
\hline
10(-5)&1.26(-2)&1.50(-2)&4.00(-2)&6.83(-2)&9.35(-2)&1.28(-1)&1.61(-1)\\
\hline
10(-6)&3.94(-3)&4.77(-3)&1.28(-2)&2.18(-2)&2.96(-2)&4.03(-2)&5.01(-2)\\
\hline
10(-7)&2.16(-3)&2.62(-3)&7.08(-3)&1.20(-2)&1.63(-2)&2.21(-2)&2.75(-2)\\
\hline
10(-8)&2.07(-3)&2.51(-3)&6.78(-3)&1.15(-2)&1.56(-2)&2.12(-2)&2.63(-2)\\
\hline
\end{tabular}
\caption{Relative errors for $\check D(\lambda)$ and $\hat D(\lambda)$. Here $\sigma=\mu_0=0.8$ and $\gamma=5/3$.}
\label{converge_table}
\end{table}


\appendix

\section{The convergence and tracking lemmas}\label{s:contrack}

\subsection{The convergence lemma}\label{s:conv}
Consider a family of first-order equations 
\be \label{gen2}
W'=A^p(x,\lambda)W
\ee
indexed by a parameter $p$, and satisfying exponential
convergence condition \eqref{udecay} uniformly in $p$.
Suppose further that
\begin{equation} \label{residualest}
|(A^p- A^p_\pm)-
(A^0- A^0_\pm)|
\le C|p|e^{-\theta |x|}, \qquad \theta>0
\end{equation}
and
\begin{equation} \label{newest}
|( A^p- A^0)_\pm)| \le C|p|.
\end{equation}
Then, we have the following generalization of Lemma \ref{conjlem},
a simplified version of the {\it convergence lemma} of \cite{PZ}.

\begin{lemma} \label{evanslimit}
Assuming \eqref{udecay} and \eqref{residualest}--\eqref{newest},
for $|p|$ sufficiently small,
there exist invertible linear transformations 
$P_+^p(x,\lambda)=I+\Theta_+^p(x,\lambda)$ 
and $P_-^0(x,\lambda) =I+\Theta_-^p(x,\lambda)$ defined
on $x\ge 0$ and $x\le 0$, respectively,
analytic in $\lambda$ as functions into $L^\infty [0,\pm\infty)$, such that
\begin{equation}
\label{cPdecay} 
| (P^p-P^0)_\pm |\le C_1 |p| e^{-\bar \theta |x|}
\quad
\text{\rm for } x\gtrless 0,
\end{equation}
for any $0<\btheta<\theta$, some $C_1=C_1(\bar \theta, \theta)>0$,
and the change of coordinates $W=:P_\pm^p Z$ reduces \eqref{gen2} to 
\begin{equation}
\label{cglimit}
Z'=A^p(x,\lambda) Z 
\quad
\text{\rm for } x\gtrless 0.
\end{equation}
\end{lemma}

\begin{proof}
Applying the conjugating transformation $W\to
(P^0_+)^{-1}W$ for the $p=0$ equations, we may reduce to the
case that $A^0$ is constant, and $P^0_+\equiv I$, noting that
the estimate \eqref{residualest} persists under well-conditioned
coordinate changes $W=QZ$, $Q(\pm \infty)=I$,
transforming to 
\be\label{calc1}
\begin{aligned}
&|\big(Q^{-1}A^pQ-Q^{-1}Q'-A^p_\pm\big) -
\big(Q^{-1}A^0Q-Q^{-1}Q'-A^0_\pm\big)|\\
&\qquad \le |Q\big((A^p-A^p_\pm)-(A^0-A^0_\pm)\big)Q^{-1}|
+ |Q^{-1}(A^p-A^0)_\pm Q-(A^p-A^0)_\pm |,
\\
\end{aligned}
\ee
where
\be\label{calc2}
|Q^{-1}(A^p-A^0)_\pm Q-(A^p-A^0)_\pm |= O(|Q-I|)|(A^p-A^0)_\pm|
= O(e^{-\theta|x|})|p|.
\ee
In this case,
\eqref{residualest} becomes just
$$ |A^p- A^p_\pm| \le
C_1|p|e^{-\theta |x},
$$
and we obtain directly from the conjugation lemma, Lemma
\ref{conjlem}, the estimate
$$
|P^p_+ - P^0_+|=
|P^p_+ - I|\le CC_1|p| e^{-\bar \theta|x|}
$$
for $x>0$, and similarly for $x<0$, verifying the result.\footnote{The
inclusion of assumption \eqref{newest}, needed in
\eqref{calc2}, repairs a minor omission in \cite{PZ}.
(It is satisfied for the applications in \cite{PZ}, but is not
listed as a hypothesis.)}
\end{proof}

\br\label{varconj}
In the case $A^p_\pm \equiv \const$, or, equivalently, for which
\eqref{residualest} is replaced by $ |A^p-A^0| \le C_1|p|e^{-\theta |x}, $
we find that the change of coordinates
$W=\tilde P^p_\pm Z$, $\tilde P^p_\pm :=(P^0)_\pm ^{-1}P_\pm ^p$, 
converts \eqref{gen2}
to $Z'=A^0Z$, where $\tilde P^p_\pm =I+\tilde \Theta^p_\pm$ with
\be\label{altconv}
|\tilde \Theta^p_\pm|\le CC_1|p| e^{-\bar \theta|x|}.
\ee
That is, we may conjugate not only to constant-coefficient equations,
but also to exponentially convergent variable-coefficient equations,
with sharp rate \eqref{altconv}.
\er

\br\label{convrmk}
As observed in \cite{PZ},
provided that the stable/unstable subspaces of $A^p_+$/$A^p_-$
converge to those of $A^0_+$/$A^0_-$, as typically holds given
\eqref{newest}-- in particular, this holds by standard matrix
perturbation theory \cite{Kato} if the stable and unstable eigenvalues
of $A^0_\pm$ are spectrally separated-- 
\eqref{cPdecay} gives immediately convergence
of the Evans functions $D^p$ to $D^0$ on compact sets of $\lambda$,
by definition \eqref{eq:evans}.
\er

\subsection{The tracking lemma}\label{s:track}
Consider an approximately block-diagonal system
\begin{equation}
W'= \bp M_1 & 0 \\ 0 & M_2 \ep(x,p) + \delta(x,p) \Theta(x,p) W,
\label{blockdiag}
\end{equation}
where $\Theta$ is a uniformly bounded matrix, $\delta(x)$ scalar, 
and $p$ a vector of parameters,
satisfying a pointwise spectral gap condition
\begin{equation}
\min \sigma(\Re M_1^\varepsilon)- \max \sigma(\Re M_2^\varepsilon)
\ge \eta(x)
\, \text{\rm for all } x.
\label{gap}
\end{equation}
(Here as usual $\Re N:= (1/2)(N+N^*)$ denotes the
``real'', or symmetric part of $N$.)
Then, we have the following
{\it tracking/reduction lemma} of \cite{MaZ3,PZ}.

\begin{lemma}[\cite{MaZ3,PZ}] \label{reduction}
Consider a system \eqref{blockdiag} under the gap assumption
\eqref{gap}, with $\Theta^\varepsilon$ uniformly bounded and
$\eta\in L^1_{\rm loc}$.
If $\sup (\delta/\eta)(x)$ is sufficiently small,
then there exist (unique) linear
transformations $\Phi_1(x,p)$ and
$\Phi_2(x,p)$, possessing the same
regularity with respect to $p$
as do coefficients $M_j$ and
$\delta\Theta$, for which the graphs $\{(Z_1,
\Phi_2 Z_1)\}$ and $\{(\Phi_1(Z_2),Z_2)\}$ are
invariant under the flow of \eqref{blockdiag}, and satisfy
\be\label{Phibd}
\sup|\Phi_1|, \, \sup|\Phi_2| \le C \sup(\delta/\eta) 
\ee
and
\ba\label{ptwise}
|\Phi^\varepsilon_1(x)|&\le
C\int_x^{+\infty} e^{\int_y^x \eta(z)dz} \delta(y)dy,
\qquad
|\Phi^\varepsilon_1(x)| \le C
\int_{-\infty}^{x} e^{\int_y^x -\eta(z)dz} \delta(y)dy.
\ea
\end{lemma}

\begin{proof}
By the change of coordinates $x\to \tilde x$, $\delta \to \tilde \delta:=
\delta/\eta$ with
$d\tilde x/dx=\eta(x)$,
we may reduce to the case $\eta\equiv {\rm constant}= 1$ treated in \cite{MaZ3}.
Dropping tildes and setting $\Phi_2:= \psi_2\psi_1^{-1}$, where 
$(\psi_1^t,\psi_2^t)^t$ satisfies \eqref{blockdiag}, 
we find after a brief calculation that $\Phi_2$ satisfies
\be
\Phi_2'= 
(M_2 \Phi_2 - \Phi_2 M_1) + \delta Q(\Phi_2),
\label{Phieqn}
\ee
where $Q$ is the quadratic matrix polynomial
$
Q(\Phi):=
\Theta_{21} + \Theta_{22}\Phi - \Phi \Theta_{11} + \Phi \Theta_{12} \Phi.
$
Viewed as a vector equation, this has the form
\be
 \Phi_2'= \cM \Phi_2 + \delta Q(\Phi_2),
\label{vectoreqn}
\ee
with linear operator
$\cM \Phi:= M_2 \Phi - \Phi M_1$.
Note that a basis of solutions of the decoupled equation
$ \Phi'= \cM \Phi$
may be obtained as the tensor product $\Phi=\phi \tilde \phi^*$
of bases of solutions of $\phi'=M_2 \phi$ and
$\tilde \phi'= -M_1^* \tilde \phi$, whence we obtain from
\eqref{gap} 
\be
e^{\cM z}\le Ce^{-\eta z}, \quad \hbox{\rm for }\; z>0,
\label{expbd}
\ee
or uniform exponentially decay in the forward direction.

Thus, assuming only that $\Phi_2$ is bounded at $-\infty$, we obtain
by Duhamel's principle the integral fixed-point equation
\be
\Phi_2(x)= \CalT \Phi_2(x):=
\int_{-\infty}^x e^{\cM (x-y)} \delta(y)Q(\Phi_2)(y)
\,dy.
\label{inteqn}
\ee
Using \eqref{expbd}, we find that $\CalT$ is a contraction
of order $O(\delta/\eta)$, hence \eqref{inteqn} determines
a unique solution for $\delta/\eta$ sufficiently small, which,
moreover, is order $\delta/\eta$ as claimed.
Finally, substituting $Q(\Phi)=O(1+|\Phi|)=O(1)$ in 
\eqref{inteqn}, we obtain
$$
|\Phi_2(x)|\le 
C\int_{-\infty}^x e^{\eta (x-y)} \delta(y)
\,dy
$$
in $\tilde x$ coordinates, or, in the original $x$-coordinates,
\eqref{ptwise}.
A symmetric argument establishes existence of $\Phi_1$ with 
the asserted bounds.
Regularity with respect to parameters is inherited as usual
through the fixed-point construction via the Implicit Function Theorem.
\end{proof}

\br
For $\eta$ constant and $\delta$ decaying at exponential rate
strictly slower that $e^{-\eta x}$ as $x\to +\infty$, 
we find from \eqref{ptwise} that $\Phi_2(x)$ decays like $\delta/\eta$
as $x\to +\infty$,
while if $\delta(x)$ merely decays monotonically as $x\to -\infty$, we
find that $\Phi_2(x)$ decays like $(\delta/\eta)$ as $x\to -\infty$,
and symmetrically for $\Phi_1$.
This and \eqref{ptwise} is a 
slight addition to the statement of \cite{MaZ3,PZ}, which did not
include pointwise information.
We will not need this observation here, but record it for general 
reference/completeness.
\er

\br
A closer look at the proof of Lemma \ref{reduction} shows
that, in the approximately block lower-triangular case,
$\delta \Theta_{21}$ not necessarily small,
there exists a {\rm block-triangularizing} transformation
$\Phi_2=O(\sup|\delta/\eta|)<<1$, under the much less restrictive conditions
$$
\sup \Big(|\delta/\eta|(| \Theta_{11}| +| \Theta_{22}|) \Big)
<1
\; \hbox{\rm and } \;
\sup(|\delta/\eta| |\Theta_{21}|) << \frac{1}{\sup|\delta/\eta|}.
$$
(We do not use this here, but remark it for general application.)
\er
%

\section{Miscellaneous energy estimates}\label{misc}

\subsection{Stability for $B_1^*=0$}\label{Bzeroest}

\bpr\label{smallprop}
Parallel ideal gas MHD shocks are stable for $B_1^*=0$
provided that the associated gas-dynamical shock is stable.
\epr

\begin{proof}
For $B_1^*=0$, the eigenvalue equations become
\begin{equation}\label{noB}
\begin{aligned}
 \lambda u + u'  &=\mu u''/\hat v,\\
 \lambda \alpha  + \alpha' &=(1/\sigma \mu_0 \hat v)
(\alpha'/\hat v)',\\
\end{aligned}
\end{equation}
or
\begin{equation}\label{noBres}
\begin{aligned}
 \lambda \hat v u + \hat v u'  &=\mu u'',\\
 \lambda \hat v \alpha  + \hat v \alpha' &=(1/\sigma \mu_0 )
(\alpha'/\hat v)'.\\
\end{aligned}
\end{equation}
Taking the real part of the complex $L^2$-inner product of 
$u$ against the first equation and
$\alpha$ against the second equation and summing gives
$$
\Re \lambda( \int \hat v (|u|^2+|\alpha|^2) =
-\int (\mu |u'|^2 + (1/\sigma \mu_0 \hat v)|\alpha'|^2)
+\int \hat v_x (|u|^2+|\alpha|^2) <0,
$$
a contradiction for $\Re \lambda \ge 0$ and $u$, $\alpha$
not identically zero.  Thus, we obtain spectral stability 
in transverse fields $(\tilde u, \tilde B)$
for $B_1^*=0$
so long as the profile density is decreasing
$\hat v_x<0$, as holds in particular for the ideal gas case,
either isentropic or nonisentropic.
Likewise, transversality and inviscid stability criteria are
easily verified in this case by the further decoupling of
$\tilde u$ and $\tilde B$ equations.
Stability in the decoupled parallel fields $(v,u_1)$ 
is of course equivalent to stability of the corresponding gas-dynamical
shock.
\end{proof}

\br
By continuity, we obtain from the above
also stability for magnetic field $B_1^*$ sufficiently small.
Stability for small magnetic field was already observed in \cite{GMWZ5,GMWZ6},
by a similar continuity argument.
\er

\subsection{Stability for infinite $\mu_0$}\label{infmu0}

\begin{proof}[Proof of Theorem \ref{enprop}, case $\mu_0$]
For $\mu_0=\infty$, equations \eqref{eval} become 
\begin{equation}\label{infsigma}
\begin{aligned}
 \lambda w + w'  &=\mu w''/\hat v,\\
 \lambda \alpha  + \alpha' - B_1^*w' &=0,
\end{aligned}
\end{equation}
hence the $w$ equation decouples and is
stable by the argument for $B_1^*=0$.  Thus, $w\equiv 0$ for
$\Re \lambda\ge 0$, and so the second equation reduces to
a constant-coefficient equation
$
 \lambda \alpha  + \alpha'  =0,
$
and thus is stable.
\end{proof}

\subsection{Transversality for large $B_1^*$}\label{sec:transest}

\begin{proposition}\label{transest}
For $B_1^*\ge \sqrt{\mu_0} + \max\Big\{ \sqrt{ \frac{\gamma \mu_0}{2}} ,
\sqrt{\frac{\gamma }{2\sigma}} \, \Big\} $, 
and all $1\ge v_+>0$,
profiles (necessarily Lax $3$-shocks) are transverse.
\end{proposition}

\begin{proof}
For $\mu=1$, the (transverse part of the) linearized traveling-wave ODE
is
\be\label{cr}
\hat v^{-1}
\bp
\mu_0  &0\\
0& 1/\sigma \mu_0
\ep
\bp \tilde u\\ \tilde B \ep'=
\bp
\mu_0 & -B_1^*\\
-B_1^*&\hat v 
\ep
\bp \tilde u\\ \tilde B \ep.
\ee
Transversality is equivalent to nonexistence of
a nontrivial $L^2$ solution of \eqref{cr}.
Taking the real part of the complex $L^2$ inner product of
$$
\hat v
\bp
\mu_0  &0\\
0& 1/\sigma \mu_0
\ep^{-1}
\bp
\mu_0 & -B_1^*\\
-B_1^*&\hat v 
\ep
\bp \tilde u\\ \tilde B \ep
$$
against both sides of \eqref{cr}, noting that
$$
\Re \Big\langle
\bp \tilde u\\ \tilde B \ep,
\bp
\mu_0 & -B_1^*\\
-B_1^*&\hat v 
\ep
\bp \tilde u\\ \tilde B \ep'
\Big\rangle= -\int \frac{\hat v_x}{2}|\tilde B|^2,
$$
and estimating 
$|\frac{\hat v_x}{2}|\le \frac{\gamma \hat v}{2}$
(see \cite{HLZ}, Appendix A for similar estimates),
we obtain
$$
\Big\langle
\bp \tilde u\\ \tilde B \ep,
\hat v (M-N)
\bp \tilde u\\ \tilde B \ep
\Big\rangle
\le 
0
$$
where
$N:=
 \bp
0&0\\
0& \frac{\gamma}{2}
\ep
$
and
$$
M:=
\bp
\mu_0 & -B_1^*\\
-B_1^*& \hat v
\ep
\bp
\mu_0  &0\\
0& 1/\sigma \mu_0
\ep^{-1}
\bp
\mu_0 & -B_1^*\\
-B_1^*&\hat v 
\ep,
$$
is positive definite for $B_1^*>\sqrt{\mu_0}$.
The first minor of $(M-N)$ is equal to the first minor
of $M$, so positive for $B_1^*>\sqrt{\mu_0}$.
Thus, $M-N>0$, giving a contradiction, if
$B_1^*>\sqrt{\mu_0}$ and 
\be\label{cont}
0<\det(M-N)= \sigma \Big(\mu_0 \hat v-(B_1^*)^2\Big)^2
-
\Big(\mu_0+\sigma \mu_0(B_1^*)^2\Big)\frac{\gamma}{2}
\ee
for all $1\ge \hat v\ge v_+ \ge 0$.
Estimating
$$
\Big(\mu_0 \hat v-(B_1^*)^2\Big)^2=
(B_1^*-\sqrt{\mu_0} )^2 (B_1^*+\sqrt{\mu_0} )^2
=
(B_1^*-\sqrt{\mu_0} )^2 \Big((B_1^*)^2+\mu_0 \Big)
$$
we find that \eqref{cont} holds for
$
(B_1^*-\sqrt{\mu_0} )^2 \ge \max\{\frac{\gamma\mu_0}{2},
\frac{\gamma}{2\sigma}\},
$
yielding the result.
\end{proof}

\br
What makes this argument work is the strong separation as $B_1^*\to \infty$
of growing and decaying modes, as evidenced by strong hyperbolicity
of the coefficient matrix on the righthand side of \eqref{cr}.
It could be phrased alternatively in terms of the tracking lemma
of Appendix \ref{s:track}.  
Also related are the ``transverse'' estimates of \cite{Go1,Go2}.
\er

\subsection{High-frequency bounds}

\begin{proof}[Proof of \ref{hfthm}]
Multiplying the first equation of \eqref{eval} by $\hat v \bar w$, 
integrating in $x$ along $\mathbb{R}$, and simplifying gives
$$
\lambda \ip \hat v |w|^{2} + \ip \hat v w'\bar w + \mu \ip |w'|^{2} = \frac{B_1^*}{\mu_{0}}\ip \alpha' \bar w.
$$
Taking the real and imaginary parts, respectively, gives
\begin{equation}
\label{real1}
\Re\lambda\ip\hat v |w|^{2} - \frac{1}{2}\ip\hat v_{x}|w|^{2} + \mu\ip |w'|^{2} = \frac{B_1^*}{\mu_{0}}\Re \ip \alpha'\bar w
\end{equation}
and
\begin{equation}
\label{imag1}
\Im\lambda\ip\hat v |w|^{2} + \Im \ip\hat v w' \bar w  = \frac{B_1^*}{\mu_{0}}\Im\ip \alpha'\bar w.
\end{equation}
Adding and simplifying, noting that $\Re  z + |\Im z| \leq \sqrt{2} |z|$ 
and  $\hat v_{x} < 0$, yields
$$
(\Re \lambda + |\Im \lambda|) \ip \hat v |w|^{2} + \mu\ip |w'|^{2} < \ip \hat v |w'| |w| + \frac{\sqrt{2}B_1^*}{\mu_{0}}\ip |\alpha'| |w|.
$$
Using Young's inequality, and noting that $\hat v_x\le 0$ and $\hat v \leq 1$, 
we have
\begin{equation}
\begin{aligned} \label{young1}
(\Re\lambda + |\Im \lambda|)&\ip \hat v |w|^{2} + 
\mu\ip |w'|^{2} < \left(\varepsilon_{1} +
\varepsilon_{2}\frac{\sqrt{2}B_1^*}{\mu_{0}}\right) \ip \hat v |w|^{2} \\
& + \frac{1}{4\varepsilon_{1}}\ip |w'|^{2} + 
\frac{\sqrt{2}B_1^*}{4 \varepsilon_{2} \mu_{0}}\ip \frac{|\alpha'|^{2}}{\hat v}.
\end{aligned}
\end{equation}
Multiplying the second equation of \eqref{eval}
by $\hat v \bar \alpha$, integrating in $x$ 
along $\mathbb{R}$, and simplifying gives
$$
\lambda \ip \hat v |\alpha|^{2} + \ip \hat v \alpha'\bar \alpha + \frac{1}{\sigma\mu_{0}} \ip \frac{|\alpha'|^{2}}{\hat v} = B_1^* \ip \hat{v} w' \bar \alpha.
$$
Taking the real and imaginary parts, respectively, gives
\begin{equation}
\label{real2}
\Re \lambda \ip \hat v |\alpha|^{2} -\frac{1}{2} \ip \hat v_{x} |\alpha|^{2}+ \frac{1}{\sigma\mu_{0}} \ip \frac{|\alpha'|^{2}}{\hat v} = B_1^* \Re \ip \hat{v} w' \bar \alpha
\end{equation}
and
\begin{equation}
\label{imag2}
\Im \lambda \ip \hat v |\alpha|^{2} + \Im \ip \hat v \alpha'\bar \alpha = B_1^* \Im\ip \hat{v} w' \bar \alpha.
\end{equation}
Adding and simplifying, again noting that $\Re z + |\Im z| \leq \sqrt{2} |z|$ and $\hat v_{x} \le  0$, yields
$$
(\Re \lambda + |\Im \lambda|) \ip \hat v |\alpha|^{2} + \frac{1}{\sigma\mu_{0}} \ip \frac{|\alpha'|^{2}}{\hat v} < \ip \hat v |\alpha| |\alpha'| + \sqrt{2} B_1^* \ip \hat{v} |w'| |\alpha|.
$$
Using Young's inequality, and noting that $\hat v \leq 1$, we have
\begin{equation}
\begin{aligned}
\label{young2}
&(\Re \lambda + |\Im \lambda|) \ip \hat v |\alpha|^{2} + \frac{1}{\sigma\mu_{0}} \ip \frac{|\alpha'|^{2}}{\hat v} < \left(\varepsilon_{3} + \sqrt{2} \varepsilon_{4} B_1^* \right) \ip \hat v |\alpha|^{2} \\
&\qquad + \frac{\sqrt{2}B_1^*}{4\varepsilon_{4}}\ip |w'|^{2} + \frac{1}{4 \varepsilon_{3}}\ip \frac{|\alpha'|^{2}}{\hat v}.\notag
\end{aligned}
\end{equation}
Adding $C$ $\times$ \eqref{young2} to \eqref{young1} yields
\begin{equation}
\begin{aligned}
&(\Re\lambda + |\Im \lambda|)\ip \hat v (|w|^{2} + C|\alpha|^{2}) + \mu\ip |w'|^{2} +  \frac{C}{\sigma\mu_{0}} \ip \frac{|\alpha'|^{2}}{\hat v}\\
&\qquad < \left(\varepsilon_{1} +\varepsilon_{2}\frac{\sqrt{2}B_1^*}{\mu_{0}}\right) \ip \hat v |w|^{2} + C \left(\varepsilon_{3} + \sqrt{2} \varepsilon_{4} B_1^* \right) \ip \hat v |\alpha|^{2}\\
&\qquad\quad + \left(\frac{1}{4\varepsilon_{1}}  + \frac{\sqrt{2}B_1^*C}{4\varepsilon_{4}}\right) \ip |w'|^{2} + \left( \frac{\sqrt{2}B_1^*}{4 \varepsilon_{2} \mu_{0}} + \frac{C}{4 \varepsilon_{3}} \right) \ip \frac{|\alpha'|^{2}}{\hat v}.
\end{aligned}
\end{equation}
By setting
$$
\varepsilon_{1} = \frac{1}{2\mu},\quad \varepsilon_{2} = \frac{B_1^*\sigma}{\sqrt{2}C},\quad\varepsilon_{3}=\frac{\mu_{0}\sigma}{2},\quad\mbox{and}\quad\varepsilon_{4}=\frac{B_1^*C}{\sqrt{2}\mu},
$$
this becomes
\begin{equation}
\begin{aligned}
(\Re\lambda + |\Im \lambda|)&\ip \hat v (|w|^{2} + C|\alpha|^{2})  \\
&<  \left( \frac{1}{2\mu}+\frac{(B_1^*)^{2}\sigma}{\mu_{0}C} \right) \ip \hat v |w|^{2} + C \left( \frac{\mu_{0}\sigma}{2}+\frac{(B_1^*)^{2}C}{\mu} \right) \ip \hat v |\alpha|^{2}.
\end{aligned}
\end{equation}
This inequality fails for all choices of $w,\alpha$, whenever
$$
\Re\lambda + |\Im \lambda| \geq \max\{ \frac{1}{2\mu}+\frac{(B_1^*)^{2}\sigma}{\mu_{0}C}, \frac{\mu_{0}\sigma}{2}+\frac{(B_1^*)^{2}C}{\mu} \}.
$$
Setting
\begin{equation}
\label{C}
C = \sqrt{\frac{\sigma\mu}{\mu_{0}}}
\end{equation}
yields the right-hand side of \eqref{est2}.
\end{proof}

\begin{corollary}
Any eigenvalue $\lambda$ of \eqref{eval} with nonnegative real part satisfies
\begin{equation}
\label{est1}
\Re \lambda < \frac{(B_1^*)^{2}}{4}\sqrt{\dfrac{\sigma}{\mu\mu_{0}}}.
\end{equation}
\end{corollary}

\begin{proof}
Adding \eqref{real1} to $C$ $\times$ \eqref{real2}, noting that $\hat v_{x}<0$, $\hat v\leq 1$, and using Young's inequality yields
\begin{equation}
\begin{aligned}
\Re \lambda & \ip\hat v (|w|^{2} + C|\alpha|^{2}) + \mu\ip |w'|^{2} + \frac{C}{\sigma\mu_{0}} \ip \frac{|\alpha'|^{2}}{\hat v}\\
 &< \frac{B_1^*}{\mu_{0}}\ip \left( \varepsilon_{1} \hat v |w|^{2} + \frac{1}{4 \varepsilon_{1}} \frac{|\alpha'|^{2}}{\hat v}\right) + B_1^*C \ip \left(\varepsilon_{2} \hat v |\alpha|^{2} +  \frac{1}{4\varepsilon_{2}} |w'|^{2}\right)
\end{aligned}
\end{equation}
Setting
$$
\varepsilon_{1} = \frac{B_1^*\sigma}{4 C}\quad\mbox{and}\quad\varepsilon_{2} = \frac{B_1^* C}{4\mu}
$$
together with \eqref{C}, yields the right-hand side of \eqref{est1}.
\end{proof}

\section{Kato basis near a branch point}\label{branch}

By straightforward computation, 
$\mu_\pm(\lambda):= \mp(\eta/2 +\sqrt{\eta^2/4 + \lambda}$
and $\mathcal{V}_\pm:=(1, \mu_\pm(\lambda))^T$
are eigenvalues and eigenvectors of 
the matrix $A$ of \eqref{egA} in Example \ref{sqrteg}. 
The associated Kato eigenvectors $V^\pm$
are determined uniquely, up to a constant factor
independent of $\lambda$, by the property that there
exist corresponding left eigenvectors $\tilde V^\pm$ such that
\begin{equation}
\label{katoprop2}
(\tilde V\cdot V)^\pm \equiv  {\rm constant}, \quad
(\tilde V \cdot \dot V)^\pm \equiv 0,
\end{equation}
where ``$\, \, \dot{    }\, \,$'' denotes $d/d\lambda$;
see Lemma \ref{katolem}(iii).

Computing dual eigenvectors $\tilde{\mathcal{V}}^\pm =
(\lambda+\mu^2)^{-1}(\lambda, \mu_\pm)$ satisfying
$(\tilde{\mathcal{V}}\cdot \mathcal{V})^\pm \equiv 1$, 
and setting $V^\pm=c_\pm\mathcal{V}^\pm$, $\tilde V^\pm=\mathcal{V}^\pm/c_\pm$,
we find after a brief calculation that \eqref{katoprop2} 
is equivalent to the complex ODE
\begin{equation}\label{kato_ode}
\begin{aligned}
\dot c_\pm &= -\Big( \frac{ 
\tilde V \cdot \dot V
}{ 
\tilde V \cdot V
}\Big)^\pm c_\pm 
=
-\Big( \frac{ 
\dot \mu
}{ 
2\mu - \eta
}\Big)_\pm c_\pm,
\end{aligned}
\end{equation}
which may be solved by exponentiation, yielding the general solution
\begin{equation}\label{generals}
c_\pm (\lambda)= C(\eta^2/4+\lambda)^{-1/4}.
\end{equation}
Initializing at a fixed nonzero point\footnote{In the numerics of Section \ref{numerics}, we typically initialize at $\lambda=10$.}  , without loss of generality $c_\pm(1)=1$, we obtain formula \eqref{charkato}.


\begin{thebibliography}{10}

\bibitem{AGJ}
J.~Alexander, R.~Gardner, and C.~Jones.
\newblock A topological invariant arising in the stability analysis of
  travelling waves.
\newblock {\em J. Reine Angew. Math.}, 410:167--212, 1990.

\bibitem{AS}
J.~C. Alexander and R.~Sachs.
\newblock Linear instability of solitary waves of a {B}oussinesq-type equation:
  a computer assisted computation.
\newblock {\em Nonlinear World}, 2(4):471--507, 1995.

\bibitem{A}
J.~E. Anderson.
\newblock {\em Magnetohydrodynamic shock waves}.
\newblock MIT Press, 1963.

\bibitem{BHRZ}
B.~Barker, J.~Humpherys, K.~Rudd, and K.~Zumbrun.
\newblock Stability of viscous shocks in isentropic gas dynamics.
\newblock {\em Comm. Math. Phys.}, 281(1):231--249, 2008.

\bibitem{Ba}
G.~K. Batchelor.
\newblock {\em An introduction to fluid dynamics}.
\newblock Cambridge Mathematical Library. Cambridge University Press,
  Cambridge, paperback edition, 1999.

\bibitem{BT}
A.~Blokhin and Y.~Trakhinin.
\newblock Stability of strong discontinuities in fluids and {MHD}.
\newblock In {\em Handbook of mathematical fluid dynamics, {V}ol. {I}}, pages
  545--652. North-Holland, Amsterdam, 2002.

\bibitem{BDG}
T.~J. Bridges, G.~Derks, and G.~Gottwald.
\newblock Stability and instability of solitary waves of the fifth-order
  {K}d{V} equation: a numerical framework.
\newblock {\em Phys. D}, 172(1-4):190--216, 2002.

\bibitem{Br2}
L.~Q. Brin.
\newblock Numerical testing of the stability of viscous shock waves.
\newblock {\em Math. Comp.}, 70(235):1071--1088, 2001.

\bibitem{BrZ}
L.~Q. Brin and K.~Zumbrun.
\newblock Analytically varying eigenvectors and the stability of viscous shock
  waves.
\newblock {\em Mat. Contemp.}, 22:19--32, 2002.
\newblock Seventh Workshop on Partial Differential Equations, Part I (Rio de
  Janeiro, 2001).

\bibitem{C}
H.~Cabannes.
\newblock {\em {Theoretical magnetofluiddynamics}}.
\newblock Academic Press, New York, 1970.

\bibitem{CHNZ}
N.~Costanzino, J.~Humpherys, T.~Nguyen, and K.~Zumbrun.
\newblock {Spectral stability of noncharacteristic boundary layers of
  isentropic Navier--Stokes equations}.
\newblock Arch. Ration. Mech. Anal., to appear, 2008.

\bibitem{EF}
J.~W. Evans and J.~A. Feroe.
\newblock Traveling waves of infinitely many pulses in nerve equations.
\newblock {\em Math. Biosci.}, 37:23--50, 1977.

\bibitem{FT}
H.~Freist{\"u}hler and Y.~Trakhinin.
\newblock On the viscous and inviscid stability of magnetohydrodynamic shock
  waves;.
\newblock {\em Physica D: Nonlinear Phenomena}, 237(23):3030--3037, 2008.

\bibitem{GJ1}
R.~Gardner and C.~K. R.~T. Jones.
\newblock A stability index for steady state solutions of boundary value
  problems for parabolic systems.
\newblock {\em J. Differential Equations}, 91(2):181--203, 1991.

\bibitem{GJ2}
R.~A. Gardner and C.~K. R.~T. Jones.
\newblock Traveling waves of a perturbed diffusion equation arising in a phase
  field model.
\newblock {\em Indiana Univ. Math. J.}, 39(4):1197--1222, 1990.

\bibitem{GZ}
R.~A. Gardner and K.~Zumbrun.
\newblock The gap lemma and geometric criteria for instability of viscous shock
  profiles.
\newblock {\em Comm. Pure Appl. Math.}, 51(7):797--855, 1998.

\bibitem{Go1}
J.~Goodman.
\newblock Nonlinear asymptotic stability of viscous shock profiles for
  conservation laws.
\newblock {\em Arch. Rational Mech. Anal.}, 95(4):325--344, 1986.

\bibitem{Go2}
J.~Goodman.
\newblock Remarks on the stability of viscous shock waves.
\newblock In {\em Viscous profiles and numerical methods for shock waves
  (Raleigh, NC, 1990)}, pages 66--72. SIAM, Philadelphia, PA, 1991.

\bibitem{GMWZ6}
O.~Gues, G.~M{\'e}tivier, M.~Williams, and K.~Zumbrun.
\newblock Viscous boundary value problems for symmetric systems with variable
  multiplicities.
\newblock {\em J. Differential Equations}, 244(2):309--387, 2008.

\bibitem{GMWZ5}
O.~Gu{\`e}s, G.~M{\'e}tivier, M.~Williams, and K.~Zumbrun.
\newblock Existence and stability of noncharacteristic hyperbolic-parabolic
  boundary-layers.
\newblock Preprint., 2009.

\bibitem{HM}
N.~Hale and D.~R. Moore.
\newblock A sixth-order extension to the matlab package bvp4c of j. kierzenka
  and l. shampine.
\newblock Technical Report NA-08/04, Oxford University Computing Laboratory,
  May 2008.

\bibitem{H}
P.~Howard.
\newblock Nonlinear stability of degenerate shock profiles.
\newblock {\em Differential Integral Equations}, 20(5):515--560, 2007.

\bibitem{HR}
P.~Howard and M.~Raoofi.
\newblock Pointwise asymptotic behavior of perturbed viscous shock profiles.
\newblock {\em Adv. Differential Equations}, 11(9):1031--1080, 2006.

\bibitem{HRZ}
P.~Howard, M.~Raoofi, and K.~Zumbrun.
\newblock Sharp pointwise bounds for perturbed viscous shock waves.
\newblock {\em J. Hyperbolic Differ. Equ.}, 3(2):297--373, 2006.

\bibitem{HoZ}
P.~Howard and K.~Zumbrun.
\newblock The {E}vans function and stability criteria for degenerate viscous
  shock waves.
\newblock {\em Discrete Contin. Dyn. Syst.}, 10(4):837--855, 2004.

\bibitem{Hu3}
J.~Humpherys.
\newblock On the shock wave spectrum for isentropic gas dynamics with
  capillarity.
\newblock {\em J. Differential Equations}, 246(7):2938--2957, 2009.

\bibitem{HLZ}
J.~Humpherys, O.~Lafitte, and K.~Zumbrun.
\newblock Stability of viscous shock profiles in the high mach number limit.
\newblock Comm. Math. Phys, to appear, 2009.

\bibitem{HLyZ2}
J.~Humpherys, G.~Lyng, and K.~Zumbrun.
\newblock Multidimensional spectral stability of large-amplitude navier-stokes
  shocks.
\newblock In preparation., 2009.

\bibitem{HLyZ1}
J.~Humpherys, G.~Lyng, and K.~Zumbrun.
\newblock {Spectral stability of ideal-gas shock layers}.
\newblock Arch. Ration. Mech. Anal., to appear, 2009.

\bibitem{HSZ}
J.~Humpherys, B.~Sandstede, and K.~Zumbrun.
\newblock Efficient computation of analytic bases in {E}vans function analysis
  of large systems.
\newblock {\em Numer. Math.}, 103(4):631--642, 2006.

\bibitem{HuZ1}
J.~Humpherys and K.~Zumbrun.
\newblock Spectral stability of small-amplitude shock profiles for dissipative
  symmetric hyperbolic-parabolic systems.
\newblock {\em Z. Angew. Math. Phys.}, 53(1):20--34, 2002.

\bibitem{HuZ2}
J.~Humpherys and K.~Zumbrun.
\newblock An efficient shooting algorithm for {E}vans function calculations in
  large systems.
\newblock {\em Phys. D}, 220(2):116--126, 2006.

\bibitem{J}
A.~Jeffrey.
\newblock {\em Magnetohydrodynamics}.
\newblock University Mathematical Texts, No. 33. Oliver \& Boyd, Edinburgh,
  1966.

\bibitem{Kato}
T.~Kato.
\newblock {\em Perturbation theory for linear operators}.
\newblock Classics in Mathematics. Springer-Verlag, Berlin, 1995.
\newblock Reprint of the 1980 edition.

\bibitem{Kaw}
S.~Kawashima.
\newblock {\em Systems of a hyperbolic--parabolic composite type, with
  applications to the equations of magnetohydrodynamics}.
\newblock PhD thesis, Kyoto University, 1983.

\bibitem{KL}
J.~Kierzenka and L.~F. Shampine.
\newblock A {BVP} solver that controls residual and error.
\newblock {\em JNAIAM J. Numer. Anal. Ind. Appl. Math.}, 3(1-2):27--41, 2008.

\bibitem{LZ}
T.-P. Liu and K.~Zumbrun.
\newblock On nonlinear stability of general undercompressive viscous shock
  waves.
\newblock {\em Comm. Math. Phys.}, 174(2):319--345, 1995.

\bibitem{MaZ3}
C.~Mascia and K.~Zumbrun.
\newblock Pointwise {G}reen function bounds for shock profiles of systems with
  real viscosity.
\newblock {\em Arch. Ration. Mech. Anal.}, 169(3):177--263, 2003.

\bibitem{MaZ4}
C.~Mascia and K.~Zumbrun.
\newblock Stability of large-amplitude viscous shock profiles of
  hyperbolic-parabolic systems.
\newblock {\em Arch. Ration. Mech. Anal.}, 172(1):93--131, 2004.

\bibitem{MeZ2}
G.~M{\'e}tivier and K.~Zumbrun.
\newblock Hyperbolic boundary value problems for symmetric systems with
  variable multiplicities.
\newblock {\em J. Differential Equations}, 211(1):61--134, 2005.

\bibitem{MeZ1}
G.~M{\'e}tivier and K.~Zumbrun.
\newblock Large viscous boundary layers for noncharacteristic nonlinear
  hyperbolic problems.
\newblock {\em Mem. Amer. Math. Soc.}, 175(826):vi+107, 2005.

\bibitem{P}
R.~L. Pego.
\newblock Stable viscosities and shock profiles for systems of conservation
  laws.
\newblock {\em Trans. Amer. Math. Soc.}, 282(2):749--763, 1984.

\bibitem{PSW}
R.~L. Pego, P.~Smereka, and M.~I. Weinstein.
\newblock Oscillatory instability of traveling waves for a {K}d{V}-{B}urgers
  equation.
\newblock {\em Phys. D}, 67(1-3):45--65, 1993.

\bibitem{PZ}
R.~Plaza and K.~Zumbrun.
\newblock An {E}vans function approach to spectral stability of small-amplitude
  shock profiles.
\newblock {\em Discrete Contin. Dyn. Syst.}, 10(4):885--924, 2004.

\bibitem{Ra}
M.~Raoofi.
\newblock {$L\sp p$} asymptotic behavior of perturbed viscous shock profiles.
\newblock {\em J. Hyperbolic Differ. Equ.}, 2(3):595--644, 2005.

\bibitem{RZ}
M.~Raoofi and K.~Zumbrun.
\newblock Stability of undercompressive viscous shock profiles of
  hyperbolic-parabolic systems.
\newblock {\em J. Differential Equations}, 246(4):1539--1567, 2009.

\bibitem{SGT}
L.~F. Shampine, I.~Gladwell, and S.~Thompson.
\newblock {\em Solving {ODE}s with {MATLAB}}.
\newblock Cambridge University Press, Cambridge, 2003.

\bibitem{TZ}
B.~Texier and K.~Zumbrun.
\newblock Hopf bifurcation of viscous shock waves in compressible gas dynamics
  and {MHD}.
\newblock {\em Arch. Ration. Mech. Anal.}, 190(1):107--140, 2008.

\bibitem{T}
Y.~Trakhinin.
\newblock A complete 2{D} stability analysis of fast {MHD} shocks in an ideal
  gas.
\newblock {\em Comm. Math. Phys.}, 236(1):65--92, 2003.

\bibitem{Z1}
K.~Zumbrun.
\newblock Multidimensional stability of planar viscous shock waves.
\newblock In {\em Advances in the theory of shock waves}, volume~47 of {\em
  Progr. Nonlinear Differential Equations Appl.}, pages 307--516. Birkh\"auser
  Boston, Boston, MA, 2001.

\bibitem{Z2}
K.~Zumbrun.
\newblock Stability of large-amplitude shock waves of compressible
  {N}avier-{S}tokes equations.
\newblock In {\em Handbook of mathematical fluid dynamics. Vol. III}, pages
  311--533. North-Holland, Amsterdam, 2004.
\newblock With an appendix by Helge Kristian Jenssen and Gregory Lyng.

\bibitem{Z3}
K.~Zumbrun.
\newblock Planar stability criteria for viscous shock waves of systems with
  real viscosity.
\newblock In {\em Hyperbolic systems of balance laws}, volume 1911 of {\em
  Lecture Notes in Math.}, pages 229--326. Springer, Berlin, 2007.

\bibitem{Z4}
K.~Zumbrun.
\newblock {A local greedy algorithm and higher order extensions for global
  numerical continuation of analytically varying subspaces}.
\newblock {\em Arxiv preprint arXiv:0809.4725}, 2008.

\bibitem{Z5}
K.~Zumbrun.
\newblock Numerical error analysis for evans function computations: a numerical
  gap lemma, centered-coordinate methods, and the unreasonable effectiveness of
  continuous orthogonalization, 2009.

\bibitem{ZH}
K.~Zumbrun and P.~Howard.
\newblock Pointwise semigroup methods and stability of viscous shock waves.
\newblock {\em Indiana Univ. Math. J.}, 47(3):741--871, 1998.

\bibitem{ZS}
K.~Zumbrun and D.~Serre.
\newblock Viscous and inviscid stability of multidimensional planar shock
  fronts.
\newblock {\em Indiana Univ. Math. J.}, 48(3):937--992, 1999.

\end{thebibliography}

\end{document}